\documentclass[reqno,12pt]{amsart}

\usepackage[margin=1.75cm]{geometry}

\usepackage{xparse}       

\usepackage{multirow}
\usepackage{colortbl,hhline}

\usepackage{hyperref}

\usepackage{graphicx}                
\usepackage{tikz}

\usepackage{amsmath, amssymb}
\usepackage{amsthm}
\usepackage{amscd}

\usepackage{mathtools}

\usepackage{bbm}         
\usepackage{bm}          

\usepackage{mathrsfs}    

\usepackage{dsfont}

\usepackage{caption}
\usepackage{subcaption}

\usepackage{float}
\usepackage{tikz-cd}

\usepackage{hyperref}
\usepackage{cleveref}

  \newcommand{\myfont}{}
  \newcommand{\mythmstyle}{plain}
  \newcommand{\mydefstyle}{definition}

  \DeclareMathOperator*{\dcup}{\sqcup}
  \DeclareMathOperator*{\bigdcup}{\bigsqcup}

  \newcommand{\Gr}[1]{#1}    
  \newcommand{\wGr}[1]{\bm #1}  

  \newcommand{\mGr}[1]{\overline #1}  

  \newcommand{\DirGr}[2]{#1^+_{#2}}  

  \newcommand{\laplDir}[2]{\lapl{\DirGr{#1}{#2}}} 

\newcommand{\G}{\Gr G}    
\newcommand{\wG}{\wGr G}  
\newcommand{\mG}{\mGr G}  

\newcommand{\lapl}[2][{}]{\Delta_{{#2}}^{{#1}}} 

\newcommand{\RmodZ}{\R/2\pi\Z}

\newcommand{\quotient}[2][]{#2/{\mathord {\sim_{#1}}}}

\newcommand{\ol}{\overline}   
\newcommand{\ori}[1]{#1^{\circ}}  

\usepackage{nicefrac}

\newcommand{\quadtext}[1]{\quad\text{#1}\quad}
\newcommand{\qquadtext}[1]{\qquad\text{#1}\qquad}
\newcommand{\itemref}[1]{\noindent(\ref{#1})}
\newcommand{\myparagraph}[1]{\noindent\textbf{\myfont{#1}}}

\numberwithin{equation}{section}

\theoremstyle{\mythmstyle}       

\newtheorem{theorem}{Theorem}[section]
\newtheorem*{theorem*}{Theorem}
\newtheorem{proposition}[theorem]{Proposition}
\newtheorem{lemma}[theorem]{Lemma}
\newtheorem{corollary}[theorem]{Corollary}

\theoremstyle{\mydefstyle}        
\newtheorem{definition}[theorem]{Definition}

\newtheorem{example}[theorem]{Example}
\newtheorem{examples}[theorem]{Examples}

\newtheorem{remark}[theorem]{Remark}

\newtheorem*{remark*}{Remark}
\newtheorem*{commentolaf*}{Comment Olaf}

\newcommand{\Sec}[1]{Section~\ref{sec:#1}}

\newcommand{\App}[1]{Appendix~\ref{app:#1}}

\newcommand{\Subsec}[1]{Subsection~\ref{subsec:#1}}

\newcommand{\Fig}[1]{Figure~\ref{fig:#1}}

\newcommand{\Tab}[1]{Table~\ref{tab:#1}}
\newcommand{\Tabs}[2]{Tables~\ref{tab:#1} and~\ref{tab:#2}}
\newcommand{\TabS}[2]{Tables~\ref{tab:#1}--~\ref{tab:#2}}

\newcommand{\Thm}[1]{Theorem~\ref{thm:#1}}

\newcommand{\Ex}[1]{Example~\ref{ex:#1}}

\newcommand{\Lem}[1]{Lemma~\ref{lem:#1}}

\newcommand{\Cor}[1]{Corollary~\ref{cor:#1}}

\newcommand{\Prp}[1]{Proposition~\ref{prp:#1}}

\newcommand{\Prpenum}[2]{Proposition~\ref{prp:#1}~(\ref{#2})}

\newcommand{\Rem}[1]{Remark~\ref{rem:#1}}

\newcommand{\Def}[1]{Definition~\ref{def:#1}}

\newcommand{\abssqr}[2][{}]{\lvert{#2}\rvert^2_{#1}} 

\newcommand{\normsymb}{\|}

\newcommand{\normsqr}[2][{}]{\normsymb{#2}\normsymb^2_{#1}} 

\newcommand{\bigiprod}[3][{}]{\bigl\langle{#2},{#3}\bigr\rangle_{#1}}

\newcommand{\set}[2]{\{ \, #1 \, | \, #2 \, \} }      

\newcommand{\Bigset}[2]{\Bigl\{ \, #1 \, \Bigl|\Bigr. \, #2 \, \Bigr\} }

\makeatletter
\@ifpackageloaded{stix}{%
}{%
  \DeclareFontEncoding{LS2}{}{\noaccents@}
  \DeclareFontSubstitution{LS2}{stix}{m}{n}
  \DeclareSymbolFont{stix@largesymbols}{LS2}{stixex}{m}{n}
  \SetSymbolFont{stix@largesymbols}{bold}{LS2}{stixex}{b}{n}
  \DeclareMathDelimiter{\lMset}{\mathopen} {stix@largesymbols}{"E8}%
                                            {stix@largesymbols}{"0E}
  \DeclareMathDelimiter{\rMset}{\mathclose}{stix@largesymbols}{"E9}%
                                            {stix@largesymbols}{"0F}
}
\makeatother

\newcommand{\map}[3]{ #1 \colon #2 \longrightarrow #3}    

\newcommand{\bd}  {\partial}          

\newcommand{\restr}[1]{{\restriction}_{#1}} 

\newcommand{\card}[1]{\lvert#1\rvert}   

\newcommand{\specsymb} {\sigma} 

\newcommand{\spec}[2][{}]   {\specsymb_{\mathrm{#1}}(#2)}
\newcommand{\bigspec}[2][{}]   {\specsymb_{\mathrm{#1}}\bigl(#2\bigr)}

\newcommand{\eps}{\varepsilon} 
\renewcommand{\phi}{\varphi}   

\newcommand{\conj}[1]{\overline {#1}}

\newcommand{\R}{\mathbb{R}} 
\newcommand{\C}{\mathbb{C}} 
\newcommand{\N}{\mathbb{N}} 
\newcommand{\Z}{\mathbb{Z}} 

\newcommand{\1}{\mathbbm 1}                    

\newcommand{\e}{\mathrm e}  
\newcommand{\im}{\mathrm i} 

\newcommand{\wt}{\widetilde}           

\newcommand{\lsymb}    {\ell}          
\newcommand{\lpspace}[1][p]    {\lsymb_{#1}}     
\newcommand{\lsqrspace}    {\lpspace[2]}          
\newcommand{\lsqr}[2][{}]{\lsqrspace^{#1}({#2})}   
\newcommand{\Lsymb}    {\mathsf L}     
\newcommand{\Lpspace}[1][p]    {\Lsymb_{#1}}     
\newcommand{\Lsqrspace}    {\Lpspace[2]}     
\newcommand{\Lsqr}[2][{}]{\Lsqrspace^{#1}({#2})} 

\title{A geometric construction of isospectral magnetic graphs}%

\author{John Stewart Fabila-Carrasco$^*$} %
\address{School of Engineering, Institute for Digital Communications,
  University of Edinburgh, Edinburgh EH9 3FB, U.K. }
\email{John.Fabila@ed.ac.uk}

\author{Fernando Lled\'o} %
\address{Department of Mathematics, University Carlos III de Madrid,
  Avda. de la Universidad 30, 28911 Legan\'es (Madrid), Spain and
  Instituto de Ciencias Matem\'aticas (CSIC-UAM-UC3M-UCM), Madrid}
\email{flledo@math.uc3m.es}

\author{Olaf Post} %
\address{Fachbereich 4 -- Mathematik, Universit\"at Trier, 54286
  Trier, Germany} \email{olaf.post@uni-trier.de}

\date{\today}

\thanks{$^*$ Corresponding author\newline%
  JSFC was supported by the Leverhulme Trust via a Research
  Project Grant (RPG-2020-158). FLl was supported by the Severo Ochoa
  Programme for Centres of Excellence in R\&D (SEV-2015-0554) and from
  the Spanish National Research Council, through the \textit{Ayuda
    extraordinaria a Centros de Excelencia Severo Ochoa} (20205CEX001)
  and by the Madrid Government under the Agreement with UC3M in the
  line of \emph{Research Funds for Beatriz Galindo Fellowships}
  (C\&QIG-BG-CM-UC3M), and in the context of the V PRICIT}

\begin{document}

\begin{abstract}
  We present a geometrical construction of families of finite
  isospectral graphs labelled by different partitions of a natural
  number $r$ of given length $s$ (the number of summands).
  Isospectrality here refers to the discrete magnetic Laplacian with
  normalised weights (including standard weights).  The construction
  begins with an arbitrary finite graph $\bm G$ with normalised weight
  and magnetic potential as a building block from which we construct,
  in a first step, a family of so-called frame graphs
  $(\bm F_a)_{a \in \N}$.  A frame graph $\bm F_a$ is constructed
  contracting $a$ copies of $G$ along a subset of vertices $V_0$.  In
  a second step, for any partition $A=(a_1,\dots,a_s)$ of length $s$
  of a natural number $r$ (i.e., $r=a_1+\dots+a_s$) we construct a new
  graph $\bm F_A$ contracting now the frames
  $\bm F_{a_1},\dots,\bm F_{a_s}$ selected by $A$ along a proper
  subset of vertices $V_1\subset V_0$. All the graphs obtained by
  different $s$-partitions of $r\geq 4$ (for any choice of $V_0$ and
  $V_1$) are isospectral and non-isomorphic.

  In particular, we obtain increasing finite families of graphs which
  are isospectral for given $r$ and $s$ for different types of
  magnetic Laplacians including the standard Laplacian, the signless
  standard Laplacian, certain kinds of signed Laplacians and, also,
  for the (unbounded) Kirchhoff Laplacian of the underlying
  equilateral metric graph.  The spectrum of the isospectral graphs is
  determined by the spectrum of the Laplacian of the building block
  $G$ and the spectrum for the Laplacian with Dirichlet conditions on
  the set of vertices $V_0$ and $V_1$ with multiplicities determined
  by the numbers $r$ and $s$ of the partition.
\end{abstract}

\maketitle

%
\section{Introduction}
\label{sec:intro}
%


Since Weyl's seminal article in 1912 on the asymptotic distributions
of eigenvalues of the Laplacian on a compact Riemannian manifold
(\cite{weyl:12}, see also~\cite{ivrii:16}) and, in particular, since
Kac's celebrated article \emph{``Can one hear the shape of a
  drum?''}~\cite{kac:66}, the relation between the spectrum of the
Laplacian and the geometric properties of the underlying space has
been studied from many different perspectives. The problem was
proposed by Kac for the usual Laplacian on two-dimensional bounded
domains with Dirichlet condition at the boundary.  The first negative
answer to Kac's original question for \emph{two-}dimensional flat
domains 
was presented in~\cite{gww:92}
.  We refer to~\cite{gps:05,lu-rowlett:15,zelditch:14} for
some surveys on different aspects of isospectrality and inverse
spectral problems. An operator-theoretic point of view to the problem
of finding non-isometric isospectral domains can be found
in~\cite{arendt:02}.  Also, the
references~\cite{afm:22,lkbbss:22,park:22} show the importance and
timeliness of isospectrality in physics.

Recall that two linear operators with discrete spectrum are called
\emph{isospectral} (or \emph{cospectral}), if their spectra coincide,
i.e.\ they have the same eigenvalues with the same multiplicity. In
order to avoid trivial cases, it is important that the underlying
spaces are non-isomorphic (in the corresponding category).  Sunada
presented in~\cite{sunada:85} a systematic method for constructing
non-isometric manifolds where the corresponding Laplacians have the
same eigenvalues, answering negatively Kac's original question.
Sunada's method is based on certain quotients of a Riemannian normal
coverings and group theoretic properties of the covering groups.  An
application to discrete graphs and examples that do not arise from
Sunada's method can be found in~\cite{brooks:99}.  We refer to
\Sec{conclusions} for further remarks and for a comparison of our
method with Sunada's method.

An extension of Sunada's method is presented in~\cite{bpbs:09}
together with examples of isospectral quantum graphs (see
also~\cite{von-below:01,gutkin-smilansky:01,kurasov-muller:pre21}).
If one considers families of magnetic graphs with fixed underlying
discrete graph and the magnetic potential as a parameter, then certain
metric graphs that are isospectral for the entire family are
isomorphic, cf.~\cite{rueckriemen:13}.  An approach using the
Dirichlet-to-Neumann map of certain subgraphs is used
in~\cite[Section~4]{kurasov-muller:pre21} in order to construct
isospectral quantum graphs leading to similar examples as ours (see
also the last paragraph of \Sec{conclusions}). Additional references
including the relation between isospectral metric and discrete graphs
and the refinement of nodal or flip counts are, for
example,~\cite{gss:05,oren-band:12,juul-joyner:18}.
In~\cite{chernyshenko-pivovarchik:20} it is shown that equilateral
quantum graphs with only a few number of vertices are determined by
their spectrum; and in~\cite{pistol:23} isospectral equilateral
quantum graphs are determined with the help of a computer algebra
programme.  Another systematic approach for constructing isospectral
finite-dimensional operators using quotients is given
in~\cite{bbjl:pre17}.

In the context of discrete graphs isospectral constructions refer to
different operators: typically, to the adjacency matrix or the
combinatorial Laplacian (and its signless version), see
e.g.~\cite{godsil-mckay:82,cds:95,merris:97,tan:98,brooks:99,%
  halbeisen-hungerbuehler:99,%
  dam-haemers:03,haemers-spence:04,brouwer-haemers:12} or the
normalised Laplacian (see
e.g.~\cite{tan:98,butler:10,cavers:10,butler-grout:11,butler:15,%
  butler-heysse:16,dsg:17,hu-li:22}).  Since there is in general no
simple relation between these spectra for non-regular graphs,
isospectrality for discrete graphs strongly depends on the operator
chosen.  In this article we will focus on Laplacians with
\emph{normalised} weights (see \Subsec{mw.lapl} for a precise
definition) and there are several reasons for it:\footnote{Note that
  some authors use the name \emph{normalised} or \emph{geometric
    Laplacian} for what we call \emph{standard Laplacian} in this
  article, a special case of normalised weights (see
  e.g.~\cite{chung:97}); also the name \emph{weighted normalised
    Laplacian} is used in~\cite[Section~2.4]{osborne:13} for what we
  simply call \emph{normalised Laplacian}.}
\begin{itemize}
\item First, standard Laplacians have a direct connection to spectral
  geometry and stochastic processes (cf.~\cite{chung:97}) and their
  eigenvalues are sometimes called \emph{geometric}
  (cf.~\cite{banerjee-jost:08}).
\item Second, there is an explicit relation between the eigenvalues of
  the standard (discrete) Laplacian and the Kirchhoff Laplacian of the
  corresponding compact equilateral metric graph (see
  e.g.~\cite{von-below:01,lledo-post:08b} and references therein;
  in~\cite{lledo-post:08b} also the relation between Laplacians with
  Dirichlet conditions on subsets of vertices are described).
\item Third, the construction of isospectral graphs for normalised
  weights has been studied less than for other operators like the
  combinatorial Laplacian or the adjacency matrix (cf.~\cite{butler-grout:11}).
\item Finally, it seems that there is a relatively small number of isospectral
 graphs with normalised weights compared to other cases.
 Table~\ref{tab:1} (taken from~\cite{butler-grout:11} and based on
 results for the combinatorial operators by~\cite{haemers-spence:04})
 shows the number of simple graphs of order at most $9$ with an
 isospectral mate for different types of Laplacians: for the
 combinatorial Laplacian, the signless (combinatorial) Laplacian and
 for the \emph{standard} (sometimes also called \emph{normalised} or
 \emph{geometric}) Laplacian.
\end{itemize}

Specific construction techniques of isospectral discrete graphs are
presented in~\cite{butler-grout:11} where the gluing of certain blocks
of graphs including some bipartite blocks resulted in pairs of
isospectral graphs for the standard
Laplacian. In~\cite{butler-heysse:16} the authors present a technique
based on toggling of three basic blocks to produce isospectral graphs
with different number of edges, see also the construction technique
presented in~\cite{butler:15} using twin graphs and a scaling
technique on general vertex and edge weights. Cavers shows
in~\cite{cavers:10} that a variation of the Godsil-McKay switching
also preserves the spectrum of the standard Laplacian.
In~\cite{dsg:17} the authors give necessary conditions for a graph to
be isospectral with a subgraph with one edge deleted; moreover, they
give classes of graphs where this cannot happen.  In~\cite{hu-li:22}
the authors construct isospectral graphs for the normalised Laplacian
by partitioning the vertices into several blocks and allowing edges
only between consecutive blocks in certain proportions.
\begin{table}
  \centering
  \begin{tabular}{c|c|c|c|c}
    \rule[-1ex]{0pt}{2.5ex}  No.\ vertices &  No.\ graphs & Combinatorial
       & Signless Comb.\ & Standard  \\
    \hline
    \rule[-1ex]{0pt}{2.5ex}  1&  1& 0 & 0 &0    \\
    \rule[-1ex]{0pt}{2.5ex}  2&  2&  0& 0 &0    \\
    \rule[-1ex]{0pt}{2.5ex}  3&  4&  0&  0& 0   \\
    \rule[-1ex]{0pt}{2.5ex}  4&  11&  0& 2 &2    \\
    \rule[-1ex]{0pt}{2.5ex}  5&  34& 0 & 4 & 4   \\
    \rule[-1ex]{0pt}{2.5ex}  6&  156& 4& 16&  14  \\
    \rule[-1ex]{0pt}{2.5ex}  7&  1044& 130 & 102&  52 \\
    \rule[-1ex]{0pt}{2.5ex}  8&  12346& 1767 & 1201 & 201  \\
    \rule[-1ex]{0pt}{2.5ex}  9&  274668& 42595  & 19001 & 1092  \\
  \end{tabular}
  \caption{Number of simple graphs (not necessarily connected) having
    a non-isomorphic graph with isospectral Laplacian for different types of Laplacians.
    Note that the table contains also non-connected graphs. \label{tab:1}}
\end{table}
Isospectrality of discrete graph Laplacians may be extended in several
directions. Lim studies the Hodge Laplacian in~\cite[Example~4.1 and
Section~5.3]{lim:20}, a higher order generalisation of the graph
Laplacian mentioned before, and considers isospectral graphs for this
operator.

In this article we propose to extend the question of isospectrality to
the discrete magnetic Laplacian with normalised weights on discrete
graphs. Recall that the magnetic field is described in this context in
terms of a magnetic potential function $\alpha$ on the oriented edges
(arcs) with values in $\RmodZ$ and we call these graphs simply
\emph{magnetic graphs} (see \cite{shubin:94,sunada:94c} in the case of magnetic Laplacians
on the lattice graph $\Z^2$).
Isospectrality of magnetic graphs refers to this operator. One
important advantage of the use of the magnetic potential is that it
interpolates between different types of Laplacians including the
usual normalised Laplacian (if $\alpha=0$), the signless Laplacian (if
$\alpha_e=\pi$, $e\in A$) or the signed Laplacian (if
$\alpha_e\in\{0,\pi\}$). Therefore, the graphs constructed here will
be isospectral for these Laplacians and, also, for the Kirchhoff
Laplacian of the corresponding equilateral graphs, see
\Subsec{met.graphs}.  We refer to~\cite{fclp:18,fclp:22a} for
additional motivation and the use of the magnetic potential in
periodic graphs and for the generation of spectral gaps.  Moreover, we
refer to~\cite{fclp:22b} for the use of magnetic potentials to give
spectral obstructions for the graph having Hamiltonian cycles or
perfect matchings.  In~\cite{fclp:pre22a} we present a perturbative
method using only spectral information (and no eigenfunctions) to
construct isospectral graphs for the normalised Laplacian (for
simplicity without magnetic field).

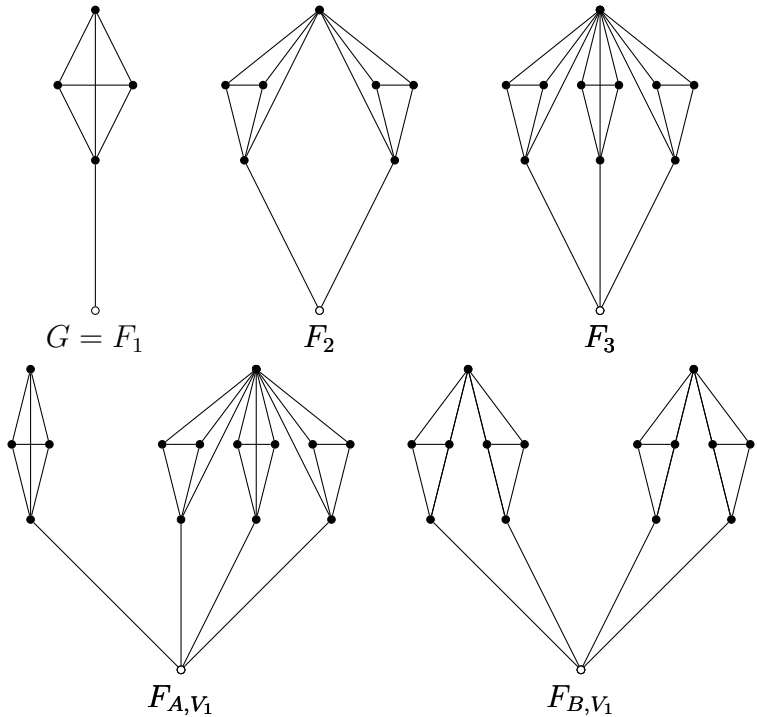
\begin{figure}[h]
		\begin{tikzpicture}[auto,vertex/.style={circle,draw=black!100,%
				fill=black!100, 
				inner sep=0pt,minimum size=1mm},scale=1]
			\foreach \x in {0}
			{\node(D) at (0,-1.00) [vertex,fill=white,label=below:${\G=\Gr F_1}$] {};
			 \node(C\x) at (\x,1) [vertex,fill=black,] {};
				\node (B-1\x) at ({\x-.5},2) [vertex,label=left:] {};
				\node (B1\x) at ({\x+.5},2) [vertex,label=right:] {};
				\node (A) at (0,3) [vertex,fill=black] {};

				\path [-](D) edge node[right] {} (C\x);
				\path [-](C\x) edge node[right] {} (B-1\x);
				\path [-](C\x) edge node[right] {} (B1\x);
				\path [-](B-1\x) edge node[right] {} (B1\x);
				\path [-](B-1\x) edge node[right] {} (A);
				\path [-](B1\x) edge node[right] {} (A);
				\path [-](C\x) edge node[right] {} (A);
			}
		\end{tikzpicture}
		\quad 
		\begin{tikzpicture}[auto,vertex/.style={circle,draw=black!100,%
				fill=black!100, 
				inner sep=0pt,minimum size=1mm},scale=1]
			\foreach \x in {1,-1}
			{\node(D) at (0,-1.00) [vertex,fill=white,label=below:$\Gr F_2$] {};
				\node(C\x) at (\x,1) [vertex,fill=black,] {};
				\node (B-1\x) at ({\x-.25},2) [vertex,label=left:] {};
				\node (B1\x) at ({\x+.25},2) [vertex,label=right:] {};
				\node (A) at (0,3) %
                                [vertex,fill=black] {};
				\path [-](D) edge node[right] {} (C\x);
				\path [-](C\x) edge node[right] {} (B-1\x);
				\path [-](C\x) edge node[right] {} (B1\x);
				\path [-](B-1\x) edge node[right] {} (B1\x);
				\path [-](B-1\x) edge node[right] {} (A);
				\path [-](B1\x) edge node[right] {} (A);
				\path [-](C\x) edge node[right] {} (A);
			}
		\end{tikzpicture}
		\quad 
	\begin{tikzpicture}[auto,vertex/.style={circle,draw=black!100,%
			fill=black!100, 
			inner sep=0pt,minimum size=1mm},scale=1]
		\foreach \x in {1,0,-1}
		{\node(D) at (0,-1.00) [vertex,fill=white,label=below:$\Gr F_3$] {};
			\node(C\x) at (\x,1) [vertex,fill=black,] {};
			\node (B-1\x) at ({\x-.25},2) [vertex,label=left:] {};
			\node (B1\x) at ({\x+.25},2) [vertex,label=right:] {};
			\node (A) at (0,3) [vertex,fill=black] {};

			\path [-](D) edge node[right] {} (C\x);
			\path [-](C\x) edge node[right] {} (B-1\x);
			\path [-](C\x) edge node[right] {} (B1\x);
			\path [-](B-1\x) edge node[right] {} (B1\x);
			\path [-](B-1\x) edge node[right] {} (A);
			\path [-](B1\x) edge node[right] {} (A);
			\path [-](C\x) edge node[right] {} (A);

		}
              \end{tikzpicture}
              \quad

  \begin{tikzpicture}[auto,vertex/.style={circle,draw=black!100,%
        fill=black!100, 
        inner sep=0pt,minimum size=1mm},scale=1]
      \foreach \x in {-2}
      {\node(D) at (0,-1.00) [vertex,fill=white] {};
        \node(C\x) at (\x,1) [vertex,fill=black,] {};
        \node (B-1\x) at ({\x-.25},2) [vertex,label=left:] {};
        \node (B1\x) at ({\x+.25},2) [vertex,label=right:] {};
        \node (A) at (-2,3) [vertex,fill=black] {};
        \path [-](D) edge node[right] {} (C\x);
        \path [-](C\x) edge node[right] {} (B-1\x);
        \path [-](C\x) edge node[right] {} (B1\x);
        \path [-](B-1\x) edge node[right] {} (B1\x);
        \path [-](B-1\x) edge node[right] {} (A);
        \path [-](B1\x) edge node[right] {} (A);
        \path [-](C\x) edge node[right] {} (A);
      }
      \foreach \x in {0,1,2}
      {\node(D) at (0,-1.00) [vertex,fill=white,label=below:$\Gr F_{A,V_1}$] {};
        \node(C\x) at (\x,1) [vertex,fill=black] {};
        \node (B-1\x) at ({\x-.25},2) [vertex,label=left:] {};
        \node (B1\x) at ({\x+.25},2) [vertex,label=right:] {};
        \node (A) at (1,3) [vertex,fill=black] {};

        \path [-](D) edge node[right] {} (C\x);
        \path [-](C\x) edge node[right] {} (B-1\x);
        \path [-](C\x) edge node[right] {} (B1\x);
        \path [-](B-1\x) edge node[right] {} (B1\x);
        \path [-](B-1\x) edge node[right] {} (A);
        \path [-](B1\x) edge node[right] {} (A);
        \path [-](C\x) edge node[right] {} (A);
      }
    \end{tikzpicture}
    \begin{tikzpicture}[auto,vertex/.style={circle,draw=black!100,%
        fill=black!100, 
        inner sep=0pt,minimum size=1mm},scale=1]
      \foreach \x in {-2,-1}
      {\node(D) at (0,-1.00) [vertex,fill=white] {};
        \node(C\x) at (\x,1) [vertex,fill=black,] {};
        \node (B-1\x) at ({\x-.25},2) [vertex,label=left:] {};
        \node (B1\x) at ({\x+.25},2) [vertex,label=right:] {};
        \node (A) at (-1.5,3) [vertex,fill=black] {};

        \path [-](D) edge node[right] {} (C\x);
        \path [-](C\x) edge node[right] {} (B-1\x);
        \path [-](C\x) edge node[right] {} (B1\x);
        \path [-](B-1\x) edge node[right] {} (B1\x);
        \path [-](B-1\x) edge node[right] {} (A);
        \path [-](B1\x) edge node[right] {} (A);
        \path [-](C\x) edge node[right] {} (A);
      }
      \foreach \x in {1,2}
      {\node(D) at (0,-1.00) [vertex,fill=white,label=below:${\Gr F}_{B,V_1}$] {};
        \node(C\x) at (\x,1) [vertex,fill=black,] {};
        \node (B-1\x) at ({\x-.25},2) [vertex,label=left:] {};
        \node (B1\x) at ({\x+.25},2) [vertex,label=right:] {};
        \node (A) at (1.5,3) [vertex,fill=black] {};

        \path [-](D) edge node[right] {} (C\x);
        \path [-](C\x) edge node[right] {} (B-1\x);
        \path [-](C\x) edge node[right] {} (B1\x);
        \path [-](B-1\x) edge node[right] {} (B1\x);
        \path [-](B-1\x) edge node[right] {} (A);
        \path [-](B1\x) edge node[right] {} (A);
        \path [-](C\x) edge node[right] {} (A);
      }
    \end{tikzpicture}
    \caption{The construction of pairs of isospectral graphs:
      \newline %
      \emph{Top row from left to right:} the building block $\wG$ (a
      so-called \emph{kite graph}), the frame members $\wGr F_1=\wG$, $\wGr F_2$ and
      $\wGr F_3$.  \newline%
      \emph{Bottom row:} The two graphs $\wGr F_{A,V_1}$
      and $\wGr F_{B,V_1}$ are isospectral but not isomorphic.  Note
      that this is the smallest possible choice of non-trivial
      partitions, namely the two different $2$-partitions of
      $A=\lMset 1,3\rMset$ (left) and $B=\lMset 2,2\rMset$ (right)
      ($4=1+3=2+2$).  Note that the kite graph can also carry a
      magnetic potential; here given by three parameters, see
      \Ex{kite}.
    \label{fig:kite.isospec}
  }
\end{figure}
We illustrate first our general construction procedure through a
motivating family of examples presented in \Sec{motivating-ex}.  In
\Fig{kite.isospec} we show, and briefly explain, a concrete example of
two isospectral graphs constructed from a kite graph as a building
block $\wG$ and frame graphs $\wGr F_1=\wG$, $\wGr F_2$ and
$\wGr F_3$. In this example, the frame graphs are then glued together
along certain vertices and according to the two
partitions\footnote{For the multiset notation we refer to
  \App{appendix}.} $A=\lMset 1,3\rMset$ and $B=\lMset 2,2\rMset$ of
the natural number $r=4$ with length $s=2$, i.e., gluing
$\wGr F_1=\wG$ with $\wGr F_3$ and $\wGr F_2$ with itself to produce
the isospectral mates $\wGr F_A$ and $\wGr F_B$. Making contact with
Kac's original question, our construction shows that for any $r\in\N$
one can hear the length of the partition $s$ but not the particular
decomposition summing up to $r$. In addition, our graphs will be
isospectral for the the standard Laplacian, the signless standard
Laplacian, the signature Laplacian corresponding to arbitrary
signatures of the building block and, more generally, arbitrary
magnetic potentials on the building block.  Our construction also
determines the spectrum of the isospectral graphs in terms of the
spectrum of $\wG$ and the spectra of the Laplacian on $\wG$ with
Dirichlet conditions of the selected sets of vertices $V_0$ and $V_1$.
Nevertheless, we emphasise the geometrical construction leading to the
families of mutually isospectral graphs rather than the determination
of the corresponding spectra. Our isospectral graphs will have the
same number of edges.

Let us finally mention that all our isospectral, but not isomorphic
discrete graphs with standard weight (and no magnetic field) directly
lead to equilateral metric graphs which are isospectral, but not
isomorphic (see \Cor{main}).

\subsection*{Structure of the article}
The following section introduces the basic notions and results on
multidigraphs and discrete magnetic Laplacians needed in this
article. We also introduce here the two basic graph operations needed
in our construction: vertex \emph{contraction} and vertex
\emph{virtualisation} leading to the corresponding Dirichlet
Laplacian. In \Sec{motivating-ex} we present our construction first in
a family of examples with a concrete building block graph and
illustrate some of the key ideas of the general proof in these
prototypes. The following section focuses on the notion of a frame
obtained from a generic building block (magnetic graph) with some
distinguished subset of vertices $V_0$.  We also calculate the
spectrum of a frame member and give a group theoretical explanation of
the different multiplicities that appear in the spectrum of the
isospectral graphs.  In \Sec{frames-and-partitions} the general
construction is presented and the main results on isospectrality
(\Thm{main} and \Cor{main}) are proved.  In the next section we
present many families of isospectral graphs exploiting systematically
the combinatorial freedom in our construction procedure.  In
\App{appendix} we recall basic facts on partitions and multisets that
will be used throughout the article.

\subsection*{Acknowledgements}
It is a pleasure to thank an anonymous referee for a careful reading
of the manuscript and many useful suggestions and remarks on the first
version of the manuscript.


%
\section{Magnetic discrete graphs and their Laplacians}
\label{sec:graphs}
%

In this section, we introduce briefly the discrete structures and
operations needed for the construction for families of isospectral
graphs. We will consider discrete locally finite graphs having
normalised weights on vertices and edges. We will also consider a
magnetic potential on the graph in terms of an $\RmodZ$-valued
function $\alpha$ on the edges. Finally we will specify the discrete normalised
magnetic Laplacian as a second order discrete analogue of the usual
Laplace operator on a Riemannian manifold with magnetic
potential.

\subsection{Discrete graphs}
\label{subsec:disc.graphs}
In this subsection, we fix the notation for discrete graphs following
the notation of Sunada (cf.~\cite[Chapter~3]{sunada:13}), see also our
previous articles~\cite{fclp:18,fclp:22a}, where certain aspects are
explained in more detail.

A \emph{(discrete multidi-)graph} $\G=(V,E,\partial)$ is given by two
disjoint and (in this article) finite sets $V$ and $E$ (the set of
\emph{vertices} and of \emph{(directed) edges} or \emph{arcs})
together with an \emph{incidence map} $\map \partial E{V \times V}$,
$\bd e=(\bd_-e,\bd_+e)$, where $\bd_-e$ is the \emph{initial}, and
$\bd_+e$ the \emph{terminal} vertex of $e \in E$.  The \emph{order} of
$\G$ is $\card \G:=\card {V(\G)}$.

An \emph{inversion map} $\map{\ol \cdot} E E$ is a map fulfilling
$\ol{\ol e} =e$, $\ol e \ne e$, $\bd_+ (\ol e)=\bd_-e$ and
$\bd_-(\ol e)=\bd_+e$ for all $e \in E$.  The inversion map
$\ol \cdot$ introduces a $\Z_2$-action on $E$, the orbits
$\{e,\ol e\}$ are called \emph{undirected} edges; in the figures, we
only plot the undirected edges; each comes with both directions $e$
and $\ol e$.

A \emph{loop} is an edge $e$ such that $\bd_+e=\bd_-e$ (note that a
loop also has two different directed edges $e$ and $\ol e$).  Two
edges $e_1$ and $e_2$ are called \emph{parallel} if $e_1\neq e_2$ with
$\bd_+e_1=\bd_+e_2$ and $\bd_- e_1 =\bd_- e_2$.  A graph is
\emph{simple} if it has no loops or parallel edges.

For a vertex $v\in V$, the set of directed edges with
origin $v$ is denoted as $E_v(\G)$ or simply $E_v$, i.e.
\begin{equation*}
  E_v(\G):=\set{e\in E} {\bd_- e=v}.
\end{equation*}
The \emph{degree} of a vertex $v$ is the cardinality of $E_v$, denoted
by $\deg v:= \card{E_v}$.  Observe that loops are counted twice in
each set $E_v$ (as $e \ne \ol e$).  We exclude \emph{isolated}
vertices and vertices with infinite degree, i.e., we assume that
$0<\deg_G v<\infty$ for all $v \in V$.  A vertex of degree $1$ is
called \emph{pendant}.  If the dependence of the graph $\G$ is
important, we write $V=V(\G)$, $E=E(\G)$ or $\deg_\G$ etc.  Let $\G$
be a graph with $n$ vertices labelled $v_1, v_2, \dots, v_n$.  The
multiset
\begin{equation}
  \label{eq:deg.set}
  \deg \G := \lMset \deg v_1, \deg v_2,\dots, \deg v_n \rMset
\end{equation}
is called the \emph{degree multiset} of $\G$ (see \App{multisets}
for more details on multisets).  A graph $\G$ is \emph{connected} if
for any two vertices $u,v$ there exists a path from $u$ to $v$, i.e.,
there is a sequence of distinct vertices $(u,v_1,\dots,v_{n-1},v)$ and
of directed edges $(e_1,e_2,\dots,e_n)$ with
$v_{i}=\bd_+ e_i=\bd_-e_{i+1}$ for $1\leq i\leq n-1$ and $\bd_-e_1=u$
and $\bd_+ e_n=v$.

A \emph{graph homomorphism} $\map \phi \G {\G'}$ between the graphs
$\G=(V(\G),E(\G))$ and $\G'=(V(\G'), E(\G'))$ is given by two maps
$\map \phi {V(\G)} {V(\G')}$ and $\map \phi {E(\G)} {E(\G')}$
respecting the incidence and inversion maps, i.e., they fulfil
$\bd_\pm(\phi(e))=\phi(\bd_\pm e)$ and
$\ol {\phi(e)}=\phi(\ol e)$ for all $e\in E(\G)$.  The map $\phi$ is a
\emph{graph isomorphism} if $\phi$ is a bijection on the vertex and
edge sets; in this case, $\G$ and $\G'$ are called \emph{isomorphic
  (discrete) graphs}.  Note that the degree
multiset~\eqref{eq:deg.set} is a graph invariant, i.e., if the degree
sequence or the degree multiset of two graphs differ, then the graphs
cannot be isomorphic.

\subsection{Graphs with normalised weights and magnetic potentials}
\label{subsec:w-graphs}
We briefly define here discrete graphs with normalised weights.  As
before, details can be found in~\cite{fclp:18,fclp:22a} and references
therein.  An \emph{(edge) weight} is a function
$\map w E {(0,\infty)}$ such that $w_{\ol e}=w_e$ for all $e\in E$,
the associated \emph{weighted degree} is defined by
\begin{equation}
  \label{eq:weighted.deg}
  \deg_\G^w v=\sum_{e \in E_v} w_e.
\end{equation}
The edge weight $w$ and vertex weight $\deg_\G^w$ are called
\emph{normalised} in the terminology
of~\cite[Section~2.2]{fclp:18,fclp:22a}.  A particular case is $w_e=1$
for all $e \in E$ and hence $\deg^w_\G v=\deg_\G v$ is the usual
degree of $v$; this weight is called \emph{standard}.\footnote{Some
  authors use the name \emph{normalised} for the corresponding
  Laplacian (see below) for we call \emph{standard} here (see,
  e.g.~\cite{chung:97,butler-grout:11}).}  We extend $w$ and
$\deg^w_\G$ naturally as discrete measures on $E(\G)$ and $V(\G)$.

Next, we define magnetic potentials on discrete graphs
(see~\cite{fclp:18,fclp:22a} and references therein for details).  A
\emph{magnetic potential} is a function $\map \alpha E \RmodZ$ such
that $\alpha_{\ol e}=-\alpha_e$ for all $e \in E$.  We call
$\wG=(\G,\alpha,w)$ a \emph{magnetic graph} (with normalised weights).
For the standard weight and no magnetic potential (i.e., $w_e=1$ and
$\alpha_e=0$ for all $e \in E$) we simply write $\G$ for the
corresponding magnetic graph $\wGr G=(\G,0,1)$.

Two magnetic graphs $\wG=(\G,\alpha,w)$ and $\wG'=(\G',\alpha',w')$
are said to be \emph{isomorphic} (as magnetic graphs) if there exists
a graph isomorphism $\map \phi \G{\G'}$ such that $w_e=w'_{\phi(e)}$
and $\alpha_e=\alpha'_{\phi(e)}$ for all $e\in E(\G)$.

\subsection{Contraction of vertices}
\label{subsec:vx.contr}

An important operation we need in this article is the
\emph{contraction of vertices} (also called \emph{gluing} or
\emph{merging} of vertices in the literature).  We introduce the
shrinking number $s$ that quantifies the reduction of the number of
vertices in the quotient graph.
\begin{definition}[Contracting vertices and shrinking number]
  \label{def:contraction}
  Let $\G=(V,E)$ be a graph and $\sim$ an equivalence relation on the
  vertex set $V$.  The graph $\quotient \G =(\quotient V,E)$ obtained
  from $\G$ by \emph{contracting the vertices} (with respect to
  $\sim$) is given by $\quotient V$ and incidence function
  $\bd(e)=([\bd_-(e)],[\bd_+(e)])$.  The \emph{shrinking number} $s$
  of the equivalence relation $\sim$ on $V(\G)$ is defined by
  $s:=\card{\G}-\card{\quotient \G}=\card V-\card{\quotient V}$.
\end{definition}

Note that if we contract two \emph{adjacent} vertices $v_1$ and $v_2$,
all edges joining $v_1$ and $v_2$ become \emph{loops} in
$\G/\mathord\sim$ (in contrast to some conventions in combinatorics
where only simple graphs are allowed).

Given a discrete graph $\G=(V,E)$ with (edge) weight $w$, we choose
$w$ also as weight on the quotient graph
$\quotient \G=(\quotient V,E)$ (denoted by the same symbol $w$).  For
the quotient, we have
\begin{equation*}
  E_{[v]}(\quotient \G)
  = \bigdcup_{v' \in [v]} E_{v'}(\G),
\end{equation*}
where $\dcup$ means that the union is \emph{disjoint}.  Hence the
weighted degree is
\begin{equation}
  \label{eq:weight.on.quot}
  \deg_{\quotient G}^w [v]
  = w(E_{[v]}(\quotient \G))
  = \sum_{e \in E_{[v]}(\quotient \G)} w_e
  =\sum_{v' \in [v]} \sum_{e \in E_v} w_e
  =\sum_{v' \in [v]} \deg_G^w v'.
\end{equation}
It follows that the quotient map $\map \pi \G {\quotient \G}$ given by
$v \mapsto [v]$ is a measure preserving map, since the edge weights
are the same (the edges are actually not changed) and since
\begin{align}
  \nonumber
  (\pi_* \deg_\G^w)(\wt V)
  = \deg_\G^w(\pi^{-1}(\wt V))
  = \sum_{v' \in \pi^{-1}(\wt V)} \deg_\G^w v'
  &=\sum_{[v] \in \wt V} \sum_{v' \in [v]} \deg_G^w(v')\\
  \label{eq:quot.meas.pres}
  &=\sum_{[v] \in \wt V} \deg_{\quotient \G}^w([v])
  =\deg_{\quotient \G}^w(\wt V)
\end{align}
for any $\wt V \subset V(\quotient \G)$.
Similarly, a magnetic potential $\alpha$ on $\G$ will be lifted to $\quotient \G$.

\subsection{Normalised magnetic Laplacians}
\label{subsec:mw.lapl}

For a graph $(\G,w)$, we associate the following natural Hilbert space
\begin{align*}
  \lsqr{V,\deg^w}
  &:=\Bigset{\map f V \C}
    {\normsqr[\lsqr{V,\deg^w}] f = \sum_{v \in V} \abssqr{f(v)} \deg^w v},
\end{align*}
as our graphs are finite, $\lsqr{V,\deg^w}$ is a $\card \G$-dimensional space.

\begin{definition}[discrete (normalised) magnetic Laplacian and its spectrum]
  \label{def:DML}
  Let $\wG=(\G,\alpha,w)$ be a magnetic graph with normalised
  weights. The \emph{(normalised) magnetic Laplacian} is the operator
  $\map{\lapl \wG}{\lsqr{V,\deg^w}}{\lsqr{V,\deg^w}}$ defined by
  \begin{equation}
    \label{eq:DML}
    \bigl(\lapl \wG f \bigr)(v)
    = \frac1{\deg^w v} \sum_{e\in E_v(\G)}
       w_e \bigl(f(v)-\e^{\im \alpha_e}f(\bd_+e)\bigr)
    = f(v)-\frac1{\deg^w v } \sum_{e\in E_v(\G)}
       w_e \e^{\im \alpha_e}f(\bd_+e)\;, v\in V,
  \end{equation}
  where the sum in Eq.~\eqref{eq:DML} is taken over all edges starting
  at $v$ and loops are counted twice; $\deg^w v=\sum_{e \in E_v}w_e$
  is the weighted degree (see~\eqref{eq:weighted.deg}).  If $\wG$ has order
  $n$, then we write the spectrum of the corresponding Laplacian
  $\lapl \wG$ as the multiset (see \App{multisets} below)
  \begin{equation*}
    \spec \wG
    = \bigl\lMset \lambda_1(\wG), \lambda_2(\wG), \dots ,\lambda_n(\wG)
    \bigr\rMset  \;,
  \end{equation*}
  where the eigenvalues $\lambda_1(\wG), \dots, \lambda_n(\wG)$ are
  written in ascending order and repeated according to their
  multiplicities.
\end{definition}
For a definition using a discrete (twisted) exterior derivative, we
refer to~\cite[Definition~2.3]{fclp:18} or~\cite[Definition~3.1]{fclp:22a}.
\begin{remark}[matrix representation of the Laplacian]
  \label{rem:matrix.rep}
  Note that the matrix representation of $\lapl \wG$ with respect to
  the orthonormal basis $(\delta_v)_{v \in V}$ of $\lsqr {V,\deg^w}$
  ($\delta_v(v')=1/\sqrt {\deg^w(v)}$ if $v=v'$ and $\delta_v(v')=0$
  otherwise) is
  \begin{equation}
    \label{eq:lapl.matrix}
    (\lapl \wG)_{v,v'}
    \begin{cases}
      \displaystyle
      1-\sum_{e \in E, \bd_-e=\bd_+e=v} w_e\e^{\im \alpha_e}/\deg^w(v), & v=v',\\[1ex]
      \displaystyle
      - \sum_{e \in E, \bd_-e=v, \bd_+e=v'} w_e \e^{\im \alpha_e}/\sqrt{\deg^w(v)\deg^w(v')},
      & v \ne v'.
    \end{cases}
  \end{equation}
  If the graph is simple and has no magnetic potential, we have
  \begin{equation*}
    (\lapl \wG)_{v,v'}
    \begin{cases}
      1,  & v=v',\\
      - 1/\sqrt{\deg^w(v) \deg^w(v')}, & \text{$v$ adjacent with $v'$,}\\
      0, & \text{otherwise.}
    \end{cases}
  \end{equation*}
  For a graph with loops and standard weight $w_e=1$, we have
  $(\lapl \wG)_{v,v}=1-2\ell(v)/\deg(v)$ if there are $2\ell(v)$ loops
  at a vertex $v$ (note that we distinguish between $e$ and $\bar e$),
  so a loop counts twice.
\end{remark}
As an example the (non-simple) magnetic graph $\wG^\theta$ in
\Fig{motivating-ex} with three vertices has the matrix representation
\begin{equation}
  \label{eq:matrix.ex}
  \lapl \wG \cong
  \begin{pmatrix}
    1 & -\frac{1+\e^{\im \theta}}{\sqrt 6} & 0\\
    -\frac{1+\e^{-\im \theta}}{\sqrt 6} & 1 & -\frac1{\sqrt 3}\\
    0 & -\frac1{\sqrt 3} & 1
  \end{pmatrix}.
\end{equation}

\begin{definition}[isospectral magnetic graphs]
  \label{def:isospec}
  Let $\wG$ and $\wt \wG$ be two finite magnetic graphs.  We say that
  $\wG$ and $\wt \wG$ are \emph{isospectral} (or \emph{cospectral)} if
  their spectra agree as multiset, i.e., if all eigenvalues agree
  including their multiplicities.
\end{definition}

\begin{remark}[multiple edges versus weighted graphs without multiple edges]
  \label{rem:mult.ed.weighed.gr}
  A weighted graph $\G$ with multiple edges can be turned into a graph
  $\wt \G$ without multiple edges.  On the other hand, a weighted
  graph with integer (edge) weights can be turned into a graph with
  multiple edges and standard weights.  We discuss for simplicity only
  the case of graphs without loops and magnetic potential.

  It turns out that both graphs (with or without multiple edges) have
  the same Laplacian.  This leads us to the following definition: We
  say that two weighted graphs $\wG=(\G,0,w)$ and
  $\wt \wG=(\wt \G,0,\wt w)$ on the same vertex set
  $V(\wG)=V(\wt \wG)$ are \emph{isolaplacian} if
  $\deg^w_{\wG}=\deg^{\wt w}_{\wt \wG}$ and
  $\lapl \wG f= \lapl {\wt \wG} f$ for all
  $f \in \lsqr{V(\wG),\deg^w}$, or, equivalently, if the matrix
  representation of their Laplacians is the same.

  Clearly, isomorphic magnetic graphs are isolaplacian, and
  isolaplacian magnetic graphs are isospectral (as the corresponding
  matrices are the same), but not vice versa.
  \begin{enumerate}
  \item
    \label{mult.ed.weighed.gr.a}
    Assume that $\wG=(\G,0,w)$ is a weighted graph without magnetic
    potential.  If $\G$ has multiple edges, we define a weighted graph
    $\wt \G=(\wt \G,0,\wt w)$ as follows: Denote by
    \begin{equation}
      \label{eq:mult.ed}
      E(v_-,v_+):=\set{e \in E(\G)}{\bd_-e=v_-, \bd_+ e=v_+}.
    \end{equation}
    Then $\card{E(v_-,v_+)}$ denotes the number of edges starting at
    $v_-$ and terminating at $v_+$.   
    We obtain a graph $\wt \G$ without multiple edges by identifying
    all edges in $E(v_-,v_+)$ into a single edge $\wt e$ provided
    $\card{E(v_-,v_+)} \ge 2$ and keeping all other edges.  As
    incidence function we set
    $\map{\wt \bd}{E(\wt \G)}{V(\G)\times V(\G)}$,
    $\wt \bd(\wt e)=(\bd_- e, \bd_+ e)$ for $e \in E(v_-,v_+)$.  The
    weight function of $\wt \wG$ is defined as
    \begin{equation*}
      \wt w_{\wt e}=\sum_{e \in \wt e} w_e.
    \end{equation*}
    Then the graphs $\wG$ and $\wt \wG$ are isolaplacian.  We call
    $\wt \wG$ the \emph{underlying simple weighted graph}.  Note that
    the graphs $\wG$ and $\wt \wG$ are not isomorphic (as a multiple
    edge changes the Betti number).

  \item
    \label{mult.ed.weighed.gr.b}
    On the other hand, a weighted graph $\wG$ with weights
    $w_e \in \N$ can be turned into a graph with standard weights
    $\wG'$ by replacing an edge $e$ by $w_e$ multiple edges $e'$ with
    weight $w'_{e'}=1$.  Again, it is easy to see that the graphs
    $\wG$ and $\wG'$ are isolaplacian, but again not isomorphic (if at
    least one weight fulfils $w_e \ne 1$).
  \end{enumerate}
\end{remark}

\begin{remark}[signed graphs as a special case of magnetic graphs]
  \label{rem:signed.graphs}
  If the magnetic potential $\alpha$ has values only in $\{0,\pi\}$,
  then the magnetic potential is also called a \emph{signature} and
  $\wG$ is called a \emph{signed graph} (see e.g.~\cite{llpp:15} and
  references therein).  The corresponding Laplacians (see below) are
  called \emph{signed Laplacian}, including the so-called
  \emph{signless} Laplacian, where $\alpha_e=\pi$ for all $e \in
  E$. In this sense the magnetic potential can be interpreted as a
  parameter interpolating different Laplacians.
\end{remark}

The construction presented later shows isospectrality for magnetic
Laplacians including signed and signless Laplacians.

The next proposition collects some well-know properties of the
discrete magnetic Laplacian (for a proof, see
e.g.~\cite[Section~3]{fclp:18}):
\begin{proposition}[magnetic Laplacians and cohomologous magnetic
  potentials]
  \label{prp:dml}
  Let $\wG=(\G,\alpha,w)$ be a normalised magnetic graph and $\lapl \wG$
  its associated Laplacian.
  \begin{enumerate}
  \item
    \label{dml.a}
    $\lapl \wG$ is a non-negative and self-adjoint operator with
    spectrum contained in $[0,2]$.
  \item
    \label{dml.b}
    If $\alpha$ and $\alpha'$ are cohomologous ($\alpha \sim \alpha'$,
    i.e., there is $\map \xi V \RmodZ$ such that
    $\alpha'_e=\alpha_e+\xi(\bd_+e)-\xi(\bd_-e)$ for all
    $e \in E(\G)$), then $\lapl {(\G,\alpha,w)}$ and
    $\lapl {(\G,\alpha',w)}$ are unitarily equivalent and have the same
    spectrum, i.e., $(\G,\alpha,w)$ and $(\G,\alpha',w)$ are
    isospectral.
  \item
    \label{dml.c}
    In particular, if $\alpha\sim 0$ then $\lapl {(\G,\alpha,w)}$ is
    unitarily equivalent with the usual normalised Laplacian
    $\lapl {(\G,0,w)}$, i.e., $(\G,\alpha,w)$ and $(\G,0,w)$ are
    isospectral.

  \item
    \label{dml.d}
    If the underlying graph of $\wG$ is a tree, then $\alpha\sim 0$
    and $\lapl {(\G,\alpha,w)}$ and $\lapl {(\G,0,w)}$ are unitarily
    equivalent for any magnetic potential $\alpha$.
  \item
    \label{dml.e}
    Let $\alpha=\pi$ and $\G$ be connected.  Then $\alpha \sim 0$ if
    and only if $\G$ is bipartite.
  \end{enumerate}
\end{proposition}
\begin{proof}
  Only \itemref{dml.e} needs some comments: ``$\Rightarrow$'':
  $\alpha \sim 0$ implies the existence of $\map \xi V \RmodZ$ such
  that $\pi=\alpha_e=\xi(\bd_+e)-\xi(\bd_-e)$.  Adding a constant to
  $\xi$ if necessary, we can assume that $\xi(V)=\{0,\pi\}$. The
  bipartite partition of vertices can be chosen as $A=\xi^{-1}(\{0\})$
  and $B=\xi^{-1}(\{\pi\})$.

  ``$\Leftarrow$'': Let $V=A\dcup B$ be a disjoint decomposition of
  the bipartite graph (i.e., there are no edges between vertices in
  $A$ respectively\ $B$).  We define $\xi(v)=0$ if $v \in A$ and
  $\xi(v)=\pi \in \RmodZ$ if $v \in B$. It can then be shown that
  $\alpha_e=\pi=\xi(\bd_+e)-\xi(\bd_-e)$ for all $e \in E$.
\end{proof}

\subsection{Dirichlet magnetic Laplacians}
\label{subsec:dir.mw.lapl}
We define here the notion of a ``Dirichlet boundary condition'' on a
subset $V_0$ of vertices for a discrete graph. We call this process
\emph{vertex virtualisation}.  Concretely, let $\wG=(\G,\alpha,w)$ be
a magnetic graph and $V_0 \subset V(\G)$.  The virtualisation of the
vertices $V_0$ specifies a partial subgraph\footnote{For a partial
  subgraph, the incidence function
  $\map \bd {E(\G)} {V(\G) \times V(\G)}$ may not map all initial or
  final vertices into $V(\G) \setminus V_0$ (see Section~2
  of~\cite{fclp:18} for additional results and
  motivation). Mohar~\cite{mohar:92} also calls edges with one vertex
  not in the graph also ``free edges''.}
$\DirGr \wG {V_0}=(\DirGr \G {V_0},w^+,\alpha^+)$, where
$V(\DirGr \G {V_0})=V(\G)\setminus V_0$.  Moreover,
$E(\G^+)=E(\G)\setminus E(V_0)$, where $E(V_0)$ is the set edges
starting and ending in $V_0$.

The weight function and the magnetic potential are defined as the
restriction to $E(\G^+)$, namely $w^+:=w \restr{E(\G^+)}$ and
$\alpha^+:=\alpha \restr{E(\G^+)}$. For simplicity of notation we will
not distinguish the restriction by the label $(\cdot)^+$, and use $w$
resp.\ $\alpha$ instead of $w^+$ resp.\ $\alpha^+$.  We denote the
normalised magnetic partial subgraph with virtualised vertices $V_0$
obtained from $\wG$ hence by
$\DirGr \wG {V_0}=(\DirGr \G {V_0},\alpha,w)$.  The \emph{order} of
$\DirGr \wG {V_0}$ is defined as
$\card{\DirGr \wG {V_0}} = \card{\wG}-\card{V_0} =
\card{V(\wG)}-\card{V_0}$.  For an example as matrix representation
see the end of this subsection.

\begin{definition}[(normalised) magnetic Dirichlet Laplacian]
  \label{def:dir.lapl}
  Let $\wG=(\G,\alpha,w)$ be a magnetic graph and $V_0 \subset V(\G)$.
  The \emph{(normalised) magnetic Dirichlet Laplacian} with Dirichlet
  conditions on $V_0$ is defined as
  \begin{equation*}
    \laplDir \wG {V_0} := \iota_{V_0}^* \lapl \wG \iota_{V_0},
  \end{equation*}
  where $\map {\iota_{V_0}}{\lsqr{V \setminus V_0,\deg^w}}{\lsqr{V,\deg^w}}$ is
  the natural extension by $0$ on $V_0$. If $V_0=\{v_0\}$,
  we simply write $\laplDir \wG {v_0}$.  We denote the spectrum of
  $\laplDir \wG{V_0}$ by $\spec{\DirGr \wG {V_0}}$.
\end{definition}
Note that the $k$-th eigenvalue of the Dirichlet Laplacian (written in
ascending order) is larger than the corresponding $k$-th eigenvalue of
the Laplacian without restrictions, a consequence of the interlacing
theorem for eigenvalues.  This justifies the superscript $(\cdot)^+$
in the notation.  The relation between both Laplacians is as follows:
if both Laplacians are represented as a matrix, then the Dirichlet
Laplacian corresponds to a principal submatrix of the original
Laplacian.

As an example we consider again the graph $\wG^\theta$ in
\Fig{motivating-ex}, now with $V_0=\{v_1\}$ as Dirichlet vertex.  The
Dirichlet Laplacian $\laplDir {(\wG^\theta)} {V_0}$ is now the principal
submatrix obtained from the original matrix~\eqref{eq:matrix.ex} by
deleting the row and column corresponding to the Dirichlet vertex
$v_1$, here the first one, i.e.,
\begin{equation*}
  \laplDir {(\wG^\theta)} {v_0} \cong
  \begin{pmatrix}
    1 & -\frac1{\sqrt 3}\\
    -\frac1{\sqrt 3} & 1
  \end{pmatrix}.
\end{equation*}
Note that this graph is now independent of the magnetic potential
$\theta$ (see also \Sec{motivating-ex}).

\subsection{Metric graph Laplacians and isospectrality}
\label{subsec:met.graphs}
Given a discrete graph $\G$ we construct a corresponding equilateral
metric graph $\mG$ as the topological graph as follows: We choose an
orientation $\ori E$ of $E$, i.e., a partition
$E=\ori E \dcup \ol {\ori E}$.  We consider the space
$\bigdcup_{e \in \ori E}[0,1] \times \{e\}$, where each edge
$e \in \ori E$ is identified with the interval $[0,1] \times \{e\}$
and having length $1$. We define
$\map\psi {\bigdcup_{e \in \ori E}\{0,1\} \times \{e\}} {V(\G)}$ by
$\psi(0,e)=\bd_-e$ and $\psi(1,e)=\bd_+e$, mapping the endpoints of
the intervals to the corresponding vertices.

The metric graph $\mG$ is then the quotient space
$\bigdcup_{e \in \ori E}[0,1] \times \{e\}/\psi$, where the interval
endpoints are identified according to the graph $\G$.  In particular,
we consider $V$ as subset of $\mG$, and it makes sense to speak of a
\emph{continuous} function at a vertex in $\mG$.  We have a natural
coordinate $x_e \in [0,1]$ and a natural measure on $\mG$, the
Lebesgue measure on each interval.  For a function $\map f \mG \C$ we
define $f_e(x):= f(x,e)$ for $e \in \ori E$ and $f_e(x):=f(1-x,\ol e)$
for $e \in \ol{\ori E}$ and $x \in [0,1]$.  The corresponding natural
Hilbert space is
$\Lsqr \mG \cong \bigoplus_{e \in \ori E} \Lsqr{[0,1] \times \{e\}}$.

The \emph{Kirchhoff} (sometimes also called \emph{standard} or
\emph{Neumann}) Laplacian $\lapl \mG$ on $\mG$, acting on functions
$f=(f_e)_e \in \Lsqr \mG$ is $(\lapl \mG f)_e = -f''_e$ for functions
$f_e$ and their two weak derivatives in $\Lsqr{[0,1]}$ satisfying
\begin{equation*}
  \text{$f$ continuous at each vertex $v$} \quadtext{and}
  \sum_{e \in E_v} f_e'(0)=0
\end{equation*}
for all $v \in V$ (note that $f_{\ol e}'(0)=-f_e'(1)$ for
$e \in \ori E$).  For more details on metric graphs, we refer for
example to~\cite{berkolaiko-kuchment:13}.  For simplicity, we consider
graphs without magnetic potentials here only.

We use a beautiful relation between the spectra of the standard
Laplacian and the Kirchhoff Laplacian (see for
example~\cite[Theorem~1]{von-below:01},
or~\cite{lledo-post:08b,kurasov-muller:pre21,kuchment:08} and
references cited therein (recall that $\G$ refers here to the weighted
graph $(\G,0,\1)$ with standard weights):
\begin{proposition}
  \label{prp:met.gr.isospec}
  Let $\G$ resp.\ $\G'$ be two (finite) discrete graphs and $\mG$
  resp.\ $\mG'$ the corresponding equilateral metric graphs.  Then the
  following are equivalent:
  \begin{enumerate}
  \item $\G$ and $\G'$ are isospectral (with respect to the discrete
    standard Laplacian), and $\G$ and $\G'$ have the same number of
    edges;
  \item the metric Kirchhoff Laplacians on $\mG$ and $\mG'$ are isospectral.
  \end{enumerate}
\end{proposition}
\begin{proof}
  The spaces $\mGr N_{\mGr \lambda}=\ker(\lapl \mG -\mGr \lambda)$ and
  $N_\lambda=\ker(\lapl \G-\lambda)$ are isomorphic for
  $\lambda=1-\sqrt {\mGr \lambda} \in (0,2)$ and
  $\sqrt{\mGr \lambda} \notin \pi \N_0$ (see
  e.g.~\cite[Theorem~1]{von-below:01}
  or~\cite[Proposition~4.1]{lledo-post:08b}\footnote{In~\cite{lledo-post:08b},
    we assumed that the underlying graphs are connected.  The results
    there extend straightforward to the case of several connected
    components.} and references therein).  In particular, if $\G$ and
  $\G'$ are isospectral, then $\mG$ and $\mG'$ are isospectral up to
  the eigenvalues of the form $\mGr \lambda_n=n^2\pi^2$ and
  $n \in \N_0$.  For $n \in 2\N_0$, an eigenfunction constant on a
  connected component of the discrete graph (with value $1$, say)
  corresponds to the eigenfunction $\phi_e(x)=\cos(n\pi x)$ on each
  edge of the same connected component in $\mG$.  For $n \in 2\N_0+1$,
  the multiplicity of the correspondent discrete eigenvalue
  $\lambda=2$ counts the number of connected bipartite components (see
  e.g.~\cite[Proposition~2.3]{lledo-post:08b} and references therein), and
  each such eigenfunction (with values $\pm 1$) leads to an
  eigenfunction $\phi_e(x)=\cos(n \pi x)$ on each edge of the
  corresponding connected bipartite component.

  The remaining eigenfunctions of $\lapl \mG$ corresponding to
  eigenvalues $\mGr \lambda_n=n^2\pi^2$ are $0$ on all vertices; we
  called them \emph{topological} in~\cite[Definition~4.4]{lledo-post:08b}.
  They are entirely determined by the homology and the bipartiteness
  of the discrete graph (\cite[Lemma~5.1 and
  Proposition~5.2]{lledo-post:08b}).  As the homology of the graph is
  determined by the number of vertices \emph{and} edges, and as the
  number of bipartite connected component can be detected from the
  spectrum of $\lapl \G$, the stated equivalence follows.
\end{proof}

%
%
\section{A motivating family of examples}
\label{sec:motivating-ex}
%

Before going into a formal description of the construction of
isospectral magnetic graphs in \Sec{frames-and-partitions} we
illustrate the main idea of the construction with a class of
examples. This motivating family of examples of isospectral graphs
generalise those given by Butler and Grout
in~\cite[Example~2]{butler-grout:11} allowing multiple edges, a
general magnetic potential and dropping the bipartiteness condition.
The construction begins with the choice of a building block, out of
which we will assemble the frame members.  The building block $\G$ in
the example is given by a graph with three vertices $\{v_1,v_2,v_3\}$,
and three edges, two of which are multiple joining $v_1$ and $v_2$
(see \Fig{motivating-ex}).  For $\theta \in \RmodZ$ we add a magnetic
potential $\alpha_e^\theta=\theta$ on one of the two parallel edges
$e$ and $\alpha^\theta_e=0$ on the remaining two edges, and call the
resulting magnetic graph
$\wG^\theta=(\G,1,\alpha^\theta)$.\footnote{Note that if $\theta=\pi$,
  then the entry in the matrix of $\lapl {\wG^\theta}$ corresponding
  to the edges joining $v_1$ and $v_2$ (cf.~\eqref{eq:lapl.matrix}) is
  actually $0$, due to the two parallel edges, one of them with
  magnetic potential $\pi$, the other one $0$, hence
  $\e^0 + \e^{\im \pi}=0$.  In particular, $\wG^\pi$ is isolaplacian
  (see \Rem{mult.ed.weighed.gr}) with the disjoint union of a graph
  with one isolated vertex (having spectrum $1$ by definition), and
  $\DirGr \G {v_1}$; in accordance with the spectrum
  $\lMset 1-1/\sqrt 3,1,1+1/\sqrt 3\rMset$ given
  in~\eqref{eq:spec.block} for $\theta=\pi$.  This special case is
  hence a procedure to ``delete'' edges while keeping the original
  degrees.} Examples with simple graphs (i.e.\ without multiple edges
or loops) are given in \Sec{examples}

Let $V_0=\{v_1,v_3\}$ be the set the two non-adjacent vertices of
$\G$.  For each $a \in \N$ we contract $a$ copies of $\wG$ by merging
$v_1$ (respectively $v_3$) in each copy to one vertex. The edge weight
and the magnetic potential remains the same after identification on
each copy.  The resulting magnetic graph is called \emph{$a$-th frame
  member} $\wGr F_a^\theta:=\wGr F_a(\wG^\theta, V_0)$, as in
\Fig{motivating-ex} (see also the general \Def{frame}). We denote by
$(\wGr F_a)_{a\in \N}$ the frame obtained by the building block
$G$ and the choice of merging vertices $V_0$.
\begin{figure}[h]
  \centering
  {
    \begin{tikzpicture}[auto,
      vertex/.style={circle,draw=black!100,fill=black!100, 
        inner sep=0pt,minimum size=1mm},scale=1]
      \foreach \x in {0}
      { \node(C\x) at (0,1) [vertex,label=left:$v_3$] {};
        \node (B\x) at (\x*.5,0) [vertex,label=left:$v_2$] {};
        \path [-](C\x) edge node[right] {} (B\x);
        \node (A\x) at (0,-1) [vertex,draw=black,fill=none,label=left:$v_1$]
             {};
        \path [bend left=7](B\x) edge node[right] {} (A\x);
        \path [bend right=7](B\x) edge node[right] {} (A\x);
      }
      \draw[] (.3,.-1.5)node[left] {$\wG^\theta=\wGr F^\theta_1$};
    \end{tikzpicture}
    \begin{tikzpicture}[auto,
      vertex/.style={circle,draw=black!100,fill=black!100, 
        inner sep=0pt,minimum size=1mm},scale=1]
      \foreach \x in {-1,1}
      { \node(C\x) at (0,1) [vertex,label=below:] {};
        \node (B\x) at (\x*.5,0) [vertex,label=below:] {};
        \path [-](C\x) edge node[right] {} (B\x);
        \node (A\x) at (0,-1) [vertex,draw=black,fill=none,label=below:] {};
        \path [bend left=7](B\x) edge node[right] {} (A\x);
        \path [bend right=7](B\x) edge node[right] {} (A\x);
      }
      \draw[] (.3,.-1.5)node[left] {$\wGr F^\theta_2$};
    \end{tikzpicture}
    \begin{tikzpicture}[auto,
      vertex/.style={circle,draw=black!100,fill=black!100, 
        inner sep=0pt,minimum size=1mm},scale=1]
      \foreach \x in {-1,0,1}
      { \node(C\x) at (0,1) [vertex,label=below:] {};
        \node (B\x) at (\x*.5,0) [vertex,label=below:] {};
        \path [-](C\x) edge node[right] {} (B\x);
        \node (A\x) at (0,-1) [vertex,draw=black,fill=none,label=below:] {};
        \path [bend left=7](B\x) edge node[right] {} (A\x);
        \path [bend right=7](B\x) edge node[right] {} (A\x);
      }
      \draw[] (.3,.-1.5)node[left] {$\wGr F^\theta_3$};
    \end{tikzpicture}
    \begin{tikzpicture}[auto,
      vertex/.style={circle,draw=black!100,fill=black!100, 
        inner sep=0pt,minimum size=1mm},scale=1]
      \foreach \x in {-2,-1,1,2}
      { \node(C\x) at (0,1) [vertex,label=below:] {};
        \node (B\x) at (\x*.5,0) [vertex,label=below:] {};
        \path [-](C\x) edge node[right] {} (B\x);
        \node (A\x) at (0,-1) [vertex,draw=black,fill=none,label=below:] {};
        \path [bend left=7](B\x) edge node[right] {} (A\x);
        \path [bend right=7](B\x) edge node[right] {} (A\x);
      }
      \draw[] (.3,.-1.5)node[left] {$\wGr F^\theta_4$};
    \end{tikzpicture}
    \begin{tikzpicture}[auto,
      vertex/.style={circle,draw=black!100,fill=black!100, 
        inner sep=0pt,minimum size=1mm},scale=1]
      \foreach \x in {-2,-1,0,1,2}
      { \node(C\x) at (0,1) [vertex,label=below:] {};
        \node (B\x) at (\x*.5,0) [vertex,label=below:] {};
        \path [-](C\x) edge node[right] {} (B\x);
        \node (A\x) at (0,-1) [vertex,draw=black,fill=none,label=below:] {};
        \path [bend left=7](B\x) edge node[right] {} (A\x);
        \path [bend right=7](B\x) edge node[right] {} (A\x);
      }
      \draw[] (.3,.-1.5)node[left] {$\wGr F^\theta_5$};
    \end{tikzpicture}
    \begin{tikzpicture}[auto,
      vertex/.style={circle,draw=black!100,fill=black!100, 
        inner sep=0pt,minimum size=1mm},scale=1]
      \foreach \x in {-3,-2,-1,1,2,3}
      { \node(C\x) at (0,1) [vertex,label=below:] {};
        \node (B\x) at (\x*.5,0) [vertex,label=below:] {};
        \path [-](C\x) edge node[right] {} (B\x);
        \node (A\x) at (0,-1) [vertex,draw=black,fill=none,label=below:] {};
        \path [bend left=5](B\x) edge node[right] {} (A\x);
        \path [bend right=5](B\x) edge node[right] {} (A\x);
      }
      \draw[] (.3,.-1.5)node[left] {$\wGr F^\theta_6$};
      \node[circle,scale=0.3] (D1) at (2,0) [vertex,label=below:] {};
      \node[circle,scale=0.3] (D2) at (2.5,0) [vertex,label=below:] {};
      \node[circle,scale=0.3] (D3) at (3,0) [vertex,label=below:] {};
    \end{tikzpicture}
  } \caption{The frame $(\wGr F^\theta_a)_{a \in \N}$ given by the
    frame members $\wGr F^\theta_a$ for $a\in \N$.  Each graph
    $\wGr F^\theta_a$ has a so-called distinguished bottom vertex
    (outlined); it will be used in the next step for the contracted
    frame union.}
  \label{fig:motivating-ex}
\end{figure}
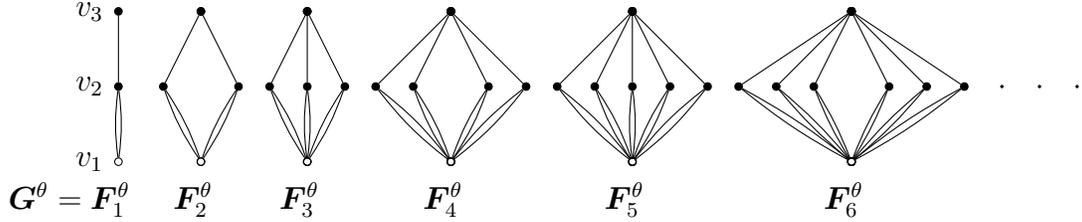

Note that the graphs in the family given in~\Fig{motivating-ex} have a
high degree of symmetry, and this fact is reflected in the spectrum
through eigenvalues with high multiplicity (see \Rem{dualgroup}). The
idea of the construction is to use the frame
$(\wGr F^\theta_a)_{a \in \N}$ to construct new non-isomorphic graphs
that preserve most of the eigenvalues of the initial frame. The
geometric procedure to assemble different frame members to specify the
families of isospectral graphs are determined by $s$-partitions of a
natural number $r\in\N$ (cf. \Def{partition} for a formal definition).

To be concrete, we construct two isospectral graphs using the family
given in~\Fig{motivating-ex} as follows. Consider,
for example, the two different $2$-partitions of the number
$4$, namely $A=\lMset 1,3 \rMset$ and $B=\lMset 2,2 \rMset$, i.e.,
\begin{equation*}
  4=1+3 =2+2
\end{equation*}
(here $\lMset 2,2\rMset$ is the multiset with one element $2$ of
multiplicity $2$, see \App{multisets}).  We choose the vertex
$v_1 \in V_0$ of degree $2$ (due to the double edge) as distinguished
vertex, and we construct a graph $\wGr F_{A,v_1}^\theta$ (called later
\emph{$v_1$-contracted frame union}, see \Def{contr.frame.union})
associated with the first partition $A=\lMset 1,4\rMset$ as follows
(see \Fig{ExampleCospectral}): contract the copies of the
distinguished vertex $v_1$ (outlined vertices) from the frame members
$\wGr F_1^\theta$ and $\wGr F_3^\theta$ in~\Fig{motivating-ex}.
Similarly, define $\wGr F_{B,v_1}^\theta$ by contracting two copies of
$\wGr F_2^\theta$ along the distinguished vertex $v_1$.
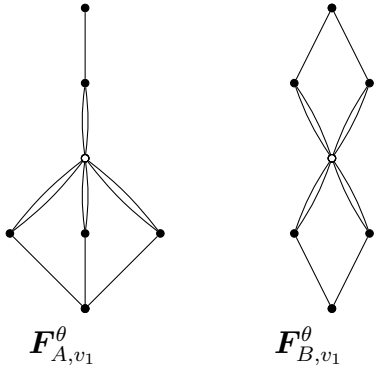
\begin{figure}[h]
  \centering
  {
    \begin{tikzpicture}[auto,
      vertex/.style={circle,draw=black!100,fill=black!100, 
        inner sep=0pt,minimum size=1mm},scale=1]
      \foreach \x in {0}
      { \node(C\x) at (0,1) [vertex,draw=black,fill=none,label=below:] {};
        \node (B\x) at (\x*.5,2) [vertex,label=below:] {};
        \path [bend left=7](C\x) edge node[right] {} (B\x);
        \path [bend right=7](C\x) edge node[right] {} (B\x);
        \node (A\x) at (0,3) [vertex,label=below:] {};
        \path [-](B\x) edge node[right] {} (A\x);
      }
      \foreach \x in {-2,0,2}
      { \node(C\x) at (0,1) [vertex,draw=black,fill=none,label=below:] {};
        \node (B\x) at (\x*.5,0) [vertex,label=below:] {};
        \path [bend left=7](C\x) edge node[right] {} (B\x);
        \path [bend right=7](C\x) edge node[right] {} (B\x);
        \node (A\x) at (0,-1) [vertex,label=below:] {};
        \path [-](B\x) edge node[right] {} (A\x);
      }
      \draw[] (.3,.-1.5)node[left] {$\wGr F^\theta_{A,v_1}$};
    \end{tikzpicture}\quad \quad \quad
    \begin{tikzpicture}[auto,
      vertex/.style={circle,draw=black!100,fill=black!100,  
        inner sep=0pt,minimum size=1mm},scale=1]
      \foreach \x in {-1,1}
      { \node(C\x) at (0,1) [vertex,draw=black,fill=none,label=below:] {};
        \node (B\x) at (\x*.5,2) [vertex,label=below:] {};
        \path [bend left=7](C\x) edge node[right] {} (B\x);
        \path [bend right=7](C\x) edge node[right] {} (B\x);
        \node (A\x) at (0,3) [vertex,label=below:] {};
        \path [-](B\x) edge node[right] {} (A\x);
      }
      \foreach \x in {-1,1}
      { \node(C\x) at (0,1) [vertex,draw=black,fill=none,label=below:] {};
        \node (B\x) at (\x*.5,0) [vertex,label=below:] {};
        \path [bend left=7](C\x) edge node[right] {} (B\x);
        \path [bend right=7](C\x) edge node[right] {} (B\x);
        \node (A\x) at (0,-1) [vertex,label=below:] {};
        \path [-](B\x) edge node[right] {} (A\x);
      }
      \draw[] (.3,.-1.5)node[left] {$\wGr F^\theta_{B,v_1}$};
    \end{tikzpicture}
  } \caption{The contracted frame unions $\wGr F^\theta_{A,v_1}$ and
    $\wGr F^\theta_{B,v_1}$ for the two different $2$-partitions
    $A=\lMset 1,3\rMset$ and $B=\lMset 2,2\rMset$ of $4$ are
    isospectral, but not isomorphic for each value $\theta \in \RmodZ$
    of the magnetic potential.}
\label{fig:ExampleCospectral}
\end{figure}
An explicit computation of the eigenvalues of the corresponding
standard Laplacians shows that both graphs $\wGr F^\theta_{A,v_1}$ and
$\wGr F^\theta_{B,v_1}$ are isospectral with spectrum
\begin{equation*}
  \Bigl\lMset
  1-\frac1{\sqrt 3}\sqrt{2+\cos\theta},
  1-\frac 1{\sqrt 3},
  1^{(3)},
  1+\frac 1{\sqrt 3},
  1+\frac1{\sqrt 3}\sqrt{2+\cos\theta},
  \Bigr\rMset,
\end{equation*}
where $1^{(3)}$ means that $1$ has multiplicity $3$, but the
corresponding discrete graphs are not isomorphic.  We actually prove
in \Thm{spec.contr.frame.union} the general result that for an
$s$-partition of $r$ the spectrum of the normalised magnetic Laplacian
is given by
\begin{equation}
  \label{eq:spec.contr.frame.union}
 \spec {\wGr F^\theta_{A,v_1}}
 = \spec {\wG^\theta}
 \uplus \spec {\DirGr {(\wG^\theta)}{v_1}}^{(s-1)}
 \uplus \spec {\DirGr {(\wG^\theta)}{V_0}}^{(r-s)},
\end{equation}
where $\uplus$ denotes the multiset union, see \App{multisets}.  In
our example we have $r=4$ and $s=2$ and the spectra of the building
block $\wG^\theta$ and the magnetic graphs with vertices $v_1$ and
$V_0=\{v_1,v_3\}$ virtualised are given respectively by
\begin{gather}
  \label{eq:spec.block}
  \spec {\wG^\theta}
  = \Bigl\lMset
  1-\frac1{\sqrt 3}\sqrt{2+\cos \theta},
  1,
  1+\frac1{\sqrt 3}\sqrt{2+\cos \theta}
  \Bigr\rMset
  =:\lMset\lambda_1^\theta,1,\lambda_3^\theta\rMset,\\
  \nonumber
  \bigspec{\DirGr {(\wG^\theta)}{v_1}}
  = \Bigl\lMset 1-\frac 1{\sqrt 3},1+\frac 1{\sqrt 3} \Bigr\rMset
  \qquadtext{and}
  \bigspec{\DirGr {(\wG^\theta)}{V_0}}
  = \lMset 1 \rMset.
\end{gather}
\begin{figure}[h]
  \centering
  \begin{tikzpicture}[auto,
    vertex/.style={circle,draw=black!100,fill=black!100, 
      inner sep=0pt,minimum size=1mm},scale=1]
    \begin{scope}
      \pgftransformcm{0.7}{0}{0}{0.7}{\pgfpoint{0cm}{0cm}}
      \begin{scope}
        \pgftransformcm{2}{1}{2.5}{0}{\pgfpoint{0cm}{0cm}}
        \foreach \x in {0}
        { \node(C\x) at (0,1) [vertex,draw=black,fill=none,label=below:] {};
          \node (B\x) at (\x*.5,2) [vertex,label=below:] {};
          \path [bend left=7](C\x) edge node[right] {} (B\x);
          \path [bend right=7](C\x) edge node[right] {} (B\x);
          \node (A\x) at (0,3) [vertex,label=below:] {};
          \path [-](B\x) edge node[right] {} (A\x);
        }
      \end{scope}
      \node (BBB) at (5,2) [label=below:]{};
      \node (AAA) at (7.5,2) [label=below:]{};
      \begin{scope}[dotted,>=stealth]
        \draw[->] (B0) -- (BBB);
        \draw[->] (A0) -- (AAA);
      \end{scope}
      \begin{scope}
        \pgftransformcm{2}{1}{2.5}{0}{\pgfpoint{0cm}{0cm}}
        \node (c) at (0,-5)  [vertex,draw=black,fill=none,label=below:] {};
        \node (b) at (0,-4) [vertex,label=below:] {};
        \node (a) at (0,-3) [vertex,label=below:] {};
          \path [bend left=7](c) edge node[right] {} (b);
          \path [bend right=7](c) edge node[right] {} (b);
          \path [-](b) edge node[right] {} (a);
        \draw[] (.3,.-5.5)node[left] {$\wG^\theta$};

        \foreach \x in {-2,0,2}
        { \node(C\x) at (0,1) [vertex,draw=black,fill=none,label=below:] {};
          \node (B\x) at (\x*.5,0) [vertex,label=below:] {};
          \path [bend left=7](C\x) edge node[right] {} (B\x);
          \path [bend right=7](C\x) edge node[right] {} (B\x);
          \node (A\x) at (0,-1) [vertex,label=below:] {};
          \path [-](B\x) edge node[right] {} (A\x);
        }
        \draw[] (.3,.-1.5)node[left] {$\wGr F^\theta_{A,v_1}$};
      \end{scope}
      \draw[] (-13.5,-1.0) node[right]
      {$\spec{\wG^\theta}=\lMset
        \lambda_1^\theta,1,\lambda_3^\theta \rMset$};
      \node (cc) at (-12.5,2) [label=below:]{};
      \node (bb) at (-10.0,2) [label=below:]{};
      \node (aa) at (-7.5,2) [label=below:]{};
      \begin{scope}[dotted,>=stealth]
        \draw[->] (c) -- (cc);
        \draw[->] (b) -- (bb);
        \draw[->] (a) -- (aa);
      \end{scope}

      \node (CC) at (2.5,2) [label=below:]{};
      \node (BB2) at (2,3) [label=below:]{};
      \node (BB0) at (0.0,2) [label=below:]{};
      \node (BB-2) at (-2,1) [label=below:]{};
      \node (AA) at (-2.5,2) [label=below:]{};
      \begin{scope}[dotted,>=stealth]
        \draw[->] (C0) -- (CC); 
         \draw[->] (B2) -- (BB2);
         \draw[->] (B0) -- (BB0);
         \draw[->] (B-2) -- (BB-2);
         \draw[->] (A0) -- (AA);
      \end{scope}
    \end{scope}
  \end{tikzpicture}

        \begin{tikzpicture}[auto,
    vertex/.style={circle,draw=black!100,fill=black!100, 
      inner sep=0pt,minimum size=1mm},scale=1]
    \begin{scope}
      \pgftransformcm{0.7}{0}{0}{0.7}{\pgfpoint{0cm}{0cm}}
      \begin{scope}
        \pgftransformcm{2}{1}{2.5}{0}{\pgfpoint{0cm}{0cm}}
        \foreach \x in {0}
        { \node(C\x) at (0,1.5) [vertex,scale=2,draw=black,fill=none,label=below:] {};
          \node (B\x) at (\x*.5,2.5) [vertex,label=below:] {};
          \path [bend left=7](C\x) edge node[right] {} (B\x);
          \path [bend right=7](C\x) edge node[right] {} (B\x);
          \node (A\x) at (0,3.5) [vertex,label=below:] {};
          \path [-](B\x) edge node[right] {} (A\x);
        }
      \end{scope}
      \draw[] (-13.5,-1.0) node[right]
      {$\spec{\DirGr{(\wG^\theta)}{v_1}}=\lMset 1-1/\sqrt 3,1+1/\sqrt 3
        \rMset$};
      \node (BBB) at (6.2,-2) [label=below:]{};
      \node (AAA) at (8.8,1.4) [label=below:]{};
      \begin{scope}[dotted,>=stealth]
        \draw[->] (B0) -- (BBB);
        \draw[->] (A0) -- (AAA);
      \end{scope}
      \begin{scope}
        \pgftransformcm{2}{1}{2.5}{0}{\pgfpoint{0cm}{0cm}}
        \node (c) at (0,-5)  [vertex,scale=2.0,draw=black,fill=none,label=below:] {};
        \node (b) at (0,-4) [vertex,label=below:] {};
        \node (a) at (0,-3) [vertex,label=below:] {};
          \path [bend left=7](c) edge node[right] {} (b);
          \path [bend right=7](c) edge node[right] {} (b);
          \path [-](b) edge node[right] {} (a);
        \draw[] (.3,.-5.5)node[left] {$\DirGr {(\wG^\theta)}{v_1}$};

        \foreach \x in {-2,0,2}
        { \node(C\x) at (0,1) [vertex,scale=2.0,draw=black,fill=none,label=below:] {};
          \node (B\x) at (\x*.5,0) [vertex,label=below:] {};
          \path [bend left=7](C\x) edge node[right] {} (B\x);
          \path [bend right=7](C\x) edge node[right] {} (B\x);
          \node (A\x) at (0,-1) [vertex,label=below:] {};
          \path [-](B\x) edge node[right] {} (A\x);
        }
        \draw[] (.3,.-1.5)node[left] {$\DirGr {(\wGr F^\theta_3)}{v_1}$};
        \draw[] (-0.7,4.7)node[left] {$\DirGr {(\wGr F^\theta_1)}{v_1}=\DirGr {(\wG^\theta)}{v_1}$};
      \end{scope}
      \node (bb) at (-10.0,-2) [label=below:]{};
      \node (aa) at (-7.5,1.4) [label=below:]{};
      \begin{scope}[dotted,>=stealth]
        \draw[->] (b) -- (bb);
        \draw[->] (a) -- (aa);
      \end{scope}

      \node (CC) at (2.5,2) [label=below:]{};
      \node (BB2) at (2,3) [label=below:]{};
      \node (BB0) at (0.0,2) [label=below:]{};
      \node (BB-2) at (-2,1) [label=below:]{};
      \node (AA) at (-2.5,-1.4) [label=below:]{};
      \begin{scope}[dotted,>=stealth]
         \draw[->] (B2) -- (BB2);
         \draw[->] (B0) -- (BB0);
         \draw[->] (B-2) -- (BB-2);
         \draw[->] (A0) -- (AA);
      \end{scope}
    \end{scope}
  \end{tikzpicture}

    \begin{tikzpicture}[auto,
    vertex/.style={circle,draw=black!100,fill=black!100, 
      inner sep=0pt,minimum size=1mm},scale=1]
    \begin{scope}
      \pgftransformcm{0.7}{0}{0}{0.7}{\pgfpoint{0cm}{0cm}}
      \begin{scope}
        \pgftransformcm{2}{1}{2.5}{0}{\pgfpoint{0cm}{0cm}}
        \foreach \x in {0}
        { \node(C\x) at (0,1.5) [vertex,scale=2.0,draw=black,fill=none,label=below:] {};
          \node (B\x) at (\x*.5,2.5) [vertex,label=below:] {};
          \path [bend left=7](C\x) edge node[right] {} (B\x);
          \path [bend right=7](C\x) edge node[right] {} (B\x);
          \node (A\x) at (0,3.5) [vertex,scale=2.0,fill=none,label=below:] {};
          \path [-](B\x) edge node[right] {} (A\x);
        }
      \end{scope}
      \draw[] (-13.5,-1.0) node[right]
      {$\spec{\DirGr{(\wG^\theta)}{V_0}}=\lMset 1 \rMset$};
      \begin{scope}
        \pgftransformcm{2}{1}{2.5}{0}{\pgfpoint{0cm}{0cm}}
        \node (c) at (0,-5)  [vertex,scale=2.0,draw=black,fill=none,label=below:] {};
        \node (b) at (0,-4) [vertex,label=below:] {};
        \node (a) at (0,-3) [vertex,scale=2.0,fill=none,label=below:] {};
          \path [bend left=7](c) edge node[right] {} (b);
          \path [bend right=7](c) edge node[right] {} (b);
          \path [-](b) edge node[right] {} (a);
        \draw[] (.3,.-5.5)node[left] {$\DirGr {(\wG^\theta)}{V_0}$};

        \foreach \x in {-2,0,2}
        { \node(C\x) at (0,1) [vertex,scale=2.0,draw=black,fill=none,label=below:] {};
          \node (B\x) at (\x*.5,0) [vertex,label=below:] {};
          \path [bend left=7](C\x) edge node[right] {} (B\x);
          \path [bend right=7](C\x) edge node[right] {} (B\x);
          \node (A\x) at (0,-1) [vertex,scale=2.0,fill=none,label=below:] {};
          \path [-](B\x) edge node[right] {} (A\x);
        }
        \draw[] (.3,.-1.5)node[left] {$\DirGr{(\wGr F^\theta_3)}{V_0}$};
        \draw[] (-0.7,4.7)node[left] {$\DirGr{(\wGr F^\theta_1)}{V_0}=\DirGr{(\wG^\theta)}{V_0}$};
      \end{scope}
      \node (bb) at (-10.0,2) [label=below:]{};
      \begin{scope}[dotted,>=stealth]
        \draw[->] (b) -- (bb);
      \end{scope}

      \node (CC) at (2.5,2) [label=below:]{};
      \node (BB2) at (2,-1) [label=below:]{};
      \node (BB0) at (0.0,2) [label=below:]{};
      \node (BB-2) at (-2,1) [label=below:]{};
      \begin{scope}[dotted,>=stealth]
         \draw[->] (B2) -- (BB2);
         \draw[->] (B0) -- (BB0);
      \end{scope}
    \end{scope}
  \end{tikzpicture}

  \caption{The three types of eigenfunctions (represented by dotted
    vertical lines with arrow pointing in a virtual third dimension);
    Dirichlet vertices are marked outlined with a bigger
    circle:\newline %
    \emph{Top:} The eigenfunctions of $\wG^\theta$ are copied
    symmetrically onto each copy of $\wG^\theta$ in
    $\wGr F^\theta_{A,v_1}$.  Hence, the three eigenfunctions (in the
    picture, it is the constant one) become three eigenfunctions on
    $\wGr F^\theta_{A,v_1}$.  \newline %
    \emph{Middle:} each eigenfunction of $\DirGr {(\wG^\theta)}{v_1}$
    becomes a symmetric copy on each member
    $\DirGr{(\wGr F_a^\theta)}{v_1}$ of $\wGr F^\theta_{A,v_1}$ for
    $a \in A=\lMset 3,1\rMset$.  To make them orthogonal to the
    symmetric ones, we can choose only $s-1=1$ one here for each of
    the two eigenfunctions of $\DirGr {(\wG^\theta)}{v_1}$, hence
    $2(s-1)=2$ eigenvalues of $\wGr F^\theta_{A,v_1}$ are captured.
    \newline %
    \emph{Bottom:} we consider the eigenfunctions of
    $\DirGr {(\wG^\theta)}{V_0}$ (here only one) onto each copy of
    $\DirGr {(\wG^\theta)}{V_0}$ in $\DirGr{(\wGr F_a^\theta)}{v_1}$
    for $a \in \lMset 3,1\rMset$.  There are here $r-s=(3-1)+(1-1)=2$
    such eigenfunctions orthogonal to the previous ones, supported
    here only on $\DirGr{(\wGr F_3^\theta)}{v_1}$.  }
  \label{fig:expl.spec}
\end{figure}
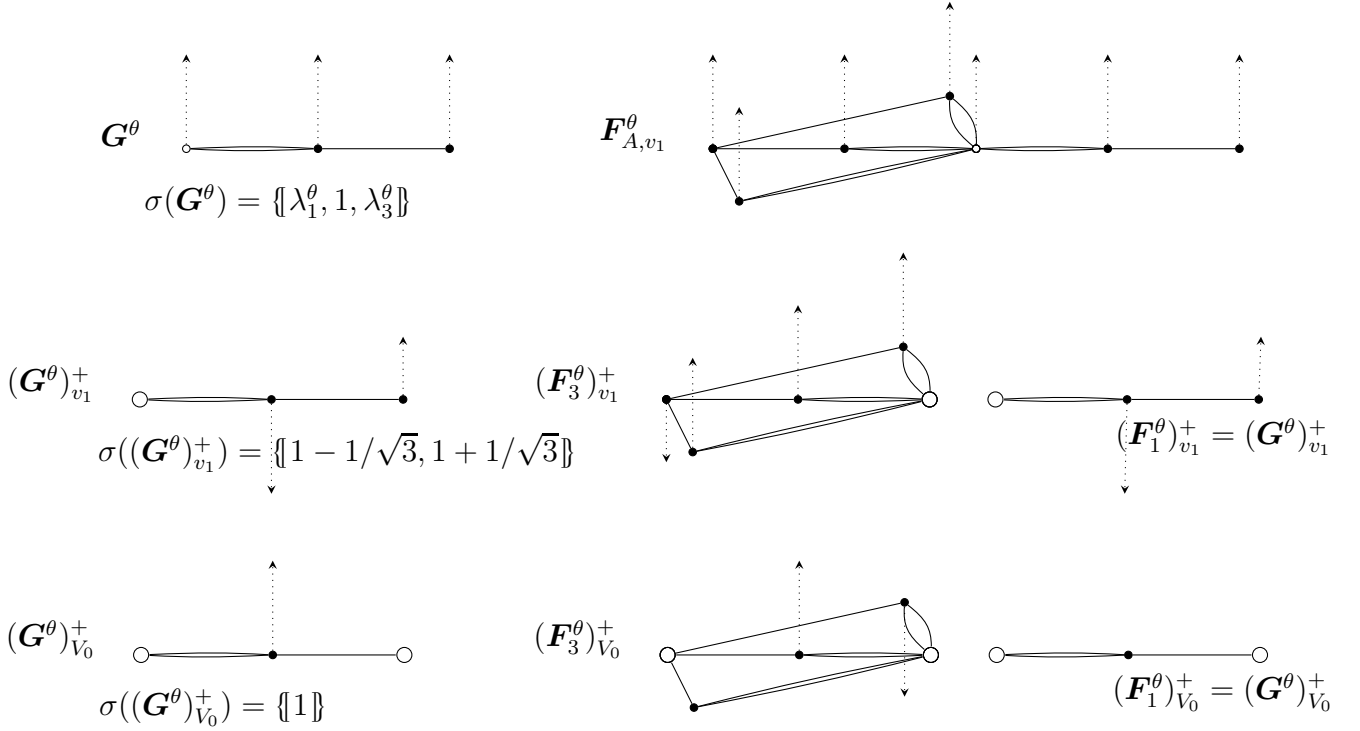

\begin{remark}[explanation of spectrum]
  We give here some heuristic reasons for the spectrum having the form
  as in~\eqref{eq:spec.contr.frame.union} and will formalise them in
  the proofs of the main results of the following sections (see \Prp{frame.spec}
  and \Thm{spec.contr.frame.union}).
  \begin{enumerate}
  \item Each eigenfunction of the building block $\wG^\theta$ carries
    over to $\wGr F^\theta_{A,v_1}$ by \emph{symmetrically} extending
    it to each copy of $\wG^\theta$ with the same eigenvalue; hence
    each eigenvalue of $\wG^\theta$ contributes one time its original
    multiplicity.  This gives three eigenvalues in our concrete
    example.
  \item Each eigenfunction of $\DirGr {(\wG^\theta)}{v_1}$ carries
    over to symmetric copies on each member $\wGr F_a^\theta$ of
    $\wGr F^\theta_{A,v_1}$ for $a \in A$ with the same eigenvalue.
    We can suitably choose $s-1$ copies of them being orthogonal to
    the symmetric ones constructed in the first item. This gives
    $2 \cdot(s-1)=2$ eigenvalues in our case as $s=2$.

  \item Finally, each eigenfunction of $\DirGr {(\wG^\theta)}{V_0}$
    carries over to an eigenfunction on each $\wGr F_a^\theta$ for
    $a \in A$ with the same eigenvalue.  There are $a-1$ of them on
    $\wGr F_a^\theta$ orthogonal to the symmetric ones.  As they
    vanish on the contracted vertices $V_0$, they remain
    eigenfunctions on $\wGr F^\theta_{A,v_1}$, and there are
    $\sum_{a \in A}(a-1)=r-s$ of them, again orthogonal to the
    previously constructed ones. Therefore we obtain $1 \cdot(r-s)=2$
    new eigenvalues in our case.
  \end{enumerate}
  One can see that all eigenvalues are captured via this procedure
  (here $r+s+1=7=3+2+2$).  In addition, the two graphs are not
  isomorphic, as $\wGr F^\theta_{A,v_1}$ has a pendant vertex while
  $\wGr F^\theta_{B,v_1}$ has no vertex of degree $1$.
\end{remark}

The important feature for the construction of families of isospectral
graphs in the preceding example is the partition of a natural number
that selects members of a frame $(\wGr F_a^\theta)_{a \in \N}$ that
will be contracted at the distinguished vertices (the outlined vertex
$v_1$ in our example).  In fact, there is an infinite collection of
(finite) families of isospectral graphs that can be constructed
similarly.  We illustrate this with all different $4$-partitions of
$8$:
\begin{equation*}
  A_1=\lMset 1,1,1,5 \rMset, \quad
  A_2=\lMset 1,1,2,4 \rMset, \quad
  A_3=\lMset 1,1,3,3 \rMset, \quad
  A_4=\lMset [1,2,2,3 \rMset \quadtext{and}
  A_5=\lMset 2,2,2,2 \rMset.
\end{equation*}
For each of the different $s$-partitions $A_q$ ($q \in \{1,2,3,4,5\}$)
($s=4$) of $r=8$ construct the graphs $\wGr F_{A_q,v_1}^\theta$ as
before (see \Fig{ExampleCospectral3}).
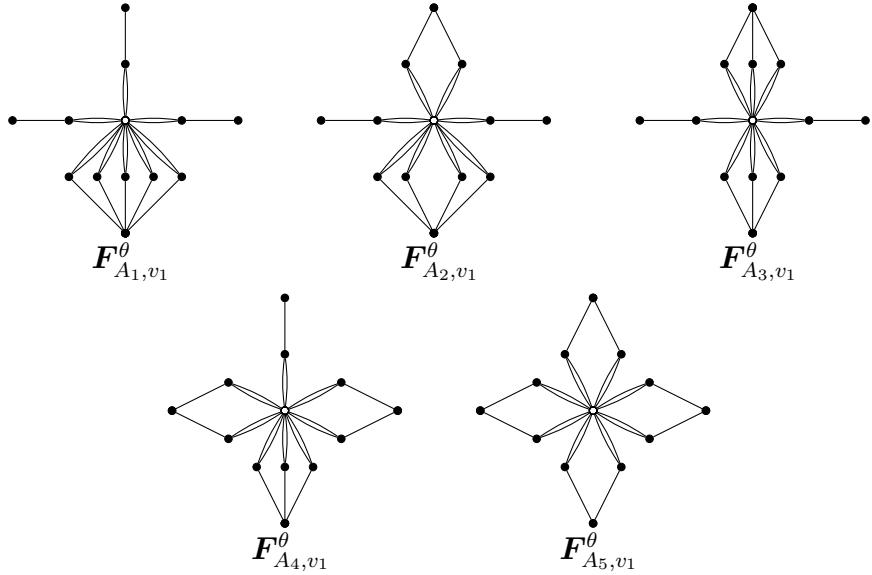
\begin{figure}[h]
  \centering
  {
    \begin{tikzpicture}[auto,
      vertex/.style={circle,draw=black!100,fill=black!100, 
        inner sep=0pt,minimum size=1mm},scale=.75]
      \foreach \x in {0}
      { \node(C\x) at (0,1) [vertex,draw=black,fill=none,label=below:] {};
        \node (B\x) at (\x*.5,2) [vertex,label=below:] {};
        \path [bend left=7](C\x) edge node[right] {} (B\x);
        \path [bend right=7](C\x) edge node[right] {} (B\x);
        \node (A\x) at (0,3) [vertex,label=below:] {};
        \path [-](B\x) edge node[right] {} (A\x);
      }
      \foreach \x in {0}
      { \node(C\x) at (0,1) [vertex,draw=black,fill=none,label=below:] {};
        \node (B\x) at (-1,1+\x*.5) [vertex,label=below:] {};
        \path [bend left=7](C\x) edge node[right] {} (B\x);
        \path [bend right=7](C\x) edge node[right] {} (B\x);
        \node (A\x) at (-2,1) [vertex,label=below:] {};
        \path [-](B\x) edge node[right] {} (A\x);
      }
      \foreach \x in {0}
      { \node(C\x) at (0,1) [vertex,draw=black,fill=none,label=below:] {};
        \node (B\x) at (1,1+\x*.5) [vertex,label=below:] {};
        \path [bend left=7](C\x) edge node[right] {} (B\x);
        \path [bend right=7](C\x) edge node[right] {} (B\x);
        \node (A\x) at (2,1) [vertex,label=below:] {};
        \path [-](B\x) edge node[right] {} (A\x);
      }
      \foreach \x in {-2,-1,0,1,2}
      { \node(C\x) at (0,1) [vertex,draw=black,fill=none,label=below:] {};
        \node (B\x) at (\x*.5,0) [vertex,label=below:] {};
        \path [bend left=7](C\x) edge node[right] {} (B\x);
        \path [bend right=7](C\x) edge node[right] {} (B\x);
        \node (A\x) at (0,-1) [vertex,label=below:] {};
        \path [-](B\x) edge node[right] {} (A\x);
      }
      \draw[] (1,.-1.5)node[left] {$\wGr F_{A_1,v_1}^\theta$};
    \end{tikzpicture}\quad \quad
    \begin{tikzpicture}[auto,
      vertex/.style={circle,draw=black!100,fill=black!100, 
        inner sep=0pt,minimum size=1mm},scale=.75]
      \foreach \x in {1,-1}
      { \node(C\x) at (0,1) [vertex,draw=black,fill=none,label=below:] {};
        \node (B\x) at (\x*.5,2) [vertex,label=below:] {};
        \path [bend left=7](C\x) edge node[right] {} (B\x);
        \path [bend right=7](C\x) edge node[right] {} (B\x);
        \node (A\x) at (0,3) [vertex,label=below:] {};
        \path [-](B\x) edge node[right] {} (A\x);
      }
      \foreach \x in {0}
      { \node(C\x) at (0,1) [vertex,draw=black,fill=none,label=below:] {};
        \node (B\x) at (-1,1+\x*.5) [vertex,label=below:] {};
        \path [bend left=7](C\x) edge node[right] {} (B\x);
        \path [bend right=7](C\x) edge node[right] {} (B\x);
        \node (A\x) at (-2,1) [vertex,label=below:] {};
        \path [-](B\x) edge node[right] {} (A\x);
      }
      \foreach \x in {0}
      { \node(C\x) at (0,1) [vertex,draw=black,fill=none,label=below:] {};
        \node (B\x) at (1,1+\x*.5) [vertex,label=below:] {};
        \path [bend left=7](C\x) edge node[right] {} (B\x);
        \path [bend right=7](C\x) edge node[right] {} (B\x);
        \node (A\x) at (2,1) [vertex,label=below:] {};
        \path [-](B\x) edge node[right] {} (A\x);
      }
      \foreach \x in {-2,-1,1,2}
      { \node(C\x) at (0,1) [vertex,draw=black,fill=none,label=below:] {};
        \node (B\x) at (\x*.5,0) [vertex,label=below:] {};
        \path [bend left=7](C\x) edge node[right] {} (B\x);
        \path [bend right=7](C\x) edge node[right] {} (B\x);
        \node (A\x) at (0,-1) [vertex,label=below:] {};
        \path [-](B\x) edge node[right] {} (A\x);
      }
      \draw[] (1,.-1.5)node[left] {$\wGr F_{A_2,v_1}^\theta$};
    \end{tikzpicture}
    \quad \quad
    \begin{tikzpicture}[auto,
      vertex/.style={circle,draw=black!100,fill=black!100, 
        inner sep=0pt,minimum size=1mm},scale=.75]
      \foreach \x in {1,0,-1}
      { \node(C\x) at (0,1) [vertex,draw=black,fill=none,label=below:] {};
        \node (B\x) at (\x*.5,2) [vertex,label=below:] {};
        \path [bend left=7](C\x) edge node[right] {} (B\x);
        \path [bend right=7](C\x) edge node[right] {} (B\x);
        \node (A\x) at (0,3) [vertex,label=below:] {};
        \path [-](B\x) edge node[right] {} (A\x);
      }
      \foreach \x in {0}
      { \node(C\x) at (0,1) [vertex,draw=black,fill=none,label=below:] {};
        \node (B\x) at (-1,1+\x*.5) [vertex,label=below:] {};
        \path [bend left=7](C\x) edge node[right] {} (B\x);
        \path [bend right=7](C\x) edge node[right] {} (B\x);
        \node (A\x) at (-2,1) [vertex,label=below:] {};
        \path [-](B\x) edge node[right] {} (A\x);
      }
      \foreach \x in {0}
      { \node(C\x) at (0,1) [vertex,draw=black,fill=none,label=below:] {};
        \node (B\x) at (1,1+\x*.5) [vertex,label=below:] {};
        \path [bend left=7](C\x) edge node[right] {} (B\x);
        \path [bend right=7](C\x) edge node[right] {} (B\x);
        \node (A\x) at (2,1) [vertex,label=below:] {};
        \path [-](B\x) edge node[right] {} (A\x);
      }
      \foreach \x in {-1,0,1}
      { \node(C\x) at (0,1) [vertex,draw=black,fill=none,label=below:] {};
        \node (B\x) at (\x*.5,0) [vertex,label=below:] {};
        \path [bend left=7](C\x) edge node[right] {} (B\x);
        \path [bend right=7](C\x) edge node[right] {} (B\x);
        \node (A\x) at (0,-1) [vertex,label=below:] {};
        \path [-](B\x) edge node[right] {} (A\x);
      }
      \draw[] (1,.-1.5)node[left] {$\wGr F_{A_3,v_1}^\theta$};
    \end{tikzpicture}\\	
    \begin{tikzpicture}[auto,
      vertex/.style={circle,draw=black!100,fill=black!100, 
        inner sep=0pt,minimum size=1mm},scale=.75]
      \foreach \x in {0}
      { \node(C\x) at (0,1) [vertex,draw=black,fill=none,label=below:] {};
        \node (B\x) at (\x*.5,2) [vertex,label=below:] {};
        \path [bend left=7](C\x) edge node[right] {} (B\x);
        \path [bend right=7](C\x) edge node[right] {} (B\x);
        \node (A\x) at (0,3) [vertex,label=below:] {};
        \path [-](B\x) edge node[right] {} (A\x);
      }
      \foreach \x in {1,-1}
      { \node(C\x) at (0,1) [vertex,draw=black,fill=none,label=below:] {};
        \node (B\x) at (-1,1+\x*.5) [vertex,label=below:] {};
        \path [bend left=7](C\x) edge node[right] {} (B\x);
        \path [bend right=7](C\x) edge node[right] {} (B\x);
        \node (A\x) at (-2,1) [vertex,label=below:] {};
        \path [-](B\x) edge node[right] {} (A\x);
      }
      \foreach \x in {1,-1}
      { \node(C\x) at (0,1) [vertex,draw=black,fill=none,label=below:] {};
        \node (B\x) at (1,1+\x*.5) [vertex,label=below:] {};
        \path [bend left=7](C\x) edge node[right] {} (B\x);
        \path [bend right=7](C\x) edge node[right] {} (B\x);
        \node (A\x) at (2,1) [vertex,label=below:] {};
        \path [-](B\x) edge node[right] {} (A\x);
      }
      \foreach \x in {-1,0,1}
      { \node(C\x) at (0,1) [vertex,draw=black,fill=none,label=below:] {};
        \node (B\x) at (\x*.5,0) [vertex,label=below:] {};
        \path [bend left=7](C\x) edge node[right] {} (B\x);
        \path [bend right=7](C\x) edge node[right] {} (B\x);
        \node (A\x) at (0,-1) [vertex,label=below:] {};
        \path [-](B\x) edge node[right] {} (A\x);
      }
      \draw[] (1,.-1.5)node[left] {$\wGr F_{A_4,v_1}^\theta$};
    \end{tikzpicture}\quad \quad
    \begin{tikzpicture}[auto,
      vertex/.style={circle,draw=black!100,fill=black!100, 
        inner sep=0pt,minimum size=1mm},scale=.75]
      \foreach \x in {1,-1}
      { \node(C\x) at (0,1) [vertex,draw=black,fill=none,label=below:] {};
        \node (B\x) at (\x*.5,2) [vertex,label=below:] {};
        \path [bend left=7](C\x) edge node[right] {} (B\x);
        \path [bend right=7](C\x) edge node[right] {} (B\x);
        \node (A\x) at (0,3) [vertex,label=below:] {};
        \path [-](B\x) edge node[right] {} (A\x);
      }
      \foreach \x in {1,-1}
      { \node(C\x) at (0,1) [vertex,draw=black,fill=none,label=below:] {};
        \node (B\x) at (-1,1+\x*.5) [vertex,label=below:] {};
        \path [bend left=7](C\x) edge node[right] {} (B\x);
        \path [bend right=7](C\x) edge node[right] {} (B\x);
        \node (A\x) at (-2,1) [vertex,label=below:] {};
        \path [-](B\x) edge node[right] {} (A\x);
      }
      \foreach \x in {1,-1}
      { \node(C\x) at (0,1) [vertex,draw=black,fill=none,label=below:] {};
        \node (B\x) at (1,1+\x*.5) [vertex,label=below:] {};
        \path [bend left=7](C\x) edge node[right] {} (B\x);
        \path [bend right=7](C\x) edge node[right] {} (B\x);
        \node (A\x) at (2,1) [vertex,label=below:] {};
        \path [-](B\x) edge node[right] {} (A\x);
      }
      \foreach \x in {1,-1}
      { \node(C\x) at (0,1) [vertex,draw=black,fill=none,label=below:] {};
        \node (B\x) at (\x*.5,0) [vertex,label=below:] {};
        \path [bend left=7](C\x) edge node[right] {} (B\x);
        \path [bend right=7](C\x) edge node[right] {} (B\x);
        \node (A\x) at (0,-1) [vertex,label=below:] {};
        \path [-](B\x) edge node[right] {} (A\x);
      }
      \draw[] (1,.-1.5)node[left] {$\wGr F_{A_5,v_1}^\theta$};
    \end{tikzpicture}

  } \caption{The graphs $\wGr F_{A_q,v_1}^\theta$ defined by the
    $4$-partitions $A_1=\lMset1,1,1,5\rMset$,
    $A_2=\lMset1,1,2,4\rMset$, $A_3=\lMset1,1,3,3\rMset$,
    $A_4=\lMset1,2,2,3\rMset$ and $A_5=\lMset2,2,2,2\rMset$ of $r=8$.
    All graphs have $r+s+1=13$ vertices ($s=4$) and $3r=24$ edges.
    There are five different $4$-partitions of $8$ (all listed
    above).}
  \label{fig:ExampleCospectral3}
\end{figure}
An explicit computation of the eigenvalues of the Laplacian of
$\wGr F_{A_q,v_1}^\theta$ shows again that the five graphs are isospectral
for $q\in\{1,\dots,5\}$ and with spectrum
\begin{equation*}
  \wGr F_{A_q,v_1}^\theta
  = \Bigl\lMset
  1-\frac13\sqrt{2+\cos\theta},
  \Bigl(1-\frac 13 \sqrt 3\Bigr)^{(3)},
  1^{(5)},
  \Bigl(1+\frac 13 \sqrt 3\Bigr)^{(3)},
  1+\frac13\sqrt{2+\cos\theta},
  \Bigr\rMset
\end{equation*}
in accordance with~\eqref{eq:spec.contr.frame.union} for $r=8$ and
$s=4$. Note that isospectrality is guaranteed for \emph{any} constant value of the
magnetic potential on one of the multiple edges.
Finally, all graphs $\wGr F_{A_q,v_1}^\theta$
in~\Fig{ExampleCospectral3} are mutually non-isomorphic, as their
corresponding degree multisets (see~\eqref{eq:deg.set}) contain the
corresponding partition $A_q$ (highlighted as bold numbers), and hence
are different.  In fact, the degree sequences of the graphs
$\wGr F_{A_q,v_1}^\theta$ are given by
\begin{center}
  \begin{tabular}{|l|l|}
    \hline
    Graph	&  Degree multiset\\
    \hline
    $\wGr F_{A_1,v_1}^\theta$
             &  $\lMset\bm{1},\bm{1},\bm{1},3^{(8)},\bm{5},16\rMset$\\
    \hline
    $\wGr F_{A_2,v_1}^\theta$
             &  $\lMset\bm{1},\bm{1},\bm{2},3^{(8)},\bm{4},16\rMset$\\
    \hline
    $\wGr F_{A_3,v_1}^\theta$
             &  $\lMset\bm{1},\bm{1},\bm{3},\bm{3},3^{(8)},16\rMset$\\
    \hline
    $\wGr F_{A_4,v_1}^\theta$
             &  $\lMset\bm{1},\bm{2},\bm{2},\bm{3},3^{(8)},16\rMset$\\
    \hline
    $\wGr F_{A_5,v_1}^\theta$
             &  $\lMset\bm{2},\bm{2},\bm{2},\bm{2},3^{(8)},16\rMset$.\\
    \hline
  \end{tabular}
\end{center}

%
\section{Frames constructed from a building block and their spectra}
\label{sec:spectra-of-frames}
%
We construct formally in this section a family of magnetic graphs, called a
\emph{frame}, constructed from a given magnetic graph as a building
block and a subset of its vertex set along which we identify copies of
the given graph.  In particular, the spectrum of this family depends
only on the given magnetic graph and the subset of vertices.

Given a magnetic graph $\wG$ and a subset $V_0$ of its vertices, we
define a geometrical construction of graphs with a symmetric
structure:
\begin{definition}[frame members and frames]
  \label{def:frame}
  Let $\wG=(\G,\alpha,w)$ be a magnetic graph (which we call \emph{building block}), $V_0\subset V(\G)$ and $a \in \N$.
  \begin{enumerate}
  \item
    \label{frame.a}
    Define the \emph{$a$-th frame member} obtained from $\G$
    identified along $V_0$ by
    \begin{equation*}
      \Gr F_a =\Gr F_a(\G,V_0):=\quotient[V_0]{\G^a},
      \qquadtext{where}
      \G^a := \bigdcup_{j \in \{1,\dots,a\}} (\G \times \{j\})
    \end{equation*}
    is the disjoint union of $a$ copies of $\G$ with vertex set
    $V(\G^a)=V(\G)\times \{1,\dots,a\}$ and edge set
    $E(\G^a)=E(\G)\times \{1,\dots,a\}$.  The equivalence relation
    $\sim_{V_0}$ on the vertex set $V(\G^a)$ of the disjoint union is
    given by
    \begin{equation*}
      (v,i) \sim (v',j) \qquadtext{if and only if}
      v=v' \quadtext{and} v \in V_0,
    \end{equation*}
    and all other pairs are equivalent only to itself, i.e., we contract
    each vertex $v \in V_0$ of all copies to one vertex. Denote by
    $[(v,j)]$ the corresponding class specifying a vertex in $F_a$.
    Moreover, we call
    \begin{equation*}
      \wGr F_a(\wG,V_0):=(\Gr F_a(\G,V_0), \alpha,w)
    \end{equation*}
    the \emph{$a$-th ($\wG$-)frame member} obtained from the magnetic graph
    $\wG$ identified along $V_0$, where the weights and vector
    potentials are the same on each copy (and also denoted by the same
    symbol), i.e., $w_{(e,j)}:=w_e$ and $\alpha_{(e,j)}:=\alpha_e$
    for all $e \in E(\G)$ and $j \in \{1,\dots,a\}$.

  \item
    \label{frame.b}
    The family $(\wGr F_a(\wG, V_0))_{a \in \N}$ is called a
    \emph{($\wG$-)frame} identified along $V_0$.
  \end{enumerate}
  If the dependence on $\G$ and $V_0$ are clear from the context we
  will denote a frame member and a frame simply by $\wGr F_a$ and
  $(\wGr F_a)_{a \in \N}$, respectively.
\end{definition}

\begin{remark}
  \label{rem:simple.frames}
  It is easily seen that for a simple graph $\wG$ and $a \ge 2$, the
  frame $\wGr F_a(\wG,V_0)$ is simple if and only if there are no
  edges inside the identified vertex set $V_0$.
\end{remark}

\begin{remark}[frames with virtualised vertices]
  \label{rem:frame.virt}
  We also allow that $\wG$ has \emph{virtualised} vertices
  $V_1 \subset V(\wG)$, i.e., vertices on which we impose Dirichlet
  boundary conditions (see \Def{dir.lapl}).  We write
  $\DirGr {(\wGr F_a)}{V_1} :=\wGr F_a(\DirGr \wG {V_1},V_0)$ for the
  $a$-th $\DirGr \wG{V_1}$-frame member identified along $V_0$.  Note
  that the case $V_1=\emptyset$ is just the case described in
  \Def{frame}.
\end{remark}
The following result is a direct result of the preceding definitions.
\begin{lemma}[order and degree multiset of frame members]
  \label{lem:frame.order}
  Let $\DirGr \wG {V_1}$ be a magnetic graph with virtualised vertices
  $V_1 \subset V(\G)$.  Moreover, let $V_0$ be a subset of vertices
  such that $V_1 \subset V_0 \subset V$.  Then the order and the
  number of edges of the $a$-th frame member
  $\DirGr {(\Gr F_a)}{V_1}:=\Gr F_a(\DirGr \G {V_1} ,V_0)$ are given
  respectively by
  \begin{align*}
    \card{\DirGr {(\Gr F_a)}{V_1}}
       =a(\card \G - \card {V_0}) + \card {V_0} - \card {V_1}
      \qquadtext{and}
    \card{E(F_a)}=\card{E(\DirGr {(\Gr F_a)}{V_1})} &=  a \card{E(\G)}.
  \end{align*}
  Moreover, the degree multiset of $\DirGr {(\Gr F_a)}{V_1}$ is
  given by
  \begin{equation*}
    \deg \bigl(\DirGr {(\Gr F_a)}{V_1}\bigr)
    = \biguplus_{v\in V(\G)\setminus V_0} \lMset \deg_\G v\rMset ^{(a)}
    \uplus
     \biguplus_{v_0\in V_0 \setminus V_1}
    \lMset a \deg_\G v_0\rMset.
  \end{equation*}
\end{lemma}

We prove next, that the spectrum of the magnetic graph
$\wGr F_a(\DirGr \wG {V_1},V_0)$ with building block
$\DirGr \wG {V_1}$ identified along $V_0 \subset V(\G)$ depends on the
spectrum of $\DirGr \wG {V_1}$ and $\DirGr \wG {V_0}$ only.
Basically, we use the cyclic symmetry of the $a$-th frame member.
\begin{proposition}[spectrum of frame members]
  \label{prp:frame.spec}
  Let $\DirGr \wG {V_1}$ be a magnetic graph with virtualised vertices
  $V_1 \subset V(\G)$.  Moreover, let $V_0$ be a subset of vertices
  such that $V_1 \subset V_0 \subset V$.  Then the spectrum of the
  $a$-th $\DirGr \wG {V_1}$-frame member
  $\wGr F_a(\DirGr \wG {V_1},V_0)$ identified along $V_0$ ($a\in\N$)
  is given by
  \begin{equation*}
    \bigspec{\wGr F_a(\DirGr \wG {V_1} ,V_0)}
    =\bigspec{\DirGr \wG {V_1}} \uplus \bigspec{\DirGr \wG {V_0}}^{(a-1)}.
  \end{equation*}
\end{proposition}
\begin{proof}
  We assume for ease of notation that $V_1=\emptyset$, i.e.,
  $\DirGr \wG {V_1}=\wG$, the general case can be treated exactly in the
  same way.

  Let $n=\card{\wG}=\card{V}$ and
  $\spec{\wG}=\lMset\rho_1,\rho_2,\dots,\rho_n\rMset$ be eigenvalues
  of $\lapl {\wG}$ with corresponding orthonormal eigenfunctions given
  by $\{f_1,f_2,\dots,f_n\}$.  For each $k \in \{1,\dots, n\}$, we
  extend first the eigenfunction $f_k$ onto the disjoint union $\wG^a$
  to a function $ \map{\wt f_k}{V(\wG^a)}{\C}$ by
  $\wt f_k(v,j)=f_k(v)$ for $j \in \{1,\dots,a\}$.  With a little
  abuse of notation we write by $\wt f_k$ also the corresponding
  function on the $a$-th frame member $\wGr F_a=\wGr F_a(\wG ,V_0)$,
  which is naturally defined since the functions have the same value
  on the contracted vertices $V_0$. The proof of the multiset
  inclusion $\bigspec{\wG}\subset\bigspec{\wGr F_a(\wG ,V_0)}$ is then
  a special case of the first part of the proof in
  \Thm{spec.contr.frame.union} (choose $A=\lMset a\rMset$ in the
  construction of the symmetric functions).

  To show that also the eigenvalues of $\DirGr \wG {V_0} $ with
  Dirichlet conditions on $V_0$ lift to eigenvalues on the frame
  member $\wGr{F_a}$ put
  $n' = \card{\DirGr \wG {V_0}} = n-\card{V_0}$. Let
  $\lambda_1,\lambda_2,\dots,\lambda_{n'}$ be the eigenvalues of
  $\DirGr \wG {V_0}$ with corresponding orthonormal eigenfunctions
  $h_1,h_2,\dots,h_{n'}$. Moreover, let $\eta \in \C$ be a
  non-trivial $a$-th root of unity, i.e., $\eta\ne1$ and $\eta^a=1$.
  For each such $\eta$ we extend the eigenfunction $h_{k'}$ to an
  eigenfunction $\wt h_{k',\eta}$ onto $\wGr F_a$ by
  \begin{equation*}
    \wt h_{k',\eta}([(v,j)])=
    \eta^j h_{k'}(v).
  \end{equation*}
  Note that this choice is also well-defined on the quotient as
  $h_{k'}(v)=0$ for $v \in V_0$.  We now show that $\wt h_{k',\eta}$
  is an eigenfunction of $\lapl {\wGr F_a}$ with eigenvalue
  $\lambda_{k'}$: If $v \in V_0$, we have $h_{k'}(v)=0$ and
  \begin{align*}
    \bigl(\lapl {\wGr F_a}\wt h_{k',\eta}\bigr)[(v,j)]
    &= 0 - \frac1{\deg_{\G^a}^w[(v,j)]}
      \sum_{e\in E_{[(v,j)]}(\Gr F_a)}
      w_e \e^{\im \alpha_e} \wt h_{k',\eta}(\bd_+e)\\
    &= 0 - \frac 1 {\deg_{\G}^w v}
      \Bigl(\sum_{j=1}^a \eta^j\Bigr)
      \sum_{e\in E_v(\G)} w_e \e^{\im \alpha_e}h_{k'}(\bd_+e)
      =0
      = \lambda_{k'} \wt h_{k',\eta}[(v,j)]
  \end{align*}
  as $\sum_{j=1}^a \eta^j=0$ since $\eta \ne 1$.

  If $v \in V\setminus V_0$, all equivalence classes
  $[(v,j)]$ contain only one element, again denoted by $(v,j)$.  In
  this case, we have
  \begin{align*}
    \bigl(\lapl {\wGr F_a}\wt h_{k',\eta}\bigr)(v,j)
    &= \wt h_{k',\eta}(v,j)
      - \frac1{\deg_{\G^a}^w[(v,j)]}
      \sum_{e\in E_{[(v,j)]}(\Gr F_a)}
      w_e \e^{\im \alpha_e} \wt h_{k',\eta}(\bd_+e)\\
    &= \eta^j\Bigl(h_{k'}(v) - \frac 1 {\deg_\G^w v}
      \sum_{e\in E_v(\G)} w_e \e^{\im \alpha_e}h_{k'}(\bd_+e)\Bigr)\\
    &= \eta^j \lambda_{k'} h_{k'}(v)
      = \lambda_{k'} \wt h_{k',\eta}(v,j).
  \end{align*}
  This shows that $\wt h_{k',\eta}$ is an
  eigenfunction of $\lapl {\wGr F_a}$ with eigenvalue $\lambda_{k'}$ as claimed.

  Next we will show that the constructed eigenfunctions on the frame
  member are mutually orthogonal. First, recall that
  $\sum_{j=1}^a \eta^j=0$ since $\eta \ne 1$.  Therefore, we have
  \begin{equation*}
    \bigiprod[\lsqr{V(\wGr F_a),\deg_{\G^a}^w}]{\wt f_k}{\wt h_{k',\eta}}
    = \Bigl(\sum_{j=1}^a \conj \eta^j\Bigr)
    \bigiprod[\lsqr{V(\wG),\deg_\G^w}]{f_k}{h_{k'}}
    =0.
  \end{equation*}
  Second, the eigenfunctions corresponding to different non-trivial
  roots of unity are also orthogonal since
  \begin{equation*}
    \bigiprod[\lsqr{V(\wGr F_a),\deg_{\G^a}^w}]{\wt h_{k',\eta}}{\wt h_{k',\eta'}}
    = \sum_{j=1}^a \bigiprod[\lsqr{V(\wG),\deg_\G^w}]{\eta^j h_{k'}}{(\eta')^j h_{k'}}
    = \sum_{j=1}^a (\eta \conj \eta')^j
    =0 \;
  \end{equation*}
  Here, $\eta\ne1\ne\eta'$ and $\eta \ne \eta'$, so that the product
  $\eta\eta'$ also determines a non-trivial root of unity. Therefore,
  we have shown that the functions $\wt f_k$ and $\wt h_{k',\eta}$ for
  $k \in \{1,\dots,n\}$, $k' \in \{1,\dots,n'\}$ and $\eta^a=1$ with
  $\eta\ne 1$ are mutually orthogonal.  In particular, we have shown
  that
  $\bigspec{\DirGr \wG {V_0}}^{(a-1)}\subset \bigspec{\wGr F_a(\wG
    ,V_0)}$, the multiplicity $a-1$ coming from the fact that there
  are $a-1$ solutions of $\eta^a=1$ with $\eta \ne 1$. Altogether we
  have $n+(a-1) n'$ mutually orthogonal eigenfunctions and by \Lem{frame.order},
  the order of $\Gr F_a$ is precisely
  \begin{equation*}
    \card{\Gr F_a}
    =(a-1) n' + n,
  \end{equation*}
  hence we have found all eigenvalues.
\end{proof}

We next present several examples that illustrate how different choices
of $V_0$ of the same underlying graph $\wG$ lead to different families
of graphs.  We apply the preceding proposition to determine their
spectra. We will mention examples with a tree as building block (hence
the magnetic potential has no effect) and building blocks with cycles
and a non-trivial magnetic potential.

\begin{remark}[group-theoretical justification]
  \label{rem:dualgroup}
  There is an elegant group-theoretical justification for the
  multiplicities appearing in the spectrum of the magnetic Laplacian
  of a frame member.  In fact, in the proof of \Prp{frame.spec} there
  is a cyclic group $\Z_a$ acting on the vertices of $\wGr F_a$ by
  shifting the label $j$, $j \in \Z_a$, numbering the branches of the
  frame member. This induces naturally an action of $\Z_a$ on the
  corresponding $\ell_2$ space which defines the regular
  representation $U$ of $\Z_a$. Since $\Z_a$ is a finite Abelian group
  the unitary dual satisfies $\widehat \Z_a \cong \Z_a$ and $U$
  decomposes into a direct sum of one-dimensional representations
  which consist of multiplication with $a$-th root of unity (see
  e.g.~\cite[Section~23.27]{hewitt-ross-1}).  The arguments in the
  proof of \Prp{frame.spec} essentially show that the lifted $\wt f$-
  and $\wt h$-eigenfunctions not only reduce the Laplacian but also
  reduce the regular representation $U$.  In fact, one has the
  following decomposition of the Laplacian
  \begin{align*}
    \lapl{\wGr F_a} \cong
    \lapl \wG \oplus
    \bigoplus_{p\in \Z_a \setminus\{0\}} \lapl {\DirGr \wG {V_0}}
  \end{align*}
  explaining the multiplicity $(a-1)$ of the eigenvalues of
  $\lapl {\DirGr \wG {V_0}}$ in the spectrum of the Laplacian on the
  frame $\wGr F_a$.  Note that the first summand arises from the
  trivial representation and recall that the constant function is
  excluded as eigenfunction of the Dirichlet Laplacian due to the
  Dirichlet conditions on $V_0$.
\end{remark}

We start with a simple building block; here a tree, hence any magnetic
potential is cohomologous to $0$.
\begin{example}[complete bipartite graphs]
  \label{ex:complete.bipartite}
  Let $\G=\Gr K_{m,1}$ be the complete bipartite graph with $m+1$
  vertices and $m$ edges (it is also called a \emph{star graph}).  We
  draw $m-1$ of the pendant vertices on the bottom and one vertex on
  top.  Let $V_0$ be the set of vertices of degree $1$ (the bottom and
  top ones).  Again we choose standard weights.  The spectrum of
  $\lapl \G$ and  $\laplDir \G {V_0}$ is given respectively by
  \begin{align*}
    \spec{\Gr K_{m,1}}
    =\lMset 0,1^{(m-1)}, 2 \rMset
    \quadtext{and}
    \bigspec{\DirGr {(\Gr K_{m,1})}{V_0}}
    = \lMset 1 \rMset.
  \end{align*}
  Note that the frame member $\Gr F_a=\Gr F_a(\G,V_0)$ is
  obtained from the complete bipartite graph $\Gr K_{m-1,a}$ by
  decorating each of the $a$ vertices on top with a pendant vertex,
  and that the pendant $a$ vertices are all identified into a single
  vertex.
  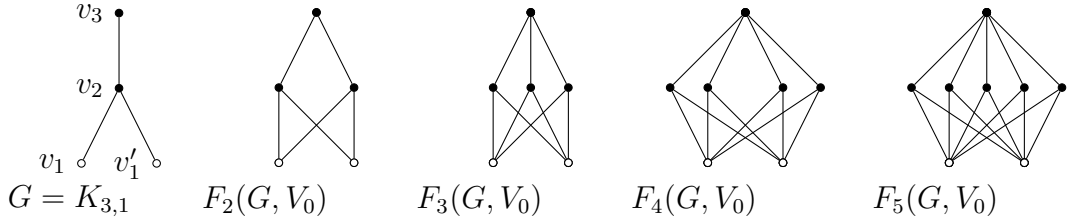
\begin{figure}[h]
    \centering
    {
      \begin{tikzpicture}[auto,vertex/.style={circle,draw=black!100,%
          fill=black!100, 
          inner sep=0pt,minimum size=1mm},scale=1]
        \foreach \x in {0}
        { \node(C\x) at (0,1) [vertex,fill=black,label=left:$v_3$] {};
          \node (B\x) at (\x*.5,0) [vertex,label=left:$v_2$] {};
          \path [-](C\x) edge node[right] {} (B\x);
          \node (A\x) at (-0.5,-1) [vertex,fill=white,label=left:$v_1$] {};
          \node (AA\x) at (0.5,-1) [vertex,fill=white,label=left:$v_1'$] {};
          \path [-](B\x) edge node[right] {} (A\x);
          \path [-](B\x) edge node[right] {} (AA\x);
        }
        \draw[] (.3,.-1.5)node[left] {$\G=\Gr K_{3,1}$};
      \end{tikzpicture}\quad
      \begin{tikzpicture}[auto,vertex/.style={circle,draw=black!100,%
          fill=black!100, 
          inner sep=0pt,minimum size=1mm},scale=1]
        \foreach \x in {-1,1}
        { \node(C\x) at (0,1) [vertex,fill=black,label=below:] {};
          \node (B\x) at (\x*.5,0) [vertex,label=below:] {};
          \path [-](C\x) edge node[right] {} (B\x);
          \node (A\x) at (-0.5,-1) [vertex,fill=white,label=below:] {};
          \node (AA\x) at (0.5,-1) [vertex,fill=white,label=left:] {};
          \path [-](B\x) edge node[right] {} (A\x);
          \path [-](B\x) edge node[right] {} (AA\x);
        }
        \draw[] (.3,.-1.5)node[left] {$\Gr F_2(\G,V_0)$};
      \end{tikzpicture}\quad
      \begin{tikzpicture}[auto,vertex/.style={circle,draw=black!100,%
          fill=black!100, 
          inner sep=0pt,minimum size=1mm},scale=1]
        \foreach \x in {-1,0,1}
        { \node(C\x) at (0,1) [vertex,fill=black,label=below:] {};
          \node (B\x) at (\x*.5,0) [vertex,label=below:] {};
          \path [-](C\x) edge node[right] {} (B\x);
          \node (A\x) at (-0.5,-1) [vertex,fill=white,label=below:] {};
          \node (AA\x) at (0.5,-1) [vertex,fill=white,label=left:] {};
          \path [-](B\x) edge node[right] {} (A\x);
          \path [-](B\x) edge node[right] {} (AA\x);
        }
        \draw[] (.3,.-1.5)node[left]  {$\Gr F_3(\G,V_0)$};
      \end{tikzpicture}\quad
      \begin{tikzpicture}[auto,vertex/.style={circle,draw=black!100,%
          fill=black!100, 
          inner sep=0pt,minimum size=1mm},scale=1]
        \foreach \x in {-2,-1,1,2}
        { \node(C\x) at (0,1) [vertex,fill=black,label=below:] {};
          \node (B\x) at (\x*.5,0) [vertex,label=below:] {};
          \path [-](C\x) edge node[right] {} (B\x);
          \node (A\x) at (-0.5,-1) [vertex,fill=white,label=below:] {};
          \node (AA\x) at (0.5,-1) [vertex,fill=white,label=left:] {};
          \path [-](B\x) edge node[right] {} (A\x);
          \path [-](B\x) edge node[right] {} (AA\x);
        }
        \draw[] (.3,.-1.5)node[left]  {$\Gr F_4(\G,V_0)$};
      \end{tikzpicture}\quad
      \begin{tikzpicture}[auto,vertex/.style={circle,draw=black!100,%
          fill=black!100, 
          inner sep=0pt,minimum size=1mm},scale=1]
        \foreach \x in {-2,-1,0,1,2}
        { \node(C\x) at (0,1) [vertex,fill=black,label=below:] {};
          \node (B\x) at (\x*.5,0) [vertex,label=below:] {};
          \path [-](C\x) edge node[right] {} (B\x);
          \node (A\x) at (-0.5,-1) [vertex,fill=white,label=below:] {};
          \node (AA\x) at (0.5,-1) [vertex,fill=white,label=left:] {};
          \path [-](B\x) edge node[right] {} (A\x);
          \path [-](B\x) edge node[right] {} (AA\x);
        }
        \draw[] (.3,.-1.5)node[left]  {$\Gr F_5(\G,V_0)$};
      \end{tikzpicture}
    }
    \caption{The frame members leading to \emph{complete bipartite graphs}, here for $m=3$.
    \label{fig:complete.bipartite}
  }
  \end{figure}
  \Prp{frame.spec} implies that the spectrum of the standard
  Laplacian on $\Gr F_a$ is given by
  $\spec{\Gr F_a}=\spec {\Gr K_{m,1}} \uplus \spec{\DirGr {(\Gr
      K_{m,1})} {V_0}}^{(a-1)}$, i.e.,
  \begin{align*}
    \bigspec{\Gr F_a(\Gr K_{m,1}, V_0)}
    =\lMset 0,1^{(m+a-2)}, 2 \rMset.
  \end{align*}
  Actually, $\Gr F_a$ is just the complete bipartite graph $K_{m+1,a}$
  if one draws the top vertex also on the bottom.  Nevertheless, we
  draw $\Gr F_a$ in this way as we will use only the $m-1$ bottom
  vertices (outlined) as distinguished ones in
  \Sec{frames-and-partitions}.
\end{example}

\begin{example}[a diamond-like graph with magnetic potential]
  \label{ex:diamond.mag}
  We start with a path graph with three vertices and add a loop with
  magnetic potential to the middle vertex $v_2$.  For
  $\theta \in \RmodZ$, we let $\wG^\theta := (\G,\alpha^\theta,1)$ be
  the magnetic graph with standard weights $(w_e=1$ for all
  $e \in E(\G)$), where $\G$ is the path graph $\Gr P_3$ with a loop
  $e_0$ attached at the middle vertex $v_2$.  We set
  $\alpha_{e_0}^\theta=\theta$ and $\alpha_e^\theta=0$ for the other
  two edges.  In this case, the spectrum of $\wG^\theta$ is given by
  the eigenvalues of
  \begin{equation*}
    \lapl {\wG^\theta}
    \cong
    \begin{pmatrix}
      1    & -1/2 & 0\\
      -1/2 & 1-(\cos \theta)/2 & -1/2\\
      0    & -1/2 & 1
    \end{pmatrix},
  \end{equation*}
  namely we have
  \begin{equation*}
    \spec{\wG^\theta}
    =\Bigl\lMset
    1-\frac{\cos \theta}4 - \frac{\sqrt{\cos^2\theta + 8}}4;
    1;
    1-\frac{\cos \theta}4 + \frac{\sqrt{\cos^2\theta + 8}}4
    \Bigr\rMset.
  \end{equation*}
  Now, we let $V_0=\{v_1,v_3\}$ be the set of the two vertices of
  degree $1$ (see \Fig{diamond-mag})
  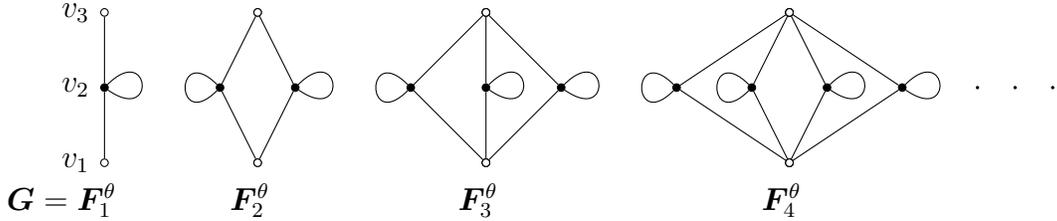
\begin{figure}[h]
    \centering
    {
      \begin{tikzpicture}[auto, vertex/.style={circle,draw=black!100,fill=black!100, 
          inner sep=0pt,minimum size=1mm},scale=1]
        \foreach \x in {0}
        { \node(C\x) at (0,1) [vertex,draw=black,fill=white,label=left:$v_3$]{};
          \node (B\x) at (\x*.5,0) [vertex,label=left:$v_2$] {};
          \draw[-] (B\x) to[min distance=8mm,in=45,out=325,looseness=1]  node[right] {} (B\x);
          \path [-](C\x) edge node[right] {} (B\x);
          \node (A\x) at (0,-1) [vertex,fill=white,label=left:$v_1$] {};
          \path [-](B\x) edge node[right] {} (A\x);
	}
	\draw[] (.3,.-1.5)node[left] {$\wG=\wGr F_1^\theta$};
      \end{tikzpicture}
      \begin{tikzpicture}[auto, vertex/.style={circle,draw=black!100,fill=black!100, 
	  inner sep=0pt,minimum size=1mm},scale=1]
        \foreach \x in {1}
	{ \node(C\x) at (0,1) [vertex,draw=black,fill=white,label=below:] {};
          \node (B\x) at (\x*.5,0) [vertex,label=below:] {};
          \draw[-] (B\x) to[min distance=8mm,in=45,out=325,looseness=1]  node[right] {} (B\x);
          \path [-](C\x) edge node[right] {} (B\x);
          \node (A\x) at (0,-1) [vertex,fill=white,label=below:] {};
          \path [-](B\x) edge node[right] {} (A\x);
	}
	\foreach \x in {-1}
	{ \node(C\x) at (0,1) [vertex,draw=black,fill=white,label=below:] {};
          \node (B\x) at (\x*.5,0) [vertex,label=below:] {};
          \draw[-] (B\x) to[min distance=8mm,in=135,out=225,looseness=1]  node[right] {} (B\x);
          \path [-](C\x) edge node[right] {} (B\x);
          \node (A\x) at (0,-1) [vertex,fill=white,label=below:] {};
          \path [-](B\x) edge node[right] {} (A\x);
	}
	\draw[] (.3,.-1.5)node[left] {$\wGr F_2^\theta$};
      \end{tikzpicture}
      \begin{tikzpicture}[auto, vertex/.style={circle,draw=black!100,fill=black!100, 
          inner sep=0pt,minimum size=1mm},scale=1]
	\foreach \x in {0,2}
	{ \node(C\x) at (0,1) [vertex,draw=black,fill=white,label=below:] {};
          \node (B\x) at (\x*.5,0) [vertex,label=below:] {};
          \draw[-] (B\x) to[min distance=8mm,in=45,out=325,looseness=1]  node[right] {} (B\x);
          \path [-](C\x) edge node[right] {} (B\x);
          \node (A\x) at (0,-1) [vertex,fill=white,label=below:] {};
          \path [-](B\x) edge node[right] {} (A\x);
	}
	\foreach \x in {-2}
	{ \node(C\x) at (0,1) [vertex,draw=black,fill=white,label=below:] {};
          \node (B\x) at (\x*.5,0) [vertex,label=below:] {};
          \draw[-] (B\x) to[min distance=8mm,in=135,out=225,looseness=1]  node[right] {} (B\x);
          \path [-](C\x) edge node[right] {} (B\x);
          \node (A\x) at (0,-1) [vertex,fill=white,label=below:] {};
          \path [-](B\x) edge node[right] {} (A\x);
	}
	\draw[] (.3,.-1.5)node[left] {$\wGr F_3^\theta$};
      \end{tikzpicture}
      \begin{tikzpicture}[auto, vertex/.style={circle,draw=black!100,fill=black!100, 
          inner sep=0pt,minimum size=1mm},scale=1]
	\foreach \x in {1,3}
	{ \node(C\x) at (0,1) [vertex,draw=black,fill=white,label=below:] {};
          \node (B\x) at (\x*.5,0) [vertex,label=below:] {};
          \draw[-] (B\x) to[min distance=8mm,in=45,out=325,looseness=1]  node[right] {} (B\x);
          \path [-](C\x) edge node[right] {} (B\x);
          \node (A\x) at (0,-1) [vertex,fill=white,label=below:] {};
          \path [-](B\x) edge node[right] {} (A\x);
	}
	\foreach \x in {-3,-1}
	{ \node(C\x) at (0,1) [vertex,draw=black,fill=white,label=below:] {};
          \node (B\x) at (\x*.5,0) [vertex,label=below:] {};
          \draw[-] (B\x) to[min distance=8mm,in=135,out=225,looseness=1]  node[right] {} (B\x);
          \path [-](C\x) edge node[right] {} (B\x);
          \node (A\x) at (0,-1) [vertex,fill=white,label=below:] {};
          \path [-](B\x) edge node[right] {} (A\x);
	}
	\draw[] (.3,.-1.5)node[left] {$\wGr F_4^\theta$};
	\node[circle,scale=0.3] (D1) at (2.5,0) [vertex,label=below:] {};
	\node[circle,scale=0.3] (D2) at (3,0) [vertex,label=below:] {};
	\node[circle,scale=0.3] (D3) at (3.5,0) [vertex,label=below:] {};
      \end{tikzpicture}
    }
    \caption{The family of decorated diamond magnetic graphs
      $\wGr F_a^\theta = \wGr F_a(\wG^\theta,V_0)$.
      \label{fig:diamond-mag}}
  \end{figure}
  Here, the spectrum of $\DirGr {(\wG^\theta)} {V_0}$ is given by
  \begin{equation*}
    \bigspec{\DirGr {(\wG^\theta)} {V_0}}
    =\Bigl\lMset
       1 - \frac {\cos \theta}2
    \Bigr\rMset.
  \end{equation*}
  The spectrum of the frame member $\wGr F_a^\theta = \wGr F_a(\wG^\theta,V_0)$ is
  hence given by
  \begin{equation*}
    \spec{\wGr F_a^\theta}
    =\Bigl\lMset
    1-\frac{\cos \theta}4 - \frac{\sqrt{(\cos\theta)^2 + 8}}4,
    1,
    1-\frac{\cos \theta}4 + \frac{\sqrt{(\cos\theta)^2 + 8}}4
    \Bigr\rMset
    \uplus
    \Bigl\lMset
       1 - \frac {\cos \theta}2
       \Bigr\rMset^{(a-1)}.
  \end{equation*}
\end{example}

\begin{example}[kite graphs]
   \label{ex:kite.frame}
   Let $\G$ be the complete graph on four vertices with one pendant
   vertex added (so $\G$ has five vertices).  The vertex set $V_0$
   consists now of the pendant vertex and one of the remaining four vertices
   (see \Fig{kite.frame}).
   Let $\wG$ be the corresponding weighted graph with standard
   weights.  We then have
   \begin{align*}
     \spec{\wG}
     = \Bigl\lMset 0, \frac{7-\sqrt 7}6,\frac 43,\frac 43,\frac{7+\sqrt 7}6
     \Bigr\rMset
     \quadtext{and}
      \spec{\DirGr \wG {V_0}}
     = \Bigl\lMset \frac{5-\sqrt 7}6,\frac 43,\frac{5+ \sqrt 7}6 \Bigr\rMset
   \end{align*}
   for the standard Laplacian without magnetic potential. As building
   block $\wG=(\G,\alpha,1)$ one can also choose an arbitrary magnetic
   potential $\map \alpha E \RmodZ$.  Note that as the Betti number of
   the building block is $3$, it can be seen that it is enough to
   consider magnetic potentials only which are supported on three
   edges only, leading to three parameters in $\RmodZ$; any other
   magnetic potential leads to a unitarily equivalent Laplacian by
   \Prpenum{dml}{dml.b}.
\begin{figure}[h]
	\centering
	{
		\begin{tikzpicture}[auto,vertex/.style={circle,draw=black!100,%
				fill=black!100, 
				inner sep=0pt,minimum size=1mm},scale=1]
			\foreach \x in {0}
			{\node(D) at (0,-1) [vertex,fill=white,label=below:$\G$] {};
			 \node(C\x) at (\x,1) [vertex,fill=black,] {};
				\node (B-1\x) at ({\x-.5},2) [vertex,label=left:] {};
				\node (B1\x) at ({\x+.5},2) [vertex,label=right:] {};
				\node (A) at (0,3) [vertex,fill=black] {};

				\path [-](D) edge node[right] {} (C\x);
				\path [-](C\x) edge node[right] {} (B-1\x);
				\path [-](C\x) edge node[right] {} (B1\x);
				\path [-](B-1\x) edge node[right] {} (B1\x);
				\path [-](B-1\x) edge node[right] {} (A);
				\path [-](B1\x) edge node[right] {} (A);
				\path [-](C\x) edge node[right] {} (A);

			}
		\end{tikzpicture}
		\begin{tikzpicture}[auto,vertex/.style={circle,draw=black!100,%
				fill=black!100, 
				inner sep=0pt,minimum size=1mm},scale=1]
			\foreach \x in {1,-1}
			{\node(D) at (0,-1) [vertex,fill=white,label=below:$\Gr F_2$] {};
				\node(C\x) at (\x,1) [vertex,fill=black,] {};
				\node (B-1\x) at ({\x-.25},2) [vertex,label=left:] {};
				\node (B1\x) at ({\x+.25},2) [vertex,label=right:] {};
				\node (A) at (0,3) %
                                [vertex,fill=black] {};
				\path [-](D) edge node[right] {} (C\x);
				\path [-](C\x) edge node[right] {} (B-1\x);
				\path [-](C\x) edge node[right] {} (B1\x);
				\path [-](B-1\x) edge node[right] {} (B1\x);
				\path [-](B-1\x) edge node[right] {} (A);
				\path [-](B1\x) edge node[right] {} (A);
				\path [-](C\x) edge node[right] {} (A);
			}
		\end{tikzpicture}
	\begin{tikzpicture}[auto,vertex/.style={circle,draw=black!100,%
			fill=black!100, 
			inner sep=0pt,minimum size=1mm},scale=1]
		\foreach \x in {1,0,-1}
		{\node(D) at (0,-1) [vertex,fill=white,label=below:$\Gr F_3$] {};
			\node(C\x) at (\x,1) [vertex,fill=black,] {};
			\node (B-1\x) at ({\x-.25},2) [vertex,label=left:] {};
			\node (B1\x) at ({\x+.25},2) [vertex,label=right:] {};
			\node (A) at (0,3) [vertex,fill=black] {};

			\path [-](D) edge node[right] {} (C\x);
			\path [-](C\x) edge node[right] {} (B-1\x);
			\path [-](C\x) edge node[right] {} (B1\x);
			\path [-](B-1\x) edge node[right] {} (B1\x);
			\path [-](B-1\x) edge node[right] {} (A);
			\path [-](B1\x) edge node[right] {} (A);
			\path [-](C\x) edge node[right] {} (A);

		}
	\end{tikzpicture}
	\begin{tikzpicture}[auto,vertex/.style={circle,draw=black!100,%
			fill=black!100, 
			inner sep=0pt,minimum size=1mm},scale=1]
		\foreach \x in {2,1,-1,-2}
		{\node(D) at (0,-1) [vertex,fill=white,label=below:$\Gr F_4$] {};
			\node(C\x) at (\x,1) [vertex,fill=black,] {};
			\node (B-1\x) at ({\x-.25},2) [vertex,label=left:] {};
			\node (B1\x) at ({\x+.25},2) [vertex,label=right:] {};
			\node (A) at (0,3) [vertex,fill=black] {};
			\path [-](D) edge node[right] {} (C\x);
			\path [-](C\x) edge node[right] {} (B-1\x);
			\path [-](C\x) edge node[right] {} (B1\x);
			\path [-](B-1\x) edge node[right] {} (B1\x);
			\path [-](B-1\x) edge node[right] {} (A);
			\path [-](B1\x) edge node[right] {} (A);
			\path [-](C\x) edge node[right] {} (A);

		}
	\end{tikzpicture}
}
    \caption{The frames starting from a so-called \emph{kite} graph.
    \label{fig:kite.frame}}
\end{figure}
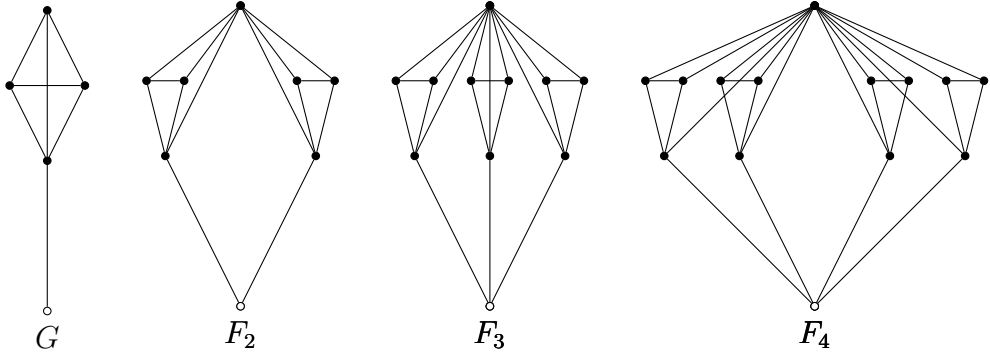
\end{example}


%

\section{Graphs constructed from frames and partitions}
\label{sec:frames-and-partitions}
In this section we present the construction of an infinite collection
of families of magnetic graphs, where all the elements in each family
are isospectral, but non-isomorphic graphs for the magnetic Laplacian
with normalised weights.  Each family consists of finitely many
graphs. Given a frame $(\Gr F_a)_{a \in \N}$ constructed as in the
previous section we will determine these families by assembling the
members of the frame according to different $s$-partitions of the
natural number $r$ (see \Def{partition} and \Sec{motivating-ex}). We
present the construction in two steps.

\subsection{Disjoint frame unions}
\label{subsec:disj.frame.union}

We start with a simple construction of isospectral, but non-isomorphic
graphs.  Note that these graphs are not connected.  The construction
in this subsection is a special case of the one in
\Subsec{contr.frame.union}, namely when $V_1=\emptyset$.

\begin{definition}[disjoint frame unions]
  \label{def:frame.union}
  Let $\wG=(\G,\alpha,w)$ be a magnetic graph (building block) and
  choose a subset $V_0\subset V(\G)$.  Consider the frame
  $(\wGr F_a(\wG,V_0))_{a\in \N}=(\wGr F_a)_{a\in \N}$ as specified in
  \Def{frame} and let $A=\lMset a_1,\dots,a_s\rMset$ be an
  $s$-partition of the natural number $r$.  The \emph{disjoint
    $A$-union of the frame $(\wGr F_a)_{a\in \N}$} is defined as
  \begin{equation}
    \label{eq:frame.union}
    \wGr F_A
    := \wGr F_A(\wG ,V_0)
    :=\bigdcup_{i=1}^s \wGr F_{a_i} \times \{i\}.
  \end{equation}
\end{definition}
Note that disjoint frame unions associated to a partition $A$ can be
similarly defined for building block graphs that have Dirichlet
conditions on some vertices, say $V_0$.  We denote the corresponding
disjoint frame union by $\wGr F_A (\DirGr \wG {V_0},V_0)$.

In fact, in the proof of \Thm{spec.contr.frame.union} we will consider
disjoint $A$-unions of frames constructed from building block graphs
that have Dirichlet conditions on $V_0$.  Note that
$\wGr F_A (\DirGr \wG {V_0},V_0)$ is actually also the disjoint union
of $r$ copies of $\DirGr \wG {V_0}$, but we need the grouping into the
partition members $a \in A$ in the proof of
\Thm{spec.contr.frame.union}.

We need the order and degree multiset of the disjoint union of frames.
\begin{lemma}[order and degree multiset of disjoint frame unions]
  \label{lem:order-fa}
  Let $A$ be an $s$-partition of $r$ and
  consider the disjoint $A$-union of frames $\Gr F_A$ defined before. Its order is given by
  \begin{align*}
    \card{\Gr F_A}
    =\sum_{a \in A} \card{\Gr F_a}
    =r(\card{\G}-\card{V_0})+ s \card {V_0}.
  \end{align*}
  The degree multiset of $\Gr F_A$ is given by
  \begin{equation*}
    \deg \Gr F_A
    =\biguplus_{a\in A} \deg \Gr F_a
    =\biguplus_{v\in V(\G)\setminus V_0} \lMset \deg_\G v\rMset ^{(r)}
    \uplus
    \biguplus_{a\in A}\Bigl(\biguplus_{v_0\in V_0} \lMset a \deg_\G v_0 \rMset\Bigr).
  \end{equation*}
\end{lemma}
\begin{proof}
  The first result follows from \Lem{frame.order} since $\Gr F_A$ is a
  disjoint union of frames.  In particular,
  $\card{\Gr F_a}=a(\card {\G}-\card{V_0})+(\card {V_0})$, hence
  \begin{equation*}
    \card{\Gr F_A}
    =\sum_{a\in A} \bigl(a(\card {\G}-\card{V_0})+(\card {V_0}-\card{V_1})\bigr)
    = r(\card {\G}-\card{V_0})+s\card {V_0}\;,
    \qedhere
  \end{equation*}
  where we used~\eqref{eq:part.number}.  The second statement follows
  similarly.
\end{proof}

As the map $A \mapsto \deg \Gr F_A$ is injective, graphs with
different $s$-partitions $A$ and $B$ of $r$ cannot be isomorphic.
Note that in order to have two different $s$-partitions of $r$, we
need $r \ge 4$ and $2 \le s \le r-2$.
\begin{lemma}
  \label{lem:not.isomorphic}
  Let $A$ and $B$ be two different $s$-partitions of $r$ with
  $r \ge 4$ and $2 \le s \le r-2$.  Then the graphs $\Gr F_A$ and $\Gr F_B$
  are not isomorphic.
\end{lemma}

As a prelude (or a special case) of the main theorem
(\Thm{spec.contr.frame.union}) we mention next a first family of
isospectral non-isomorphic graphs labelled by different $s$-partitions
of $r$.  Note that these graphs are not connected.

\begin{proposition}
  \label{prp:not-isomorphic}
  The spectrum of $\wGr F_A$ is given by
  \begin{align*}
    \spec{\wGr F_A}=
    \biguplus_{a \in A} \spec{\wGr F_a}
    = \spec{\wG}^{(s)} \uplus \spec{\DirGr \wG {V_0}}^{(r-s)}.
  \end{align*}
  In particular, for two different $s$-partitions $A$ and $B$ of $r$
  with $r \ge 4$ and $s \ge 2$, the graphs $\wGr F_A$ and $\wGr F_B$
  are isospectral, but not isomorphic.
\end{proposition}
\begin{proof}
  The spectrum of a disjoint union of graphs is the multiset sum of
  its spectra, hence the first equality holds. For the second, we use
  \Prp{frame.spec} and~\eqref{eq:part.number}.  As only $r$ and $s$
  enter in $\spec{\wGr F_A}$, we have
  $\spec{\wGr F_A}=\spec{\wGr F_B}$.  By \Lem{not.isomorphic}, the
  graphs are not isomorphic.
\end{proof}

\subsection{Contracted frame unions}
\label{subsec:contr.frame.union}
Next, we construct contracted frame unions by merging a subset of
distinguished vertices $V_1\subset V_0$.

\begin{definition}[contracted frame union]
  \label{def:contr.frame.union}
  Let $\wG=(\G,\alpha,w)$ be a magnetic graph (building block) and
  choose a subset $V_0\subset V(\G)$.  Consider the frame
  $(\wGr F_a(\wG,V_0))_{a\in \N}=(\wGr F_a)_{a\in \N}$ as specified in
  \Def{frame} and let $A=\lMset a_1,\dots,a_s\rMset$ be an
  $s$-partition of the natural number $r$.  For a subset
  $V_1 \subset V_0$, called the \emph{set of distinguished vertices}, we
  define the \emph{$V_1$-contracted $A$-union of the frame
    $(\wGr F_a)_{a\in \N}$} by
  \begin{equation}
    \label{eq:contr.frame.union}
    \wGr F_{A,V_1}
    := \wGr F_{A,V_1}(\wG,V_0)
    := \quotient[V_1]{\wGr F_A}
    = \quotient[V_1]{\Bigl(\bigdcup_{i=1}^s
        \wGr F_{a_i}\times \{i\}\Bigr)},
  \end{equation}
  where ${\sim}_{V_1}$ contracts the vertices
  $([v_1],i) \in \wGr F_{a_i} \times \{i\}$ ($i \in \{1,\dots,s\}$)
  for each $v_1 \in V_1$ into a single vertex, denoted again by $v_1$ for
  simplicity.
\end{definition}

The next result establishes the order of the graphs constructed
before. It follows directly from the definition.
\begin{lemma}[order and degree multiset of contracted frame
  union]
  \label{lem:order-fa-quot}
  The order of $\Gr F_{A,V_1} = \quotient[V_1]{\Gr F_A}$ is
  \begin{equation*}
    \card{\Gr F_{A,V_1}}
    =r(\card {\G}-\card{V_0}) + s (\card{V_0}-\card{V_1}) + \card{V_1}
  \end{equation*}
  and the degree multiset is
  \begin{equation*}
    \deg \Gr F_{A,V_1}
    = \biguplus_{v\in V(\G)\setminus V_0} \lMset \deg_\G v\rMset ^{(r)}
    \uplus
    \biguplus_{a\in A}
    \Big(\biguplus_{v_0\in V_0 \setminus V_1} \lMset a \deg_\G v_0 \rMset \Bigr)
    \uplus
    \biguplus_{v_1 \in V_1} \lMset r\deg_\G v_1 \rMset.
  \end{equation*}
\end{lemma}

We now calculate the spectrum of the $V_1$-contracted $A$-union of the
frame $(\wGr F_a)_{a\in \N}$ in terms of its building block $\wG$ and
certain Dirichlet conditions.
\begin{theorem}[Spectrum of contracted frame unions]
  \label{thm:spec.contr.frame.union}
  Let $\wG=(\G,\alpha,w)$ be a magnetic graph, with underlying
  discrete graph $\G=(V,E,\partial)$, $V_0\subset V$, and let
  $(\wGr F_a)_{a\in\N}$ with $\wGr F_a=\wGr F_a(\wG,V_0)$ be a frame
  constructed from the building block $\G$ by identifying vertices
  along $V_0$.  Moreover, let $A$ be an $s$-partition of the natural
  number $r$, and choose a subset of distinguished vertices
  $V_1 \subset V_0$.  Then the spectrum of the normalised Laplacian of
  the $V_1$-contracted $A$-union
  $\wGr F_{A,V_1} = \wGr F_{A,V_1}(\wG,V_0)$
  (cf. \Def{contr.frame.union}) is given by
  \begin{equation*}
    \bigspec{\wGr F_{A,V_1}}
    =\spec{\wG}
    \uplus \bigspec{\DirGr \wG {V_0}}^{(r-s)}
    \uplus \bigspec{\DirGr \wG {V_1}}^{(s-1)} .
  \end{equation*}
\end{theorem}
\begin{proof}
  The graphs given in~\eqref{eq:contr.frame.union} come from two
  successive vertex contractions; we hence denote a generic vertex by
  \begin{equation*}
    \wt v =[([(v,j)],i)] \in V(\Gr F_{A,V_1})
    \qquadtext{for}
    j \in \{1,\dots,a_i\}
    \quadtext{and}
    i \in \{1,\dots,s\}\;,
  \end{equation*}
  where $v\in V$ is a vertex in the building block graph $\G$, $j$
  labels the branch in the frame member $F_{a_i}$ and $i$ numerates
  the frame member determined by the partition
  $A=\lMset a_1,\dots, a_s\rMset$.

  \myparagraph{Symmetric eigenfunctions.}  We start with the
  ``symmetric'' functions obtained from eigenfunctions of $\wG$
  extended symmetrically to $\Gr F_{A,V_1}$: Let $n=\card \wG$
  and denote by $\{f_1,f_2,\dots,f_n\}$ the orthonormal eigenfunctions
  of $\lapl {\wG}$ with eigenvalues $\{\rho_1,\dots,\rho_n\}$,
  respectively. For each $k \in \{1,\dots, n\}$, we extend the
  eigenfunction $f_k$ onto the disjoint union $\wGr F_A$ to a function
  \begin{equation*}
    \map{\wt f_k}{V(\wGr F_A)}{\C}
    \qquadtext{by}
    \wt f_k(\wt v)=f_k(v).
  \end{equation*}
  With a little abuse of notation we write by $\wt f_k$ also the
  corresponding function on the quotient $\wGr F_{A,V_1}$, which
  is well-defined since $\wt f$ takes the same value on all contracted
  vertices.  We show that $\wt f_k$ is an eigenfunction of
  $\lapl {\wGr F_{A,V_1}}$ with the same eigenvalue $\rho_k$
  considering three different cases for the vertex $v$.

  If $v \in V_1$, then the equivalence class $\wt v$ consists of
  $r=\sum_{i=1}^s a_i$ vertices all contracted into a single vertex,
  while the adjacent edges remain.  In particular,
  $\deg_{\Gr F_{A,V_1}}^w (\wt v) = r \deg_\G v$ which gives
  \begin{align*}
    \bigl(\lapl {\wGr F_{A,V_1}} \wt f_k\bigr)(\wt v)
    &= \wt f_k(\wt v)
      - \frac1{\deg_{\Gr F_{A,V_1}}^w(\wt v)}
      \sum_{e\in E_{\wt v}(\Gr F_{A,V_1})}
      w_e \e^{\im \alpha_e} \wt f_k(\bd_+e)\\
    &= f_k(v) - \frac 1 {r \deg_\G ^w (v)}
       \cdot r \sum_{e\in E_v(\G)} w_e \e^{\im \alpha_e}f_k(\bd_+e)\\
    &=  \bigl(\lapl {\wG}f_k\bigr)(v)=\rho_k f_k(v)
    =\rho_k \wt f_k (\wt v).
  \end{align*}
  If $v \in V_0 \setminus V_1$ the equivalence class
  $\wt v=[([(v,j)],i)]$ consists of $a_i$ elements, denoted for
  simplicity by $(v,i)$.  Moreover, we have
  $\deg_{\Gr F_{A,V_1}} (v,i) = a_i \deg_\G^w v$, and hence
   \begin{align*}
    \bigl(\lapl {\wGr F_{A,V_1}} \wt f_k\bigr)(v,i)
    &= \wt f_k(v,i)
      - \frac1{\deg_{\Gr F_{A,V_1}}^w (v,i)}
      = \sum_{e\in E_{(v,i)}(\wt{\Gr F}_{A,V_1})}
         w_e \e^{\im \alpha_e} \wt f_k(\bd_+e)\\
    &= f_k(v) - \frac 1 {a_i \deg_\G ^w v_0}
      \cdot a_i \sum_{e\in E_v(\G)} w_e \e^{\im \alpha_e}f_k(\bd_+e)\\
    &=  \bigl(\lapl {\wG}f_k\bigr)(v)
      =\rho_k f_k(v)
      =\rho_k \wt f_k (v,i).
  \end{align*}
  Finally, if $v \in V \setminus V_0$, the equivalence class
  $\wt v=[([(v,j)],i)]$ has just one element which is denoted simply
  by $(v,j,i)$. In this case we have
  $\deg_{\Gr F_{A,V_1}}^w(v,j,i)=\deg_\G^w v$ which gives
  \begin{align*}
    \bigl(\lapl {\wGr F_{A,V_1}} \wt f_k\bigr)(v,j,i)
    &= \wt f_k(v,j,i)-\frac1{\deg_{\Gr F_{A,V_1}}^w(v,j,i)}
      \sum_{e\in E_{(v,j,i)}(\Gr F_{A,V_1})}
         w_e \e^{\im \alpha_e}\wt f_k(\bd_+e)\\
    &= f_k(v)-\frac1{\deg_\G^w v} \sum_{e\in E_v(\G)}
         w_e \e^{\im \alpha_e}f_k(\bd_+e)\\
    &=  \bigl(\lapl \wG f_k\bigr)(v)
      =\rho_k f_k(v)
      = \rho_k \wt f_k (v,j,i).
  \end{align*}
  This shows that $\spec{\wG}$ is contained in
  $\bigspec{\wGr F_{A,V_1}}$ as a multiset.

  \myparagraph{Eigenfunctions with Dirichlet conditions on $V_1$.}
  Dirichlet conditions on $V_1$ will disconnect the graph
  $\wGr F_{A,V_1}$ into $s$ frame member components determined by the
  partition $A$. Let
  $n'= \card{\DirGr \wG {V_1}}=\card V - \card{V_1}$ and denote by
  $\{g_1, \dots, g_{n'}\}$ the eigenfunctions of
  $\lapl {\DirGr \wG {V_1}}$ with eigenvalues
  $\{\mu_1,\dots, \mu_{n'}\}$. We will lift the eigenfunctions
  $g_{k'}$ to eigenfunctions of $\lapl {\wGr F_{A,V_1}}$ supported on
  each frame member $F_{a_i}$. Denote now by $\eps$ a non-trivial
  $s$-th root of unity, i.e., $\eps\not=1$ and $\eps^{s}=1$. For
  $\wt v=[([(v,j)],i)]$ we define
  \begin{equation*}
    \wt g_{k',\eps}(\wt v)=
    \eps^i g_{k'}(v).
  \end{equation*}
  Note that now the function value of a vertex $[(v,j)]$ on
  $\wGr F_{a_i}$ is the same for all $j \in \{1,\dots,a_i\}$ as the
  root of unity takes the same value on each frame member.  It can be
  seen similarly as before (and as in the proof of \Prp{frame.spec})
  that
  \begin{align*}
    \bigl(\lapl {\wGr F_{A,V_1}} \wt g_{k',\eps}\bigr)(\wt v)
    = \mu_{k'} \wt g_{k',\eps}(\wt v)
  \end{align*}
  for all $\wt v \in V(\wGr F_{A,V_1})$. Moreover, since $\eps \ne 1$
  we have $\sum_{i=1}^s \eps^i=0$ and, therefore, the functions
  $\wt f_k$ and ${\wt g_{k',\eps}}$ are mutually orthogonal, as we
  have
  \begin{equation*}
    \bigiprod[\lsqr{V(\wGr F_{A,V_1}),\deg^{w}}]{\wt f_k}{\wt g_{k',\eps}}
    = \sum_{i=1}^s \bigiprod[\lsqr{V(\wGr F_{a_i}),\deg^w}]{f_k}{\eps^i g_{k'}}
    =0.
  \end{equation*}
  In particular, the eigenfunctions $\wt f_k$ are orthogonal to the
  eigenfunctions $\wt g_{k',\eps}$ with $k=1,\dots,n$,
  $k'=1,\dots,n'$ and $\eps\ne 1$ with $\eps^s=1$.

  \myparagraph{Eigenfunctions with Dirichlet conditions on $V_0$.}
  Since $V_1\subset V_0$ we have by construction that eigenfunctions
  on $\wGr F_{A,V_1}$ with Dirichlet conditions on $V_0$ are
  eigenfunctions of the disjoint union
  $\wGr F_A(\DirGr \wG {{V_0}},V_0)$ (see \Def{frame.union}). In fact,
  we have
  \begin{equation*}
    \DirGr {(\wGr F_{A,V_1}(\wG,V_0))}{V_0}
    = \wGr F_A(\DirGr \wG {V_0},V_0).
  \end{equation*}
  Exploiting the symmetry of the frame members we will construct next
  a family of eigenfunctions which are lifted from eigenfunctions of
  the Laplacian on $\wG$ with Dirichlet conditions on $V_0$ and which
  are orthogonal to the symmetric eigenfunctions $\wt f_k$ constructed
  in the first step.  Let
  $n''=\card{\DirGr \wG {V_0}} =\card V - \card{V_0 }$ and denote by
  $\{h_1, \dots, h_{n''}\}$ the eigenfunctions of
  $\lapl {\DirGr \wG {V_0}}$ with eigenvalues
  $\{\lambda_1,\dots, \lambda_{n''}\}$. For each $i\in\{1,\dots,s\}$ we
  will lift the eigenfunctions $h_k$ to eigenfunctions of
  $\lapl {\wGr F_{A,V_1}}$ supported on each frame member
  $F_{a_i}$. Denote by $\eta$ a non-trivial $a_i$-th root of unity,
  i.e., $\eta\not=1$ and $\eta^{a_i}=1$. For $\wt v=[([(v,j)],i)]$ we
  define
  \begin{equation}\label{eq:def.h}
   \wt h_{k'',\eta,i}(\wt v)= \eta^j h_{k''}(v)
  \end{equation}
  and extend by $0$ if the label of the frame member is different from
  $i$. Since the functions $\wt h_{k'',\eta, i}$ are supported on each
  frame member the proof that the preceding functions are
  eigenfunctions of $\lapl {\wGr F_{A,V_1}}$ with eigenvalue
  $\lambda_{k''}$ can be reduced to the analysis on each frame
  $\wGr F_{a_i}$.  This was shown in the proof of \Prp{frame.spec},
  and the inclusion of multisets
  $\bigspec{\DirGr \wG {V_0}}^{(a_i-1)}\subset \bigspec{\wGr F_{a_i}(
    \DirGr \wG {V_0})}$ follows.  Since this inclusion holds for each
  $i\in\{1,\dots,s\}$ and since $a_1+\dots +a_s=r$ we conclude that
  \begin{align*}
    \bigspec{\DirGr \wG {V_0}}^{(r-s)}\subset \bigspec{\wGr F_{A,V_1}}.
  \end{align*}
  As in the proof of \Prp{frame.spec} we also see that the functions
  $\wt f_k$ and $\wt h_{k'',\eta,i}$ are mutually orthogonal since the
  functions $\wt h$ are supported on each frame member. Moreover, the
  eigenfunctions $\wt g_{k',\eps}$ constructed in the second step and
  $\wt h_{k'',\eta,i}$ turn out
  to be also mutually orthogonal with a similar computation with the root of
  unity as in the second step. This shows that
  \begin{align*}
    \bigspec{\DirGr \wG {V_1}}^{(s-1)}\subset \bigspec{\wGr F_{A,V_1}}.
  \end{align*}

  Finally, we have to check that we identified all eigenvalues.
  In fact, we have identified $\card \G$ symmetric eigenfunctions and
  \begin{equation*}
    \sum_{a \in A} (a-1)(\card{\wG}-\card{V_0})
    =(r-s)(\card{\wG}-\card{V_0})
  \end{equation*}
  eigenfunctions from $\DirGr \wG {V_0}$.  Moreover, the remaining
  eigenfunctions are $\wt g_{k',\eps}$ for
  $k' \in \{1,\dots, \card \G - \card{V_1}\}$ and $\eps^s=1$ with
  $\eps\ne1$ determine $(s-1)(\card \G-\card{V_1})$ additional eigenvalues.
  Altogether, we have specified
  \begin{equation*}
    \card \G
    + (r-s)(\card{\wG}-\card{V_0})
    + (s-1)(\card \G-\card{V_1})
    =r(\card{\wG}-\card{V_0}) + s(\card{V_0}-\card{V_1}) + \card{V_1}
  \end{equation*}
  eigenvalues corresponding to a mutually orthogonal set of
  eigenfunctions.  Since, according to \Lem{order-fa-quot}, the order
  of $\Gr F_{A,V_1}$ is precisely
  $r(\card{\wG}-\card{V_0}) + s(\card{V_0}-\card{V_1}) + \card{V_1}$
  we conclude that we found all eigenvalues and the spectrum of the
  magnetic Laplacian is determined.
\end{proof}

We can now formulate our main result on the construction of
isospectral, non-isomorphic and, now, connected graphs.
\begin{theorem}[main theorem]
  \label{thm:main}
  Let $\wG=(\G,\alpha,w)$ be a magnetic graph, $V_0\subset V(\G)$ and
  let $V_1 \subset V_0$ be a set of distinguished vertices.  For any
  different pair $A$ and $B$ of $s$-partitions of a natural number
  $r \ge 4$ the graphs $\wGr F_{A,V_1}$ and $\wGr F_{B,V_1}$ (the
  $V_1$-contracted $A$- respectively\ $B$-union of
  $(\wGr F_a(\wG,V_0))_{a\in \N}$, cf.\ \Def{contr.frame.union}) are
  isospectral and not isomorphic.
\end{theorem}
\begin{proof}
  In the calculation of the spectrum in \Thm{spec.contr.frame.union},
  only the natural number $r$ and the length $s$ of the partition are
  relevant and not the concrete partitions $A$ and $B$.  Therefore
  different $s$-partitions of $r$ lead to isospectral graphs.
  Moreover, the fact that two graphs determined by different
  partitions are not isomorphic follows from \Lem{order-fa-quot} since
  the corresponding degree multisets are different.
\end{proof}

\begin{corollary}[isospectral equilateral metric graphs]
  \label{cor:main}
  Let $\G$ be a discrete graph (with standard weights),
  $V_0\subset V(\G)$ and let $V_1 \subset V_0$ be a set of
  distinguished vertices.  For any different pair $A$ and $B$ of
  $s$-partitions of a natural number $r \ge 4$ the corresponding
  equilateral metric graphs $\mGr F_{A,V_1}$ and $\mGr F_{B,V_1}$
  constructed according to the corresponding discrete graphs are
  isospectral and not isomorphic.
\end{corollary}
\begin{proof}
  The proof follows from \Thm{main} and \Prp{met.gr.isospec}, as in
  our construction, $\Gr F_{A,V_1}$ and $\Gr F_{B,V_1}$ have the same
  number of edges.
\end{proof}

%
%
\section{Examples of isospectral magnetic graphs}
\label{sec:examples}
%
In this section we give more examples of isospectral magnetic graphs
constructed as contracted frame union from various building blocks.
Let $\wG=(\G,\alpha,w)$ be such a general building block, i.e., a
discrete weighted magnetic graph.  From $\wG$, we construct a frame
$(\wGr F_a)_{a \in \N}$, where the frame members
$\wGr F_a=\wGr F_a(\wG,V_0)$ are identified along
$V_0 \subset V:=V(\G)$, cf.\ \Def{frame}.  Moreover, let $A$ be an
$s$-partition of $r \in \N$.  As we need at least two different
$s$-partitions of $r$ we restrict ourselves to
$r \in \{4,5,6,\dots \}$ and $s \in \{2,3,\dots,r-2\}$.  Our
isospectral and non-isomorphic graphs will be given by the
$V_1$-contracted frame unions $\wGr F_{A,V_1}$ and $\wGr F_{B,V_1}$,
where $V_1$ is some subset of $V_0$.

\subsection{Special cases and very small contracted frame unions}
\label{subsec:small-iso-frames}

We start with some extreme or trivial cases for the choice of merging vertex sets $V_0$ and $V_1$.
\begin{examples}[no distinguished vertex]
  \label{ex:no.distinguished.vx}
  If $V_1=\emptyset$ in \Def{contr.frame.union}, then
  $\wGr F_{A,\emptyset}=\wGr F_A (:=\bigdcup_{a \in A} \wGr
  F_a(G,V_0))$, i.e., $\wGr F_{A,\emptyset}$ is the disjoint frame
  union of \Def{frame.union}.  Note that the
  graph $\wGr F_{A,\emptyset}$ is not connected.  In particular, for
  different $s$-partitions of $r$, we obtain two isospectral,
  non-isomorphic, but non-connected graphs $\wGr F_A$ and $\wGr F_B$.
\end{examples}

\begin{examples}[all vertices are distingued]
  \label{ex:all.distingushed.vx}
  If $V_1=V_0$, then $\wGr F_{A,V_1}=\wGr F_r$ is the $r$-th frame
  member.  In particular, it depends only on $r$, but not on $s$ any
  more, hence we will not obtain non-isomorphic graphs for different
  partitions of $r$.
\end{examples}
We will hence assume in the sequel that $V_1$ is a \emph{non-trivial}
subset of $V_0$, i.e., that $V_1 \ne \emptyset$ and $V_1 \ne V_0$.
 \begin{table}[h]
   \centering
\begin{tabular}[h]{|l||c|c|c|c|c|c|c|c|c|}
   \hline
   label&$\G$ & $\card V$ &$\card {V_0}$ & $\card {V_1}$ & $\Gr F_2$ & $\Gr F_3$ & $\Gr F_{A,V_1}$
   & $\Gr F_{B,V_1}$ & $\card{\Gr F_{C,V_1}}$, $\spec{\Gr F_{C,V_1}}$\\\hline\hline
        && $2$ & $2$ & $1$&$2$&$2$&$3$&$3$&$s+1$\\
\hline
 \end{tabular}
 \caption{All examples of tree building blocks with two or three vertices
   and their different contracted frame unions (including those with
   multiple edges) for the simplest non-trivial partition
   $A=\lMset 1,3\rMset$, $B=\lMset 2,2\rMset$. On the bottom there are
   the vertices in $V_1$ (outlined) and on top, there are the vertices
   $V_0 \setminus V_1$.}
   \label{tab:all.ex.3vx}
 \end{table}
 \begin{table}[h]
   \centering
 \begin{tabular}[h]{|l||c|c|c|c|c|c|c|c|c|}
   \hline
   label&$\G$ & $\card V$ &$\card {V_0}$ & $\card {V_1}$ & $\Gr F_2$ & $\Gr F_3$ & $\Gr F_{A,V_1}$
   & $\Gr F_{B,V_1}$ & $\card{\Gr F_{C,V_1}}$, $\spec{\Gr F_{C,V_1}}$\\\hline\hline
 & & $3$ & $2$ & $1$&$4$&$5$&$7$&$7$ & $r+s+1$\\
\hline
 \end{tabular}
 \caption{All examples of non-tree building blocks with three vertices
   and their different contracted frame unions (including those with
   multiple edges) for the simplest non-trivial partition
   $A=\lMset 1,3\rMset$, $B=\lMset 2,2\rMset$.}
   \label{tab:all.ex.3vx'}
 \end{table}

\begin{examples}[frames identified along all vertices]
  \label{ex:all.vx.id}
  If $V_0=V(\G)$, then $\wGr F_a(\wG,V(\G))$ is the $r$-fold edge copy
  of $\wG$, i.e., the vertex set is the same, but each edge is copied
  $r$ times keeping its original weight and magnetic potential.
  Examples are the graphs from the building blocks labelled~2.1, 3.3,
  3.3', 3.4, 3.4', 3.6, 3.6' in \Tabs{all.ex.3vx}{all.ex.3vx'}.

  Moreover, the order of the $V_1$-contracted frame union
  $\wGr F_{A,V_1}$ for some $V_1 \subset V_0$ now is
  \begin{equation*}
    \card{\wGr F_{A,V_1}}
    =s(\card{V_0}-\card{V_1}) +\card {V_1},
  \end{equation*}
  i.e., it does not depend on $r$.  In particular, we obtain an
  infinite family of isospectral graphs, as an $s$-partition for
  \emph{any} $r \in \{s+2,s+3, \dots\}$ leads to another
  non-isomorphic graph (isomorphy as multigraph, of course).  We have
  another class of examples in \Ex{all.but.1.contracted}.
\end{examples}

\begin{remark}[frames identified along all vertices, weighted graphs]
  If $\wG=(\G,0,1)$ has standard weight, then the frames for
  $V_0=V(\G)$ as in the previous example are the $r$-fold edge copy of
  $\G$ (i.e., each edge with standard weight is repeated $r$ times).
  Note that $\G$ and its $r$-fold edge copy are isolaplacian (as the
  common factor $r$ cancels out in the matrix representation of the
  Laplacian).

  Moreover, each frame member and also the $V_1$-contracted frame
  union $\Gr F_{A,V_1}$ for some proper subset $V_1 \subset V_0$ can
  be turned into a simple graph (called \emph{underlying simple
    weighted graph}) with edge weights now given by the number of
  parallel edges in the original graph, see \Rem{mult.ed.weighed.gr}.
  Note that as weighted graphs, two contracted frame unions for
  different $s$-partitions are still non-isomorphic (as weighted
  graphs), as the weights are given by the partition $A$.
\end{remark}

A special case of the last example is the following:
\begin{examples}[the smallest non-trivial example]
  Let $\G$ be the graph on two vertices and one edge joining them with
  standard weight.  Let $V_0=V=\{v_1,v_2\}$ and $V_1=\{v_1\}$ (see the
  first row in \Tab{all.ex.3vx}).  Moreover, let
  $A=\lMset a_1,\dots,a_s\rMset$ be an $s$-partition of $r$.  Then
  $\wGr F_{A,v_1}$ is a star graph with $s+1$ vertices and multiple
  edges according to the partition $A$ ($a_1$ edges joining the first
  non-central vertex, $a_2$ the second etc.).  Note that for two
  different $s$-partitions, the corresponding contracted frame unions
  are not isomorphic (as multigraphs).

  The spectrum of $\wGr F_{A,v_1}$ is $\lMset 0, 1^{(s-1)},2\rMset$
  for any $s$-partition $A$.  For $s=2$, the underlying simple
  weighted graph is a path graph with three vertices.  Actually, it
  can be seen that the spectrum of this graph is always
  $\lMset 0,1,2\rMset$ for \emph{any} weights on the two edges.
\end{examples}

 \begin{table}[h]
   \centering
\begin{tabular}[h]{|l||c|c|c|c|c|c|c|c|c|}
   \hline
   label&$\G$ & $\card V$ &$\card {V_0}$ & $\card {V_1}$ & $\Gr F_2$ & $\Gr F_3$ & $\Gr F_{A,V_1}$
   & $\Gr F_{B,V_1}$ &$\card{\Gr F_{C,V_1}}$, $\spec{\Gr F_{C,V_1}}$\\\hline\hline
   & & $4$ & $2$ & $1$&$6$&$8$&$11$&$11$& $2r+s+1$\\
\hline\hline
  \end{tabular}
  \caption{All examples of tree building blocks with four
    vertices and their different contracted frame unions leading to
    \emph{simple} graphs.}
   \label{tab:all.ex.4vx}
 \end{table}

\subsection{Other small contracted frame unions}
\label{subsec:small-iso-frames2}

We now list all possible cases with standard weights leading to
isospectral and non-isomorphic examples starting from building blocks
$\G$ with two, three or four vertices.  We start with a certain graph
$\G$ as building block, and draw the vertices in $V_0$ on top and
bottom; while the vertices in $V_1$ (the ``distinguished'' vertices,
along the frames are contracted) are drawn on bottom and outlined).
For one graph $\G$ one might hence have several possibilities of frames
and contracted frame unions.

If $\G$ has two or three vertices, we allow that the isospectral
graphs have parallel edges (see \Tabs{all.ex.3vx}{all.ex.3vx'}).  If $\G$ has four
vertices, we restrict ourselves to examples leading only to
\emph{simple} graphs, see \Rem{simple.frames}.  \Tab{all.ex.4vx}
contains all examples with a building block being a tree with four
vertices while \Tab{all.ex.4vx'} lists all remaining cases.

As simplest non-trivial partitions in \TabS{all.ex.3vx}{all.ex.4vx'}
we choose $A=\lMset 1,3\rMset$ and $B=\lMset 2,2\rMset$.  Moreover,
$C$ denotes an arbitrary $s$-partition of $r$.
\begin{examples}[all examples with building blocks with three
  vertices]
   \label{ex:small.ex}
   Examples of building blocks with three vertices and standard
   weights can be seen in \Tabs{all.ex.3vx}{all.ex.3vx'} (labels starting with 3).
   Only one building block leads to simple isospectral examples (label
   3.1); this example class was already presented
   in~\cite[Example~2]{butler-grout:11} under the name \emph{inflated
     stars} there.

   All other examples leads to graphs with parallel edges.  Note that
   one can always replace parallel edges by a single weighted edge
   with the number of parallel edges as weight, see
   \Rem{mult.ed.weighed.gr}.  The corresponding Laplacians are the
   same, hence we also obtain isospectral, but non-isomorphic
   examples; these examples are not isomorphic as weighted graphs as
   their lists of weights are different.
 \end{examples}

 \begin{table}[h]
   \centering
\begin{tabular}[h]{|l||c|c|c|c|c|c|c|c|c|}
   \hline
   label&$\G$ & $\card V$ &$\card {V_0}$ & $\card {V_1}$ & $\Gr F_2$ & $\Gr F_3$ & $\Gr F_{A,V_1}$
   & $\Gr F_{B,V_1}$ & $\card{\Gr F_{C,V_1}}$, $\spec{\Gr F_{C,V_1}}$\\\hline\hline
   & & $4$ & $2$ & $1$&$6$&$8$&$11$&$11$&$2r+s+1$\\
\hline
 \end{tabular}
 \caption{All examples of non-tree building blocks with four vertices
   and their different contracted frame unions leading to
   \emph{simple} graphs.}
   \label{tab:all.ex.4vx'}
 \end{table}

 \begin{examples}[all examples with building blocks with four vertices]
   \label{ex:small.ex4}
   For simplicity, we list here only those cases leading to
   \emph{simple} isospectral examples.  \Tab{all.ex.4vx} contains all
   cases where the building block $\G$ is a tree while
   \Tab{all.ex.4vx'} contains the remaining cases with building
   blocks not being a tree.

   Note that only the cases with label~4.2', 4,4 and 4.4' can be
   treated using Theorem~1 of~\cite{butler-grout:11}.  In all other
   cases, the graphs $G_j[A',B]$ (in the notation
   of~\cite{butler-grout:11}) are not complete bipartite graphs as
   required by~\cite[Thm~.1]{butler-grout:11}.

   The case with label~4.4 can also be treated with the perturbative
   technique described in~\cite[Section~4.2]{fclp:pre22a}, where we
   used the name \emph{fuzzy complete bipartite graph}.  With the
   method in the present paper, it is based on the frames used in
   \Ex{complete.bipartite}.
 \end{examples}

 \subsection{Some other examples with possible non-trivial magnetic
   potential}
\label{subsec:mag-iso-frames}
Here, we give examples where the underlying building block is not a
tree, hence there exist magnetic potentials having an effect on the
spectrum.
\begin{example}[a diamond-like graph with magnetic potential]
  \label{ex:diamond.mag.iso}
  Taking as building block the magnetic graph $\wG^\theta$ of
  \Ex{diamond.mag}, and choosing as set $V_1=\{v_1\}$ just the bottom
  (outlined) vertex, then the spectrum of
  $\DirGr {(\wG^\theta)} {V_1}$ is given by
  \begin{equation*}
    \bigspec{\DirGr {(\wG^\theta)} {V_1}}
    =\Bigl\lMset
    1-\frac{\cos\theta}4 - \frac12\sqrt{1+\frac{\cos^2\theta}4},
    1-\frac{\cos\theta}4 + \frac12\sqrt{1+\frac{\cos^2\theta}4}
    \Bigr\rMset.
  \end{equation*}
  For $A=\lMset 1,3\rMset$ and $B=\lMset 2,2\rMset$, we have the
  isospectral magnetic graphs as in \Fig{diamond-mag.iso}.
  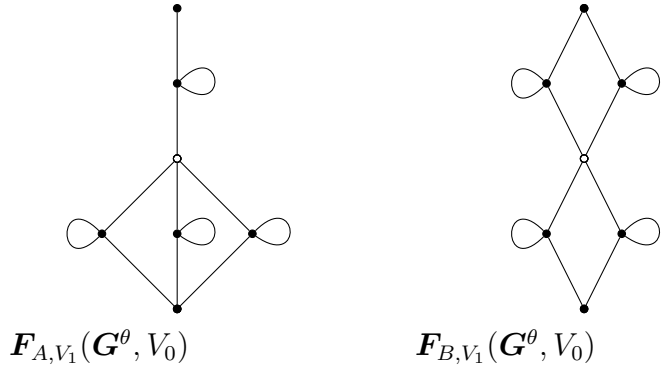
\begin{figure}[h]
  \centering
  {
    \begin{tikzpicture}[auto, vertex/.style={circle,draw=black!100,fill=black!100, 
        inner sep=0pt,minimum size=1mm},scale=1]
      \foreach \x in {0}
      { \node(C\x) at (0,1) [vertex,draw=black,fill=none,label=below:] {};
        \node (B\x) at (\x*.5,2) [vertex,label=below:] {};
        \draw[-] (B\x) to[min distance=8mm,in=45,out=325,looseness=1]  node[right] {} (B\x);
        \path [-](C\x) edge node[right] {} (B\x);
        \node (A\x) at (0,3) [vertex,label=below:] {};
        \path [-](B\x) edge node[right] {} (A\x);
      }
      \foreach \x in {0,2}
      { \node(C\x) at (0,1) [vertex,draw=black,fill=none,label=below:] {};
        \node (B\x) at (\x*.5,0) [vertex,label=below:] {};
        \draw[-] (B\x) to[min distance=8mm,in=45,out=325,looseness=1]  node[right] {} (B\x);
        \path [-](C\x) edge node[right] {} (B\x);
        \node (A\x) at (0,-1) [vertex,label=below:] {};
        \path [-](B\x) edge node[right] {} (A\x);
      }
      \foreach \x in {-2}
      { \node(C\x) at (0,1) [vertex,draw=black,fill=none,label=below:] {};
        \node (B\x) at (\x*.5,0) [vertex,label=below:] {};
        \draw[-] (B\x) to[min distance=8mm,in=135,out=225,looseness=1]  node[right] {} (B\x);
        \path [-](C\x) edge node[right] {} (B\x);
        \node (A\x) at (0,-1) [vertex,label=below:] {};
        \path [-](B\x) edge node[right] {} (A\x);
      }
      \draw[] (.3,.-1.5)node[left] {$\wGr F_{A,V_1}(\wG^\theta,V_0)$};
    \end{tikzpicture}\quad \quad \quad
    \begin{tikzpicture}[auto, vertex/.style={circle,draw=black!100,fill=black!100, 
        inner sep=0pt,minimum size=1mm},scale=1]
      \foreach \x in {-1}
      { \node(C\x) at (0,1) [vertex,draw=black,fill=none,label=below:] {};
        \node (B\x) at (\x*.5,2) [vertex,label=below:] {};
        \draw[-] (B\x) to[min distance=8mm,in=135,out=225,looseness=1]  node[right] {} (B\x);
        \path [-](C\x) edge node[right] {} (B\x);
        \node (A\x) at (0,3) [vertex,label=below:] {};
        \path [-](B\x) edge node[right] {} (A\x);
      }
      \foreach \x in {1}
      { \node(C\x) at (0,1) [vertex,draw=black,fill=none,label=below:] {};
        \node (B\x) at (\x*.5,2) [vertex,label=below:] {};
        \draw[-] (B\x) to[min distance=8mm,in=45,out=325,looseness=1]  node[right] {} (B\x);
        \path [-](C\x) edge node[right] {} (B\x);
        \node (A\x) at (0,3) [vertex,label=below:] {};
        \path [-](B\x) edge node[right] {} (A\x);
      }
      \foreach \x in {1}
      { \node(C\x) at (0,1) [vertex,draw=black,fill=none,label=below:] {};
        \node (B\x) at (\x*.5,0) [vertex,label=below:] {};
        \draw[-] (B\x) to[min distance=8mm,in=45,out=325,looseness=1]  node[right] {} (B\x);
        \path [-](C\x) edge node[right] {} (B\x);
        \node (A\x) at (0,-1) [vertex,label=below:] {};
        \path [-](B\x) edge node[right] {} (A\x);
      }
      \foreach \x in {-1}
      { \node(C\x) at (0,1) [vertex,draw=black,fill=none,label=below:] {};
        \node (B\x) at (\x*.5,0) [vertex,label=below:] {};
        \draw[-] (B\x) to[min distance=8mm,in=135,out=225,looseness=1]  node[right] {} (B\x);
        \path [-](C\x) edge node[right] {} (B\x);
        \node (A\x) at (0,-1) [vertex,label=below:] {};
        \path [-](B\x) edge node[right] {} (A\x);
      }

      \draw[] (.3,.-1.5)node[left] {$\wGr F_{B,V_1}(\wG^\theta,V_0)$};
    \end{tikzpicture}
  }
  \caption{Two magnetic isospectral graphs. Note that the magnetic
    potential is the same on each loop.}
  \label{fig:diamond-mag.iso}
  \end{figure}
\end{example}


\begin{example}[kite graphs]
   \label{ex:kite}
   Let $\G$ be the complete graph on four vertices with one pendant
   vertex added (so $\G$ has five vertices) as in the frame
   construction of \Ex{kite.frame}.  The vertex set $V_0$ consists now
   of the pendant vertex and one of the remaining four.  The set of
   distinguished vertices $V_1$ consists just of the pendant vertex.
   Let $\wG$ be the corresponding weighted graph with standard weights
   and no magnetic potential, then
   \begin{align*}
     \spec{\DirGr \wG {V_1}}
     = \Bigl\lMset \frac{4-\sqrt {13}}6,\frac 43,\frac 43,\frac{4+\sqrt {13}}6
     \Bigr\rMset.
   \end{align*}
   An example of isospectral graphs is given in \Fig{kite.isospec}
   (right hand side).  Note that this construction also works for any
   magnetic potential, see \Ex{kite.frame}.
\end{example}

\begin{example}[all non-trivial magnetic graphs with building block
  with four vertices]
  Note that all examples of \Tab{all.ex.4vx'} can carry also a
  magnetic potential on the building block leading to a non-trivial
  family with one parameter (labels 4.5, 4.5' and 4.6) respectively
  two parameters (label 4.7) in
  $\RmodZ$.
\end{example}

 \subsection{A class of weighted isospectral graphs}
\label{subsec:weighted-iso}

\begin{example}[contracted along all but one vertex]
  \label{ex:all.but.1.contracted}
  Let $\G$ be a graph with standard weights and $V_0=V(\G)$.  Then the
  frame members $\Gr F_a(\wG,V_0)$ are the $r$-fold edge copy of $\G$
  as in \Ex{all.vx.id} and in particular isolaplacian with $\G$ (see
  second and third graph in \Fig{all.but.1.identified}).  Note that
  Butler~\cite{butler:15} also uses the idea of changing weights.
  Special cases are actually the graphs with label 2.1, 3.3', 3.4' and
  3.6' in \Tabs{all.ex.3vx}{all.ex.3vx'} where $\G$ is a path graph
  with two or three vertices or a cycle graph with three vertices.

  Let now $\Gr F_{A,V_1}=\Gr F_{A,V_1}(\G,V(\G))$ be the contracted
  frame union for $V_1=V(\G) \setminus \{v^*\}$ and some $s$-partition
  of $r$.  Here, $v^* \in V(\G)$ is one chosen vertex.  Its order is
  $\card{\Gr F_{A,V_1}}=s+\card{V_1}=\card \wG +s-1$ independently of
  $r$ as in \Ex{all.vx.id}.  The spectrum of $\DirGr \G {V_1}$ is just
  $\lMset 1 \rMset$, as $V(\G) \setminus V_1=\{v^*\}$ contains just
  one point.  From \Thm{spec.contr.frame.union} we conclude that
  \begin{equation*}
    \spec{\Gr F_{A,V_1}}
    = \spec \G \uplus \lMset 1\rMset^{(s-1)}.
  \end{equation*}
  Note that the matrix representation of Laplacian on $\Gr F_{A,V_1}$
  (with standard weights) has the matrix representation of $\G^*$ as
  principal submatrix, where $\G^*$ denotes the graph $\G$ with $v^*$
  and all adjacent edges removed.  It could also be observed that
  passing to the corresponding simple weighted graphs (with weights
  $r$ on all edges in $\G^*$) and weights $a \in A$ on the edges
  joining the $s$ copies of $v^*$.  Dividing the weights by the common
  factor $1/r$ (this does not change the Laplacian), we end up with an
  isolaplacian graph $\G_{A,v^*}$ with standard weights ($w_e=1$) for
  all edges inside $\G^*$, and weights $a/r$ ($a \in A$) for the
  remaining edges joining the extra $s$ copies of $v^*$ (see
  \Fig{all.but.1.identified}).  In particular, we have a way of
  keeping the spectrum of $\G$ and adding eigenvalues $1$.
  \begin{figure}[h]
    \centering
    {
      \begin{tikzpicture}[auto,vertex/.style={circle,draw=black!100,%
          fill=black!100, 
          inner sep=0pt,minimum size=1mm},scale=1]
        \foreach \x in {0}
        { \node(C\x) at (\x*0.5,1) [vertex,fill=black,label=left:] {};
          \node (B\x) at (-0.5,0) [vertex,fill=white,label=left:] {};
          \node (BB\x) at (0.5,0) [vertex,fill=white,label=left:] {};
          \path [-](C\x) edge node[right] {} (B\x);
          \path [-](C\x) edge node[right] {} (BB\x);
          \path [-](B\x) edge node[right] {} (BB\x);
          \node (A\x) at (-0.5,-1) [vertex,fill=white,label=left:] {};
          \node (AA\x) at (0.5,-1) [vertex,fill=white,label=left:] {};
          \path [-](B\x) edge node[right] {} (A\x);
          \path [-](BB\x) edge node[right] {} (AA\x);
          \path [-](A\x) edge node[right] {} (AA\x);
        }
        \draw[] (0.2,.-1.5)node[left] {$\G$};
        \draw[] (0.5,1.3)node[left] {$v^*$};
      \end{tikzpicture}\quad
      \begin{tikzpicture}[auto,vertex/.style={circle,draw=black!100,%
          fill=black!100, 
          inner sep=0pt,minimum size=1mm},scale=1]
        \foreach \x in {0}
        { \node(C\x) at (\x*0.5,1) [vertex,fill=black,label=left:] {};
          \node (B\x) at (-0.5,0) [vertex,fill=white,label=left:] {};
          \node (BB\x) at (0.5,0) [vertex,fill=white,label=left:] {};
          \path [bend left=7](C\x) edge node[right] {} (B\x);
          \path [bend right=7](C\x) edge node[right] {} (B\x);
          \path [bend left=7](C\x) edge node[right] {} (BB\x);
          \path [bend right=7](C\x) edge node[right] {} (BB\x);
          \node (A\x) at (-0.5,-1) [vertex,fill=white,label=left:] {};
          \node (AA\x) at (0.5,-1) [vertex,fill=white,label=left:] {};
          \path [bend left=7](B\x) edge node[right] {} (BB\x);
          \path [bend right=7](B\x) edge node[right] {} (BB\x);
          \path [bend left=7](B\x) edge node[right] {} (A\x);
          \path [bend right=7](B\x) edge node[right] {} (A\x);
          \path [bend left=7](BB\x) edge node[right] {} (AA\x);
          \path [bend right=7](BB\x) edge node[right] {} (AA\x);
          \path [bend left=7](A\x) edge node[right] {} (AA\x);
          \path [bend right=7](A\x) edge node[right] {} (AA\x);
        }
        \draw[] (1,.-1.5)node[left] {$\Gr F_2(\G,V_0)$};
      \end{tikzpicture}\quad
      \begin{tikzpicture}[auto,vertex/.style={circle,draw=black!100,%
          fill=black!100, 
          inner sep=0pt,minimum size=1mm},scale=1]
        \foreach \x in {0}
        { \node(C\x) at (\x*0.5,1) [vertex,fill=black,label=left:] {};
          \node (B\x) at (-0.5,0) [vertex,fill=white,label=left:] {};
          \node (BB\x) at (0.5,0) [vertex,fill=white,label=left:] {};
          \path [-](C\x) edge node[right] {} (B\x);
          \path [bend left=10](C\x) edge node[right] {} (B\x);
          \path [bend right=10](C\x) edge node[right] {} (B\x);
          \path [-](C\x) edge node[right] {} (BB\x);
          \path [bend left=10](C\x) edge node[right] {} (BB\x);
          \path [bend right=10](C\x) edge node[right] {} (BB\x);
          \node (A\x) at (-0.5,-1) [vertex,fill=white,label=left:] {};
          \node (AA\x) at (0.5,-1) [vertex,fill=white,label=left:] {};
          \path [-](B\x) edge node[right] {} (BB\x);
          \path [bend left=10](B\x) edge node[right] {} (BB\x);
          \path [bend right=10](B\x) edge node[right] {} (BB\x);
          \path [-](B\x) edge node[right] {} (A\x);
          \path [bend left=10](B\x) edge node[right] {} (A\x);
          \path [bend right=10](B\x) edge node[right] {} (A\x);
          \path [-](BB\x) edge node[right] {} (AA\x);
          \path [bend left=10](BB\x) edge node[right] {} (AA\x);
          \path [bend right=10](BB\x) edge node[right] {} (AA\x);
          \path [-](A\x) edge node[right] {} (AA\x);
          \path [bend left=10](A\x) edge node[right] {} (AA\x);
          \path [bend right=10](A\x) edge node[right] {} (AA\x);
        }
        \draw[] (1,.-1.5)node[left]  {$\Gr F_3(\G,V_0)$};
      \end{tikzpicture}\quad
      \begin{tikzpicture}[auto,vertex/.style={circle,draw=black!100,%
          fill=black!100, 
          inner sep=0pt,minimum size=1mm},scale=1]
          \node (B) at (-0.5,0) [vertex,fill=white,label=left:] {};
          \node (BB) at (0.5,0) [vertex,fill=white,label=left:] {};
          \node (A) at (-0.5,-1) [vertex,fill=white,label=left:] {};
          \node (AA) at (0.5,-1) [vertex,fill=,label=left:] {};
          \path [bend left=5](B) edge node[right] {} (BB);
          \path [bend right=5](B) edge node[right] {} (BB);
          \path [bend left=5](B) edge node[right] {} (A);
          \path [bend right=5](B) edge node[right] {} (A);
          \path [bend left=5](BB) edge node[right] {} (AA);
          \path [bend right=5](BB) edge node[right] {} (AA);
          \path [bend left=5](A) edge node[right] {} (AA);
          \path [bend right=5](A) edge node[right] {} (AA);
          \path [bend left=14](B) edge node[right] {} (BB);
          \path [bend right=14](B) edge node[right] {} (BB);
          \path [bend left=14](B) edge node[right] {} (A);
          \path [bend right=14](B) edge node[right] {} (A);
          \path [bend left=14](BB) edge node[right] {} (AA);
          \path [bend right=14](BB) edge node[right] {} (AA);
          \path [bend left=14](A) edge node[right] {} (AA);
          \path [bend right=14](A) edge node[right] {} (AA);
          \node(C1) at (-0.5,1) [vertex,fill=black,label=left:] {};
          \node(C2) at (0.5,1) [vertex,fill=black,label=left:] {};
          \path [-](C1) edge node[right] {} (B);
          \path [-](C1) edge node[right] {} (BB);
          \path [bend right=10](C2) edge node[right] {} (B);
          \path [bend left=10](C2) edge node[right] {} (B);
          \path [bend right=10](C2) edge node[right] {} (BB);
          \path [bend left=10](C2) edge node[right] {} (BB);
          \path [-](C2) edge node[right] {} (B);
          \path [-](C2) edge node[right] {} (BB);
        \draw[] (1,.-1.5)node[left]  {$\Gr F_{A,V_1}(\G,V_0)$};
      \end{tikzpicture}\quad
      \begin{tikzpicture}[auto,vertex/.style={circle,draw=black!100,%
          fill=black!100, 
          inner sep=0pt,minimum size=1mm},scale=1]
          \node (B) at (-0.5,0) [vertex,fill=white,label=left:] {};
          \node (BB) at (0.5,0) [vertex,fill=white,label=left:] {};
          \node (A) at (-0.5,-1) [vertex,fill=white,label=left:] {};
          \node (AA) at (0.5,-1) [vertex,fill=,label=left:] {};
          \path [bend left=5](B) edge node[right] {} (BB);
          \path [bend right=5](B) edge node[right] {} (BB);
          \path [bend left=5](B) edge node[right] {} (A);
          \path [bend right=5](B) edge node[right] {} (A);
          \path [bend left=5](BB) edge node[right] {} (AA);
          \path [bend right=5](BB) edge node[right] {} (AA);
          \path [bend left=5](A) edge node[right] {} (AA);
          \path [bend right=5](A) edge node[right] {} (AA);
          \path [bend left=14](B) edge node[right] {} (BB);
          \path [bend right=14](B) edge node[right] {} (BB);
          \path [bend left=14](B) edge node[right] {} (A);
          \path [bend right=14](B) edge node[right] {} (A);
          \path [bend left=14](BB) edge node[right] {} (AA);
          \path [bend right=14](BB) edge node[right] {} (AA);
          \path [bend left=14](A) edge node[right] {} (AA);
          \path [bend right=14](A) edge node[right] {} (AA);
          \node(C1) at (-0.5,1) [vertex,fill=black,label=left:] {};
          \node(C2) at (0.5,1) [vertex,fill=black,label=left:] {};
          \path [bend right=7](C1) edge node[right] {} (B);
          \path [bend right=7](C1) edge node[right] {} (BB);
          \path [bend right=7](C2) edge node[right] {} (B);
          \path [bend right=7](C2) edge node[right] {} (BB);
          \path [bend left=7](C1) edge node[right] {} (B);
          \path [bend left=7](C1) edge node[right] {} (BB);
          \path [bend left=7](C2) edge node[right] {} (B);
          \path [bend left=7](C2) edge node[right] {} (BB);
        \draw[] (1,.-1.5)node[left]  {$\Gr F_{B,V_1}(\G,V_0)$};
      \end{tikzpicture}\quad
      \begin{tikzpicture}[auto,vertex/.style={circle,draw=black!100,%
          fill=black!100, 
          inner sep=0pt,minimum size=1mm},scale=1]
          \node (B) at (-0.5,0) [vertex,fill=white,label=left:] {};
          \node (BB) at (0.5,0) [vertex,fill=white,label=left:] {};
          \node (A) at (-0.5,-1) [vertex,fill=white,label=left:] {};
          \node (AA) at (0.5,-1) [vertex,fill=white,label=left:] {};
          \node(C1) at (-0.5,1) [vertex,fill=black,label=left:] {};
          \node(C2) at (0.5,1) [vertex,fill=black,label=left:] {};
          \path [-](B) edge node[right] {} (BB);
          \path [-](B) edge node[right] {} (A);
          \path [-](BB) edge node[right] {} (AA);
          \path [-](A) edge node[right] {} (AA);
          \path [-](C1) edge node[right] {} (B);
          \path [-](C1) edge node[right] {} (BB);
          \path [-](C2) edge node[right] {} (B);
          \path [-](C2) edge node[right] {} (BB);
          \draw[] (-0.40,0.6)node[left]  {$\frac14$};
          \draw[] (0.40,0.6)node[right]  {$\frac34$};
          \draw[] (0.1,1)node[left]  {$\frac14$};
          \draw[] (-0.1,1)node[right]  {$\frac34$};
        \draw[] (0.8,.-1.5)node[left]  {$\wG_{A,v^*}$};
      \end{tikzpicture}\quad
      \begin{tikzpicture}[auto,vertex/.style={circle,draw=black!100,%
          fill=black!100, 
          inner sep=0pt,minimum size=1mm},scale=1]
          \node (B) at (-0.5,0) [vertex,fill=white,label=left:] {};
          \node (BB) at (0.5,0) [vertex,fill=white,label=left:] {};
          \node (A) at (-0.5,-1) [vertex,fill=white,label=left:] {};
          \node (AA) at (0.5,-1) [vertex,fill=white,label=left:] {};
          \node(C1) at (-0.5,1) [vertex,fill=black,label=left:] {};
          \node(C2) at (0.5,1) [vertex,fill=black,label=left:] {};
          \path [-](B) edge node[right] {} (BB);
          \path [-](B) edge node[right] {} (A);
          \path [-](BB) edge node[right] {} (AA);
          \path [-](A) edge node[right] {} (AA);
          \path [-](C1) edge node[right] {} (B);
          \path [-](C1) edge node[right] {} (BB);
          \path [-](C2) edge node[right] {} (B);
          \path [-](C2) edge node[right] {} (BB);
          \draw[] (-0.40,0.6)node[left]  {$\frac12$};
          \draw[] (0.40,0.6)node[right]  {$\frac12$};
          \draw[] (0.1,1)node[left]  {$\frac12$};
          \draw[] (-0.1,1)node[right]  {$\frac12$};
        \draw[] (0.8,.-1.5)node[left]  {$\wG_{B,v^*}$};
      \end{tikzpicture}
    }
    \caption{All vertices in a frame are identified ($V_0=V(\G)$), and for
      $V_1=V_0 \setminus \{v^*\}$.  On the right, the pair of isospectral
      weighted graphs with edge weights indicated (if different from $1$).
    \label{fig:all.but.1.identified}
  }
  \end{figure}
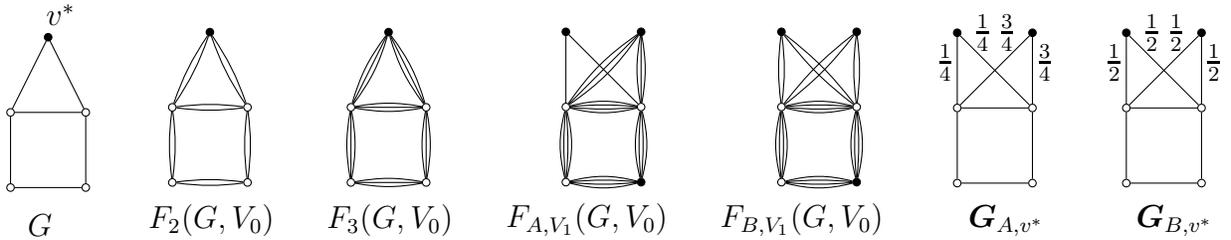
\end{example}
%
\section{Conclusions}
\label{sec:conclusions}
%

Sunada's method~\cite{sunada:85} as well as some of its
generalisations~\cite{bpbs:09,parzanchevski-band:10,bbjl:pre17} use
representation theory as a key ingredient to construct the isospectral
manifolds or graphs which appear as quotient graphs. It is worth to point
out differences and similarities between the method presented here and
the one given, for example, in~\cite{bpbs:09} even if we think that
both methods are different in nature. We do not exclude though that
both approaches may be combined.

In~\cite{bpbs:09} the metric graph considered first carries the action
of a finite (in general non-commutative) group.  Two different
subgroups are chosen which will determine the isospectral quotient
graphs provided the corresponding induced representations are
equivalent. Neumann and Dirichlet conditions on the metric graph
appear naturally.

The method presented in this article has a more combinatorial nature
even if the underlying symmetry is, of course, expressed in terms of
cyclic (commutative) groups acting on the copies of the building block
as mentioned in \Rem{dualgroup}.  But the construction of the quotient
here goes through \emph{for any} choice of the pair of subsets $V_0$,
$V_1$, and we see no way to implement this freedom through
subgroups of $\Z_a$.  Another instance is that the action of the cyclic
groups is not free on the quotient, as the vertices in $V_0$ are fixed
points.  An additional difference is the importance of the eigenfunctions
with Dirichlet conditions on the vertices $V_0$ and $V_1$ and the
specific form of the partition that appears explicitly in the degree
lists of the different quotients but not in the corresponding spectra.
Nevertheless, Brooks result on isospectral (regular) graphs not
arising from Sunada's method (cf.~\cite{brooks:99}) as well as a
theory of groups acting not freely on graphs might give a hint how to
understand our examples in the spirit of representation theory.

Last but not least the approach using Dirichlet-to-Neumann maps of
subgraphs of quantum graphs in~\cite{kurasov-muller:pre21} produces
some of our examples of isospectral graphs using different methods; and
it may also apply to discrete graphs.  Moreover, as the authors also
write in~\cite{kurasov-muller:pre21} it is not obvious how their
approach can be linked with representation theory.

%
%
\appendix

\section{Multisets and partitions}
\label{app:appendix}

\subsection{Multisets}
\label{app:multisets}
Multisets will be a convenient tool to deal with spectra and degree
lists of graphs.
\begin{itemize}
\item A \emph{multiset} is an ordered pair $(A,m)$ where $A$ is a set
  and $\map m A {\N_0=\{0,1,2,\dots\}}$ is a map on $A$.  We say that
  $x \in A$ appears \emph{$m(x)$-times} in $A$.  In order to simplify
  notation, we set $m_A(x)=0$ whenever $x \notin A$.

\item We often simply refer to $A$ as a multiset without mentioning
  the map $m$, and hence refer to the multiplicity map of $A$ as
  $m_A$.  Sets $A$ can be considered as multisets by setting
  $m_A(x)=1$ whenever $x \in A$ and $0$ otherwise.
\item The \emph{cardinality} of the multiset $A$ is the sum of the
  multiplicities of all its elements, i.e.,
  $\card A =\sum_{x\in A} m_A(x)$.
\item When $A$ is finite (i.e., $\card A < \infty$) we write
  $A=\{a_1,a_2,\dots,a_s\}$, and a multiset $(A,m)$ will then also be
  written as
  \begin{align*}
    A&=\lMset a_1^{(m(a_1))},a_2^{(m(a_2))},\dots,a_s^{(m(a_s))} \rMset
     \;=\;\lMset \underbrace{a_1,\dots,a_1}_{m(a_1)},
       \underbrace{a_2,\dots,a_2}_{m(a_2)},\dots,
       \underbrace{a_s,\dots,a_s}_{m(a_s)}\rMset.
  \end{align*}
  Multisets with elements in $\N$ or $\R$ (or any other linearly
  ordered set) can be also seen as ordered lists and sometimes it will
  be convenient for the exposition to take this point of view.

\item The \emph{sum} of two multisets $A$ and $B$ denoted as
  $ A \uplus B$ is the set $A \cup B$ with multiplicity map
  $m_{A_\uplus B}(x):=m_A(x)+m_B(x)$ for all $x \in A \cup B$.

\item The \emph{difference} of two multisets, denoted by $A \ominus B$
  is the set $A \setminus B$ with multiplicity map
  $m_{A \ominus B}(x) := \max\{m_A(x)-m_B(x),0\}$.

\item The \emph{$k$-th multiple} of a multiset $A$ (with multiplicity
  map $m_A$), denoted by $A^{(k)}$, is the multiset $A$ with
  multiplicity map $m_{A^{(k)}}(x)=km_A(x)$ for all $x\in A$, i.e.,
  the multiplicity of each element in $A^{(k)}$ is multiplied by $k$.
\end{itemize}

\subsection{Partitions of a natural number}
\label{app:partitions}

The notion of a partition of a natural number will be important for
our construction of isospectral graphs.
\begin{definition}[partition of a number]
  \label{def:partition}
  A partition of a natural number $r$ of length $s$ (or an $s$-\emph{partition} $r$ for short) is a multiset $A$ of natural numbers with sum $r$, i.e.,
  \begin{equation*}
    A= \lMset a_1,a_2,\dots,a_s \rMset
    \quadtext{such that}
    \sum_{i= 1}^s a_i =r.
  \end{equation*}
  The number $s$ is called the \emph{length of the partition}.
\end{definition}
We frequently use the following equations
\begin{equation}
  \label{eq:part.number}
  \sum_{a \in A} 1=s
  \qquadtext{and}
  \sum_{a \in A} a=r,
\end{equation}
where $A$ is understood as a multiset.

%
%

\newcommand{\etalchar}[1]{$^{#1}$}
  \def\cprime{$'$}
\providecommand{\MRhref}[2]{%
  \href{http://www.ams.org/mathscinet-getitem?mr=#1}{#2}
}
\providecommand{\href}[2]{#2}

\end{document}